\documentclass[11pt,reqno,a4paper]{amsart}

\usepackage[T1]{fontenc}
\usepackage[utf8]{inputenc}

\usepackage{libertinus}
\usepackage{microtype}

\usepackage{geometry}

\usepackage{xcolor}

\usepackage{amsmath,amssymb,amsthm,amsfonts}
\usepackage{mathtools}
\usepackage{mathrsfs}

\usepackage[
  colorlinks=true,
  linkcolor=blue,
  citecolor=blue,
  urlcolor=blue,
  pagebackref=true
]{hyperref}

\renewcommand*{\backrefalt}[4]{%
  \ifcase #1\relax
  \or
    {\footnotesize Cited on page #2.}%
  \else
    {\footnotesize Cited on pages #2.}%
  \fi
}

\usepackage{graphicx}

\usepackage{titlesec}

\titleformat{\section}
  {\normalfont\large\bfseries\centering}
  {\thesection.}
  {0.75em}
  {}

\titlespacing*{\section}
  {0pt}
  {3.2ex plus 1ex minus .2ex}
  {1.6ex plus .3ex}

\titleformat{\subsection}
  {\normalfont\normalsize\bfseries}
  {\thesubsection.}
  {0.75em}
  {}

\titlespacing*{\subsection}
  {0pt}
  {2.4ex plus .8ex minus .2ex}
  {1.0ex plus .2ex}

\titleformat{\subsubsection}
  {\normalfont\normalsize\bfseries\itshape}
  {\thesubsubsection.}
  {0.75em}
  {}

\titlespacing*{\subsubsection}
  {0pt}
  {2.0ex plus .6ex minus .2ex}
  {0.7ex plus .2ex}

\usepackage{etoolbox}

\makeatletter
\AtBeginEnvironment{thebibliography}{%
  \def\@mklab#1{#1\hfil}%
}
\makeatother
  
\numberwithin{equation}{section}

\usepackage{aliascnt}

\newtheoremstyle{compact}
  {6pt}   
  {1 pt}   
  {\itshape}
  {}
  {\bfseries}
  {.}
  {0.5em}
  {}

\theoremstyle{compact}

\newtheorem{theorem}{Theorem}[section]

\newaliascnt{proposition}{theorem}
\newtheorem{proposition}[proposition]{Proposition}
\aliascntresetthe{proposition}

\newaliascnt{lemma}{theorem}
\newtheorem{lemma}[lemma]{Lemma}
\aliascntresetthe{lemma}

\newaliascnt{corollary}{theorem}
\newtheorem{corollary}[corollary]{Corollary}
\aliascntresetthe{corollary}

\theoremstyle{definition}
\newaliascnt{definition}{theorem}

\aliascntresetthe{definition}

\theoremstyle{remark}

\newtheorem*{unnumberedremark}{Remark}

\AtBeginEnvironment{proof}{\vspace{6pt}}

\usepackage{enumitem}

\setlist{
itemsep=2pt,
topsep=4pt
}

\newcommand{\bS}{\mathbb{S}}

\newcommand\subsetsim{\mathrel{%
\ooalign{\raise0.2ex\hbox{$\subset$}\cr\hidewidth\raise-0.8ex\hbox{\scalebox{0.9}{$\sim$}}\hidewidth\cr}}}

\usepackage{comment}

\makeatletter
\renewcommand{\l@section}[2]{%
  \par\addpenalty\@secpenalty
  \addvspace{1.0em plus 1pt}%
  \@tempdima 1.5em
  \begingroup
    \parindent \z@
    \rightskip \@pnumwidth
    \parfillskip -\@pnumwidth
    \leavevmode
    \bfseries
    \advance\leftskip\@tempdima
    \hskip -\leftskip
    #1\nobreak\hfil \nobreak\hb@xt@\@pnumwidth{\hss #2}\par
  \endgroup}
\makeatother

\usepackage{bbm}
\usepackage{amscd}
\usepackage[all,cmtip]{xy}
\usepackage{changepage}
\usepackage{calc}
\usepackage{scalerel}
\usepackage{extarrows}

\DeclareMathSymbol{\shortminus}{\mathbin}{AMSa}{"39}

\newcommand{\ignore}[1]{}

\begin{document}

\title{Hyperuniform Delone Realizations and Rigidity}

\author{Michael Bj\"orklund}
\address{Department of Mathematics, Chalmers University of Technology and University of Gothenburg, Gothenburg, Sweden}
\email{micbjo@chalmers.se}



\begin{abstract}
We prove a measurable realization theorem for hyperuniform Delone point
processes. In dimensions \(d\geq2\), for every prescribed \(q\geq1\),
every essentially free ergodic p.m.p.\ action of \(\mathbb R^d\) admits,
at every sufficiently large prescribed intensity, a generating Delone
realization whose return-time point process \(\eta\) is measurably
isomorphic to the original action and whose Bartlett spectrum
\(\sigma_\eta\) satisfies
\[
    \sigma_\eta(B_\varepsilon)=o(\varepsilon^{2q})
    \qquad(\varepsilon\downarrow0).
\]
Thus arbitrarily high finite-order
low-frequency suppression can be imposed without changing the prescribed
measurable dynamics. The same realizations can be chosen with surface-order
ball variance and linear rigidity to any prescribed finite order, while also
being maximally rigid and almost surely bounded-displacement equivalent to a
lattice. For essentially free Euclidean-motion actions whose translation
subaction is ergodic, the construction can be made isotropic and \(V\)-ergodic,
and hence \(V\)-weakly mixing. In dimension one, every essentially free
ergodic flow admits generating Delone realizations with logarithmic interval
discrepancy, maximal rigidity, and near-quadratic decay of the Bartlett
spectrum at the origin.
\end{abstract}

\maketitle


\section{Introduction}
\label{sec:introduction}

Let \(\eta\) be a translation-invariant point process on \(\mathbb R^d\), and
let
\[
    N_R(P):=\#(P\cap B_R)
\]
denote the number of points of a configuration \(P\) in the ball \(B_R\) of
radius \(R\). The process is called \emph{hyperuniform} if
\[
    \operatorname{Var}_\eta(N_R)=o(R^d)
    \qquad(R\to\infty).
\]
Important examples include lattices, certain cut-and-project sets as
established by Bj\"orklund and Hartnick \cite{BH24}, and determinantal point
processes associated with projection kernels \cite{Sos00}; see also
\cite{Tor18,TS03}. Recent work has developed general procedures for producing
hyperuniform processes from other stationary random measures by transport and
related spatial transformations \cite{KLLY25,LK26}.

In this paper the starting point is not a spatial model but a prescribed
measure-preserving action. Given an ergodic p.m.p.\ action of \(\mathbb R^d\),
we ask whether the action can be represented by a hyperuniform Delone point
process without changing its measurable dynamics. More precisely, we require
the point-process representation to be \emph{generating}: the resulting point
configuration must determine the original state almost surely. Thus the point
process is required to be measurably isomorphic to the given action, not merely
a factor of it.

In dimensions \(d\geq2\), the answer is very flexible. Every essentially free
ergodic \(\mathbb R^d\)-action admits, at every sufficiently large prescribed
intensity, generating Delone realizations whose Bartlett spectrum vanishes at
the origin to any prescribed finite polynomial order. These realizations can
simultaneously be chosen maximally rigid, linearly rigid to any prescribed
finite order, and almost surely bounded-displacement equivalent to a lattice.
The construction is compatible with rotational symmetry as well: for suitable
Euclidean-motion actions with ergodic translation subaction, the resulting
process can be chosen isotropic and \(V\)-ergodic, hence \(V\)-weakly mixing,
with surface-order ball variance.

The situation in dimension one is different. Every essentially free ergodic
flow still admits generating Delone realizations with maximal rigidity and
logarithmic interval discrepancy. On the other hand, the companion paper
\cite{Bjo26} shows that, for an ergodic point process of positive intensity
and local second moments, the condition
\[
    \int_{0<|\xi|<1}\frac{1}{\xi^2}\,d\sigma_\eta(\xi)<\infty
\]
forces a nonzero Koopman eigenvalue of the translation action, and in fact
lattice induction. In particular, a point process with this degree of
low-frequency Bartlett suppression cannot be translation weakly mixing.
Uniformly bounded interval discrepancy already imposes a nontrivial spectral
restriction in the present setting. Thus sufficiently strong suppression near
the origin constrains the measurable dynamics in dimension one.

By contrast, in dimensions \(d\geq2\), finite-order Bartlett suppression,
surface-order variance, maximal rigidity, finite-order linear rigidity, and
bounded-displacement lattice geometry do not by themselves constrain the
underlying measurable dynamics: they can be imposed on generating
point-process coordinates for an arbitrary essentially free ergodic action.
The local moment-cancellation principle used to obtain the spectral estimates
has important precedents in tiling, quadrature, and transport constructions
\cite{GJT08,KLLY25,LK26}. The additional requirement here is that these
spatial properties coexist with a Borel equivariant encoding from which the
prescribed dynamical state can be recovered.

This paper is a companion to \cite{Bjo26}, where generating stealthy
Delone realizations are constructed for lattice-induced actions. In dimension
one, the obstruction above shows that sufficiently strong low-frequency
suppression forces nontrivial spectral structure in the translation action.
The present paper shows that no analogous restriction occurs for suppression
of any fixed finite order in dimensions \(d\geq2\). Whether genuinely
stronger low-frequency conditions impose dynamical constraints in higher
dimensions remains open.


\subsection{Hyperuniformity and point-process coordinates}
\label{subsec:introduction-coordinates}

Let \(V\) be a real inner-product space of dimension \(d\), let
\(K\leq O(V)\) be closed, and put \(G_K=V\rtimes K\). We identify \(V\)
with the space \(K\backslash G_K\) of left \(K\)-cosets by
\[
    K(v,k)\longmapsto k^{-1}v.
\]
Under this identification, the action
\[
    g.(Kh):=Khg^{-1}
\]
of \(G_K\) on \(K\backslash G_K\) becomes
\[
    (v,k).u=ku-v.
\]
Accordingly, \(G_K\) acts on configurations by
\[
    (v,k).\delta_P=\delta_{kP-v}.
\]
Let \(\mathcal N_s(V)\) denote the space of locally finite simple
counting measures on \(V\), and let \(\eta\) be a translation-invariant
point process on \(\mathcal N_s(V)\) with positive intensity and local
second moments. For
\(f\in\mathcal S(V)\), write
\[
    \bS f(\omega):=\int_V f\,d\omega,
    \qquad
    \widehat f(\xi)
    :=
    \int_V f(u)e^{-2\pi i\langle u,\xi\rangle}\,d\lambda_V(u).
\]
Its \emph{Bartlett spectrum} is the unique positive translation-bounded Radon
measure \(\sigma_\eta\) satisfying
\[
    \operatorname{Var}_\eta(\bS f)
    =
    \int_V |\widehat f(\xi)|^2\,d\sigma_\eta(\xi).
\]
For a bounded Borel set \(A\subset V\), write
\[
    N_A(\omega):=\omega(A).
\]
Writing \(B_R\) for the ball of radius \(R\) about the origin, the
process is \emph{hyperuniform} if
\[
    \operatorname{Var}_\eta(N_{B_R})=o(R^d)
    \qquad(R\to\infty),
\]
equivalently,
\[
    \sigma_\eta(B_\varepsilon)=o(\varepsilon^d)
    \qquad(\varepsilon\downarrow0);
\]
see \cite[Proposition~3.3]{BH24}. It is \emph{stealthy} if
\(\sigma_\eta\) vanishes on a neighborhood of the origin. For an integer
\(p\geq1\), when we say that \(\eta\) has \emph{Bartlett suppression of
order \(p\)} we mean
\[
    \sigma_\eta(B_\varepsilon)=o(\varepsilon^{2p})
    \qquad(\varepsilon\downarrow0).
\]
We use the more precise phrase \emph{surface-order number variance} for
\[
    \operatorname{Var}_\eta(N_{B_R})=O(R^{d-1})
    \qquad(R\to\infty).
\]
The scale \(O(R^{d-1})\) corresponds to class-I hyperuniformity in the
usual classification \cite{Tor18}; we use the explicit estimate because no
absolute continuity of the Bartlett spectrum is assumed.

Let \(G_K\curvearrowright(X,\mu)\) be a p.m.p.\ Borel action. As usual
in the measure-preserving setting, all cross-section assertions are
understood modulo invariant null sets: a Borel set \(Y\subset X\) is a
\emph{translation cross-section} if, on some \(G_K\)-invariant conull Borel
set \(X_0\subset X\), every \(V\)-orbit meets \(Y\cap X_0\) and
\[
    Y_x:=\{v\in V:(v,e).x\in Y\}
\]
is locally finite for every \(x\in X_0\). Replacing \(Y\) by
\(Y\cap X_0\), we may assume \(Y\subset X_0\). Its return-time map and
return-time process are
\[
    \kappa_Y(x):=\delta_{Y_x},
    \qquad
    \eta_Y:=(\kappa_Y)_*\mu.
\]
The \emph{intensity} of \(Y\) is, by definition, the intensity of
\(\eta_Y\). We call \(Y\) \emph{Delone} if the sets \(Y_x\), on an invariant
conull set as above, are uniformly separated and uniformly relatively dense.
It is \emph{generating} if
\(\kappa_Y\) is injective on a \(G_K\)-invariant conull Borel subset of
\(X\). In that case \(\kappa_Y\) is a measurable isomorphism from such
a conull subset of \(X\) onto its image, which has full \(\eta_Y\)-measure;
if \(Y\) is also \(K\)-invariant, this isomorphism is \(G_K\)-equivariant.
Generation is the information-preserving requirement in the constructions
below.

If \(Y\) is invariant under \(K\), then
\[
    Y_{(v,k).x}=kY_x-v,
    \qquad (v,k)\in G_K,\quad x\in X,
\]
so \(\kappa_Y\) is \(G_K\)-equivariant and \(\eta_Y\) is
\(G_K\)-invariant. For the higher-dimensional results, the restricted \(V\)-action is
assumed ergodic. This assumption is stronger than \(G_K\)-ergodicity and
matters in the isotropic case. For example, rotationally averaging an
anisotropic process may produce an isotropic \(G_K\)-ergodic law while
leaving a global orientation fixed by translations. Translation ergodicity
rules out such a parameter. When \(d\geq2\) and \(K=SO(V)\) or \(O(V)\),
translation ergodicity also implies translation weak mixing by
Corollary~\ref{cor:isotropic-ergodic-weak-mixing}. The same hypothesis later
removes the zero-frequency atoms of the auxiliary covariance spectra.


\subsection{Main results}
\label{subsec:introduction-main-results}

We first state the two realization theorems. We call a point process
\emph{maximally rigid} if, for every bounded Borel set \(A\subset V\), the
restriction of a realization to \(A^c\) measurably determines the entire
realization almost surely.

\begin{theorem}[One-dimensional realization]
\label{thm:main-dimension-one}
Let
\(\mathbb R\curvearrowright(X,\mu)\) be an essentially free ergodic
p.m.p.\ Borel action. There exists \(\rho_0<\infty\) such that,
for every \(\rho\geq\rho_0\), there is a generating Delone
translation cross-section \(Y\subset X\) of intensity \(\rho\)
whose return-time process is maximally rigid.

Moreover, there are a constant \(C_\rho<\infty\) and an invariant conull
Borel set \(X_0\subset X\) such that
\[
    \left|
        \#(Y_x\cap I)-\rho|I|
    \right|
    \leq
    C_\rho\log(2+|I|)
\]
for every \(x\in X_0\) and every bounded interval
\(I\subset\mathbb R\).
\end{theorem}

\begin{theorem}[Higher-dimensional realization]
\label{thm:main-higher-dimensional}
Let \(d\geq2\), let \(K\leq O(V)\) be closed, and let
\(G_K=V\rtimes K\curvearrowright(X,\mu)\) be an essentially free
p.m.p.\ Borel action whose restricted \(V\)-action is ergodic.
For every integer \(p\geq1\) there exists \(\rho_0<\infty\) such
that, for every \(\rho\geq\rho_0\), there is a \(K\)-invariant
generating Delone translation cross-section \(Y\subset X\) of intensity
\(\rho\) whose return-time process is maximally rigid and satisfies
\[
    \sigma_{\eta_Y}(B_\varepsilon)
    =
    o(\varepsilon^{2p})
    \qquad(\varepsilon\downarrow0).
\]
If \(2p>d+1\), then moreover
\[
    \operatorname{Var}_{\eta_Y}(N_{B_R})
    =O(R^{d-1})
    \qquad(R\to\infty).
\]
\end{theorem}

We next state the corresponding conclusions entirely in point-process
language. A point process carries its canonical translation action on
configuration space. Applying a realization theorem to this action replaces
the original process by an equivariantly isomorphic one; the resulting map on
configurations is made explicit below.

\begin{corollary}[One-dimensional point-process realization]
\label{cor:main-dimension-one-point-process}
Let \(\eta_0\) be a translation-ergodic point process on \(\mathbb R\) whose
canonical translation action on \((\mathcal N_s(\mathbb R),\eta_0)\) is
essentially free. For every sufficiently large \(\rho\), there are a
translation-ergodic point process \(\eta\), supported on Delone configurations
and of intensity \(\rho\), and a translation-equivariant measurable
isomorphism
\[
    T:(\mathcal N_s(\mathbb R),\eta_0)
      \longrightarrow
      (\mathcal N_s(\mathbb R),\eta).
\]
The process \(\eta\) is maximally rigid and satisfies
\[
    \operatorname{Var}_\eta(N_{[0,R)})=O(\log^2 R)
    \qquad(R\to\infty)
\]
and
\[
    \sigma_\eta([-\varepsilon,\varepsilon])
    =
    O\!\left(
        \varepsilon^2
        \log^2\!\frac{e}{\varepsilon}
    \right)
    \qquad(\varepsilon\downarrow0).
\]
In particular, \(\eta\) is hyperuniform.
\end{corollary}

The logarithmic discrepancy in Theorem~\ref{thm:main-dimension-one}
cannot in general be replaced by a uniform bound.
Proposition~\ref{prop:bounded-discrepancy-eigenvalue} shows that uniformly
bounded interval discrepancy at intensity \(\rho\) forces \(\rho\) to be a
nonzero eigenvalue of the Koopman representation of the flow. Thus uniformly
bounded discrepancy imposes a genuine spectral restriction on the underlying
dynamics; in particular, it cannot occur for a weakly mixing flow. The
argument does not establish optimality of the logarithmic bound, and
intermediate sublogarithmic rates remain open.

The higher-dimensional theorem gives the following recoding result. The
relevant notions of rigidity are defined in
Subsection~\ref{subsec:rigidity-definitions}.

\begin{corollary}[Point-process realization]
\label{cor:main-point-process-realization}
Let \(d\geq2\), let \(K\leq O(V)\) be closed, and let \(\eta_0\) be a
\(G_K\)-invariant point process on \(V\). Suppose that its canonical
\(G_K\)-action on \((\mathcal N_s(V),\eta_0)\) is essentially free and that
the restricted \(V\)-action is ergodic. Given integers \(q\geq1\) and
\(\ell\geq0\), every sufficiently large prescribed intensity \(\rho\) admits a
\(G_K\)-invariant, \(V\)-ergodic point process \(\eta\), supported on Delone
configurations and of intensity \(\rho\), together with a \(G_K\)-equivariant
measurable isomorphism
\[
    T:(\mathcal N_s(V),\eta_0)
      \longrightarrow
      (\mathcal N_s(V),\eta).
\]
Moreover,
\[
    \sigma_\eta(B_\varepsilon)
    =
    o(\varepsilon^{2q})
    \qquad(\varepsilon\downarrow0),
\]
\[
    \operatorname{Var}_\eta(N_{B_R})=O(R^{d-1})
    \qquad(R\to\infty),
\]
and \(\eta\) is maximally rigid and linearly \(\ell\)-rigid.
\end{corollary}

\begin{proof}
Apply Theorem~\ref{thm:main-higher-dimensional} to the canonical
\(G_K\)-action on \((\mathcal N_s(V),\eta_0)\), with an integer \(p\) satisfying
\[
    p\geq q,
    \qquad
    p>d+\ell.
\]
Let \(Y\) be the resulting generating translation cross-section and put
\(T:=\kappa_Y\) and \(\eta:=T_*\eta_0\). Since \(Y\) is generating,
\(T\) is a measurable \(G_K\)-isomorphism onto its image. The theorem gives
maximal rigidity and
\[
    \sigma_\eta(B_\varepsilon)=o(\varepsilon^{2p})
    =o(\varepsilon^{2q}).
\]
Since \(p>d+\ell\) and \(d\geq2\), one has \(2p>d+1\), so the theorem also
gives surface-order ball variance. Finally,
Proposition~\ref{prop:polynomial-rigidity} gives linear \(\ell\)-rigidity.
\end{proof}

The map \(T\) is a recoding of the whole configuration, not a pointwise
matching between the points of the two configurations. In the terminology of
stationary random measures it is a factor map; this is distinct from an
allocation rule, which acts on physical space and transports one random
measure to another; see \cite{KM25}. A pointwise ``perturbation'' description
would require an additional equivariant matching, which is not used here.

The surface-order variance estimate is optimal for the processes constructed
here. The averaged Beck-type lower bound of the author and Byl\'ehn
\cite[Theorem~5.1]{BB26}, applied to the nonzero Bartlett spectrum of the
constructed process and the absence of an atom at the origin, gives
\[
    0<\limsup_{R\to\infty}
    \frac{\operatorname{Var}_\eta(N_{B_R})}{R^{d-1}}<\infty.
\]
Proposition~\ref{prop:surface-order-optimality} verifies the hypotheses and
carries out the passage from the averaged lower bound to this limsup.
Consequently, the ball-variance bound cannot be improved to
\(o(R^{d-1})\), despite arbitrarily high finite-order decay of the Bartlett
spectrum. Here and below,
\(\mathcal H^{d-1}\) denotes the \((d-1)\)-dimensional Hausdorff measure.

Recall that two Delone sets are \emph{bounded-displacement equivalent} if
there is a bijection between them that moves every point by a uniformly
bounded distance.

\begin{corollary}[Bounded-displacement realizations]
\label{cor:main-bounded-displacement}
Under the hypotheses of Corollary~\ref{cor:main-point-process-realization},
the process \(\eta\) may be chosen so that, after fixing an orthonormal
identification \(V\cong\mathbb R^d\), there is \(C_\rho<\infty\) with
\[
    \left|\#(P\cap U)-\rho\lambda_V(U)\right|
    \leq
    C_\rho\,\mathcal H^{d-1}(\partial U)
\]
for \(\eta\)-almost every \(P\) and every finite union \(U\) of unit cubes in
the corresponding integer grid.
\end{corollary}

\begin{unnumberedremark}
By Laczkovich's criterion \cite{Lac92}, in the form recorded in
\cite[Theorem~1.4]{Sol14}, the estimate in
Corollary~\ref{cor:main-bounded-displacement} implies that \(P\) is
bounded-displacement equivalent, for \(\eta\)-almost every \(P\), to a lattice
of covolume \(\rho^{-1}\), and hence bi-Lipschitz equivalent to that lattice.
The bounded-displacement bijection is not asserted to be measurable or
translation equivariant; this pointwise equivalence therefore does not
identify the measurable translation dynamics with those of a lattice.
\end{unnumberedremark}

Taking \(K=SO(V)\) gives a point-process version of the isotropic realization
result.

\begin{corollary}[Isotropic realizations]
\label{cor:main-isotropic-realizations}
Let \(d\geq2\), and let \(\eta_0\) be an isotropic, \(V\)-ergodic point process
on \(V\) whose canonical \(G_{SO(V)}\)-action is essentially free. Given
\(q\geq1\) and \(\ell\geq0\), every sufficiently large prescribed intensity
is realized by an isotropic, \(V\)-weakly mixing point process \(\eta\),
supported on Delone configurations, together with a
\(G_{SO(V)}\)-equivariant measurable isomorphism
\[
    T:(\mathcal N_s(V),\eta_0)
      \longrightarrow
      (\mathcal N_s(V),\eta).
\]
Moreover,
\[
    \sigma_\eta(B_\varepsilon)
    =
    o(\varepsilon^{2q})
    \qquad(\varepsilon\downarrow0),
\]
\[
    \operatorname{Var}_\eta(N_{B_R})=O(R^{d-1})
    \qquad(R\to\infty),
\]
and \(\eta\) is maximally rigid and linearly \(\ell\)-rigid.
\end{corollary}

The same conclusion holds with \(SO(V)\) replaced by \(O(V)\), provided
\(\eta_0\) is orthogonally invariant and its canonical \(G_{O(V)}\)-action is
essentially free.

Examples of Euclidean-motion actions satisfying the translation-ergodicity
hypothesis are collected in Subsection~\ref{subsec:introduction-examples}.


\subsection{Ideas of the proof}
\label{subsec:introduction-proof-ideas}

We describe the higher-dimensional construction first. The one-dimensional
construction uses different technical tools because the natural
nearest-neighbor return graph on a Delone subset of \(\mathbb R\) is
two-ended, so the one-ended-forest step below is unavailable.
Section~\ref{sec:dimension-one} replaces it by nested return partitions. The
two constructions nevertheless solve the same two structural problems:
integerizing local point counts and retaining enough finite information to
recover the original state.

\emph{1. From an auxiliary cross-section to integer local masses.}
We first choose an auxiliary cocompact cross-section \(C\) for the full
\(G_K\)-action. This is not the translation cross-section appearing in the
main theorem. Rather, the full return sets to \(C\), projected through
\(K\backslash G_K\cong V\), give uniformly Delone scaffolds in \(V\). Each
scaffold site carries a frame in \(K\) and a label in \(C\), and these data
allow local objects constructed at the label to be placed equivariantly in
physical coordinates.

The associated Voronoi cells have uniformly controlled geometry. After
resolving boundary ties by the labels, they form an exact Borel partition. At
intensity \(\rho\), a cell \(D\) carries the real mass
\[
    \rho\,\lambda_V(D),
\]
whereas the next stage of the construction must replace each local piece by a
finite set of points. We therefore need local pieces whose \(\rho\)-masses are
integers. The bounded-range p.m.p.\ Borel graph constructed below is
nowhere two-ended, so Theorem~\ref{thm:measured-one-ended-forest} supplies a
one-ended spanning subforest. This measured forest theorem is the principal
nonstandard external input in the integerization. The resulting parent map is
used to modify neighboring cells by sets whose \(\rho\)-masses are the
relevant fractional parts. After
placement at the scaffold sites, the resulting sets \(\Pi_\rho(c)\) form an
exact partition of \(V\), have uniformly controlled geometry, and satisfy
\[
    \rho\lambda_V(\Pi_\rho(c))=q_\rho(c)\in\mathbb N.
\]

\emph{2. Replacing each packet by points without losing information.}
The packet \(\Pi_\rho(c)\) is replaced by exactly \(q_\rho(c)\) points. The
local points are chosen uniformly separated and so that the discrete and
continuous measures have the same moments through degree \(p-1\):
\[
    \sum_{u\in Q_\rho(c)}u^{\mathbf j}
    =
    \rho\int_{\Pi_\rho(c)}u^{\mathbf j}\,d\lambda_V(u),
    \qquad |\mathbf j|<p.
\]

Moment matching alone would not make the realization generating. Each local
set \(Q_\rho(c)\) is therefore required to contain the \((d+1)\)-point set
\(hS_{b(c)}\), which is left fixed by the moment correction. After placement in
physical coordinates, a set \(p+r hS_{b(c)}\) determines the position \(p\),
the frame \(r\), and the value \(b(c)\); since \(b\) is injective, it also
determines the label \(c\). Equation~\eqref{eq:framed-recovery} then recovers
the original state, proving injectivity of the point map. Such complete finite
sets occur arbitrarily far from every bounded region, which yields maximal
rigidity.

\emph{3. Zeroth-order cancellation and bounded displacement.}
After the local configurations are placed in physical coordinates, compare the
resulting counting measure with Lebesgue measure at density \(\rho\):
\[
    \Delta_x:=\delta_{P_x}-\rho\lambda_V.
\]
Because every packet has the correct total mass, \(\Delta_x\) has a first-order
divergence representation
\[
    \Delta_x
    =
    \sum_{i=1}^d\partial_i\Xi_{i,x},
\]
where the signed measures \(\Xi_{i,x}\) have uniformly bounded total variation
on unit balls. Testing this identity against smoothed indicators of unions of
unit cubes gives
\[
    \bigl|\#(P_x\cap U)-\rho\lambda_V(U)\bigr|
    \leq
    C_\rho\,\mathcal H^{d-1}(\partial U).
\]
Laczkovich's bounded-displacement criterion then implies that \(P_x\) is
bounded-displacement equivalent to the lattice
\(\rho^{-1/d}\mathbb Z^d\).

\emph{4. Higher moments and the Bartlett spectrum.}
The higher moment identities give
\[
    \Delta_x
    =
    \sum_{|\mathbf j|=p}
        \partial^{\mathbf j}\Theta_{\mathbf j,x},
\]
where the signed measures \(\Theta_{\mathbf j,x}\) are translation-covariant
and have uniformly bounded variation on unit balls. Their spectral measures
then give
\[
    \sigma_\eta(B_\varepsilon)
    =
    o(\varepsilon^{2p})
    \qquad(\varepsilon\downarrow0).
\]
Here translation ergodicity is used to exclude atoms at the origin. The same
estimate gives surface-order ball variance when \(2p>d+1\), and polynomial
cutoffs give linear \(\ell\)-rigidity when \(p>d+\ell\).

The roles of the three main ingredients are thus distinct: the forest produces
integer local masses, higher-order moment cancellation gives the Bartlett decay and
linear rigidity, and the finite sets \(S_{b(c)}\) make the resulting point map
injective. Together these ingredients
allow the final Delone configuration to retain the prescribed
dynamics while satisfying the geometric, spectral, and rigidity conclusions.
The final \(K\)-invariant translation cross-section is then recovered from the
resulting equivariant configuration by
Proposition~\ref{prop:configuration-realization}.

The one-dimensional construction has the same division of labor. The nested
return-block hierarchy of Section~\ref{sec:dimension-one} replaces the
one-ended parent forest and performs the integer rounding across successively
larger blocks. The three-point markers then play the role of the finite sets
\(S_{b(c)}\): they make the resulting point configuration generating and,
from arbitrarily distant complete markers, maximally rigid. Thus the two
dimensions use different geometric implementations of the same realization
strategy.


\subsection{Examples with ergodic translations}
\label{subsec:introduction-examples}

The translation-ergodicity hypothesis in the isotropic realization result
occurs in several quite different settings.
\begin{itemize}
\item
\emph{Restrictions of simple-group actions.}
After an orthogonal identification \(V\simeq\mathbb R^d\), there is an
embedding
\[
    \mathbb R^d\rtimes SO(d)
    \longrightarrow
    SL_{d+1}(\mathbb R),
    \qquad
    (v,A)\longmapsto
    \begin{pmatrix}
        A&v\\
        0&1
    \end{pmatrix}.
\]
By the Howe--Moore theorem \cite{HM79}, every ergodic p.m.p.\ action of
\(SL_{d+1}(\mathbb R)\) restricts to a Euclidean-motion action whose
translation subaction is mixing. The class includes both homogeneous actions
such as \(\Gamma\backslash SL_{d+1}(\mathbb R)\) for lattices \(\Gamma\) and
arbitrary ergodic p.m.p.\ actions of the ambient simple group.

\item
\emph{Gaussian actions.}
If \(\varphi:V\to\mathbb R\) is a continuous normalized positive-definite
\(SO(V)\)-invariant function, the centered stationary Gaussian field with
covariance \(\mathbb E[Z_uZ_w]=\varphi(u-w)\) carries a p.m.p.\
\(V\rtimes SO(V)\)-action. For such Gaussian actions in dimension at
least two, translation ergodicity implies translation mixing. By the Gaussian ergodicity criterion
\cite{Mar49}, ergodicity excludes atoms in the spectral measure; rotational
invariance and the Bessel decay of spherical measures then imply that this
atomless radial measure is Rajchman. Hence \(\varphi(v)\to0\) as
\(\|v\|\to\infty\), which for a stationary Gaussian field is equivalent
to mixing of the translation action.

\item
\emph{Determinantal processes and random tessellations.}
A translation-invariant determinantal point process with radial kernel is
orthogonally invariant, and its translation action is mixing \cite{Sos00}.
Likewise, an isotropic STIT tessellation is orthogonally invariant, and its
translation action is mixing \cite{LR11}. These give probabilistic and
stochastic-geometric examples not naturally obtained by restriction from
an ambient semisimple-group action.

\item
\emph{Pinwheel tilings.}
The planar pinwheel tiling hull is uniquely ergodic under translations, and
its unique invariant measure is \(SO(2)\)-invariant \cite{Rad93,Rad94}.
This gives a deterministic aperiodic example whose translation action is
weakly mixing by Corollary~\ref{cor:isotropic-ergodic-weak-mixing}.
Lotz--Klatt \cite[Example~5.6]{LK26} study related point processes obtained
by placing points in the pinwheel tiles, including the tile-centroid process,
by transport methods.
\end{itemize}
The examples illustrate the translation-ergodicity hypothesis; essential
freeness is an additional assumption in the realization theorems.


\subsection{Relation to other constructions}
\label{subsec:introduction-related-constructions}

\emph{Spatial constructions and transport.}
Many important hyperuniform processes arise from a prescribed spatial
mechanism: lattices and cut-and-project sets, determinantal and Coulomb
systems, perturbed lattices, and related models; see \cite{BH24,Tor18} and
the references therein. Perturbed lattices illustrate how sensitively
hyperuniformity may depend on the displacement field. In dimensions at least
three, bounded dependent perturbations can destroy hyperuniformity
\cite{DFHL24}, while sufficiently strong mixing assumptions can yield
surface-order fluctuations; see, for example, Flimmel \cite{Fli26}.
Gabrielli, Joyce and Torquato \cite{GJT08} constructed hyperuniform point
processes from equal-volume tilings and showed how preservation of successive
cell moments yields increasingly high orders of vanishing at low frequency.
Recent transport constructions provide a broader framework for related ideas:
Klatt, Last, Lotz and Yogeshwaran \cite{KLLY25} hyperuniformize broad classes
of stationary random measures, and Lotz and Klatt \cite{LK26} obtain
isotropic noncrystalline processes with arbitrarily high finite orders of
low-frequency suppression. These works, together with \cite{GJT08}, establish
much of the spatial mechanism behind the high-order suppression used here. In
particular, the present construction uses the same basic principle of
cellwise moment cancellation. The different constraint is the realization
problem: the measurable action is
prescribed first, and the construction must produce Delone point-process
coordinates that are equivariant and generating, so that the entire
dynamical state is recoverable. The forest integerization and the finite
recovery sets are used to make the established local cancellation mechanism
compatible with that requirement. The Beck-type surface-order lower bound
used above originates in Beck \cite{Bec87}; the author and Byl\'ehn develop
its random-measure form in \cite{BB26}.

\emph{Separated nets and bounded displacement.}
The deterministic geometry of the higher-dimensional realizations belongs to
the classical theory of uniformly spread sets. Laczkovich characterized
bounded-displacement equivalence to a lattice by boundary-discrepancy estimates
\cite{Lac92}. Bounded-displacement equivalence is stronger than bi-Lipschitz
equivalence, while arbitrary separated nets need not even be bi-Lipschitz
equivalent to a lattice, as shown by Burago and Kleiner
\cite{BK98}. Of
particular relevance here, Haynes, Kelly and Weiss studied bi-Lipschitz
(BL) and bounded-displacement (BD) classes of separated nets arising as
return times to sections of structured
\(\mathbb R^d\)-actions \cite{HKW14}. Our result runs in the opposite
direction: starting from an arbitrary essentially free ergodic measurable
action, we construct a generating section whose return-time configurations
are BD, hence BL, equivalent to a lattice on an invariant conull set. The BD
matching is geometric and is not asserted to vary measurably or equivariantly
with the state.

\emph{Diffraction and measurable dynamics.}
The distinction between the measurable action and the spatial statistics of
a particular realization has a standard analogue in diffraction theory.
Dworkin's correspondence identifies diffraction with spectral measures
generated by distinguished linear observables of a translation-bounded
measure dynamical system; in general these observables generate only a
proper subspace of the full Koopman representation
\cite{BL04,BLvE15,DM08,Gou03}. For an ergodic point process, let \(\rho_\eta\) denote the intensity and
\(\gamma_\eta\) the autocorrelation in the usual normalization. Then
\[
    \widehat\gamma_\eta=\rho_\eta^2\delta_0+\sigma_\eta.
\]
Accordingly, the Bartlett spectrum is the centered point-process part of
diffraction and records only the corresponding Dworkin subspace. It depends on
the chosen point-process coordinates and need not be preserved by measurable
dynamical isomorphism. Osada proved, for example, that a large class of
translation-invariant determinantal processes are measurably isomorphic to
homogeneous Poisson processes \cite{Osa21}, despite very different canonical
second-order statistics. As a concrete instance, consider the homogeneous
Poisson translation action in \(d\geq2\). The same measurable action then admits generating
Delone coordinates that combine arbitrarily high finite Bartlett order with
surface-order variance, maximal rigidity, any prescribed finite linear
rigidity, and BD equivalence to lattices.

\emph{Rigidity.}
Number rigidity was introduced by Ghosh and Peres \cite{GP17}: the exterior
of a bounded region determines the number of points inside. Their work also
showed that for zeros of the planar Gaussian analytic function the exterior
determines the first moment, and Ghosh and Krishnapur developed a hierarchy
of higher moment rigidity \cite{GK21}. The relation between suppressed
fluctuations and rigidity has since been studied in several directions; see
\cite{GL17,LR24,LR25}. In particular, Lachi\`eze-Rey \cite{LR24}
develops spectral criteria for higher-order linear rigidity of stationary
random measures. At the extreme end, Ghosh and Lebowitz proved maximal
rigidity for stealthy hyperuniform processes \cite{GL18}. In the present
construction maximal nonlinear rigidity has a different source: the isolated copies of the finite sets \(S_{b(c)}\) make the entire configuration recoverable from any
bounded exterior. Higher-order spectral cancellation independently yields
arbitrarily high finite orders of \emph{linear} rigidity in \(d\geq2\).

\emph{Spectral gaps.}
The companion paper \cite{Bjo26} treats the stealthy regime and the
one-dimensional inverse-square obstruction described above.


The paper is organized as follows. Section~\ref{sec:preliminaries} collects
the measure-theoretic, geometric, spectral and rigidity preliminaries.
Section~\ref{sec:dimension-one} gives the one-dimensional construction.
Section~\ref{sec:orbit-geometry-integerization} develops the framed orbit
geometry, one-ended forest and integerization used in higher dimensions.
Section~\ref{sec:local-filling-assembly} discretizes the integer-mass packets,
imposes the exact local moment identities, and proves global generation and
maximal rigidity.
Section~\ref{sec:spectral-decay-rigidity} establishes deterministic boundary
discrepancy, bounded displacement, Bartlett decay, number variance and linear
rigidity. Section~\ref{sec:proofs-main-theorems} verifies the main theorems and corollaries from these constructions.


\subsection{Acknowledgments}

The author was supported by the Swedish Research Council under
grant VR~11253322. The author thanks Mattias Byl\'ehn, Tobias Hartnick,
G\"unter Last and Rapha\"el Lachi\`eze-Rey for valuable discussions related
to this work. He is especially grateful to Luca Lotz and
Michael Klatt for sharing their related work, for a detailed comparison of
the two constructions, and for extensive comments on an earlier version. In
particular, they pointed out the connections with invariant and fair
partitions, averaging sets and designs, moment-preserving hyperuniform
constructions, and higher-order rigidity; their comments led to a substantially
expanded account of the relevant literature and of the relation between the
different approaches.


\section{Preliminaries}
\label{sec:preliminaries}

Throughout, a p.m.p.\ Borel action means a probability-measure-preserving
Borel action on a standard probability space. Such an action is
\emph{essentially free} if the stabilizer of almost every point is trivial.

\subsection{Euclidean and Fourier conventions}
\label{subsec:Euclidean-Fourier-conventions}

Let \(V\) be a real inner-product space of dimension \(d\geq1\).
We write \(\langle\cdot,\cdot\rangle\) and \(\|\cdot\|\) for its
inner product and norm, \(\lambda_V\) for the induced Lebesgue
measure, and \(B_R(v)\) for the open ball of radius \(R\) about
\(v\); we abbreviate \(B_R:=B_R(0)\). Fix once and for all an
orthonormal basis \(e_1,\ldots,e_d\) of \(V\). Multi-index notation,
including \(u^{\mathbf j}\), refers to the resulting coordinates.
Statements involving all polynomials of degree at most \(\ell\) are
independent of this choice of basis.

Let \(K\leq O(V)\) be closed, let \(e\) denote its identity, and set
\[
    G_K:=V\rtimes K,
    \qquad
    (v,k)(w,\ell):=(v+kw,k\ell).
\]
Thus \((v,k)^{-1}=(-k^{-1}v,k^{-1})\). Let \(K\backslash G_K\) denote the space of left \(K\)-cosets. We
identify it with \(V\) by
\[
    K(v,k)\longmapsto k^{-1}v.
\]
Under this identification, the left \(G_K\)-action
\[
    g.(Kh):=Khg^{-1},
    \qquad g,h\in G_K,
\]
on \(K\backslash G_K\) becomes
\[
    (v,k).u=ku-v,
    \qquad (v,k)\in G_K,\quad u\in V.
\]

We denote by \(\smash{C_c(V)}\), \(\smash{C_c^\infty(V)}\), and
\(\mathcal S(V)\) the spaces of compactly supported continuous
functions, compactly supported smooth functions, and Schwartz
functions on \(V\), respectively. We write
\[
    S^{d-1}:=\{u\in V:\|u\|=1\}
\]
for the unit sphere. For a Lipschitz function \(f\),
\(\operatorname{Lip}(f)\) denotes its Lipschitz constant, and for a
finite signed measure \(\nu\), \(\|\nu\|_{\rm TV}\) denotes its total
variation norm. For \(f\in L^1(V)\), we use the
Fourier transform
\[
    \widehat f(\xi)
    :=
    \int_V f(u)e^{-2\pi i\langle u,\xi\rangle}
    \,d\lambda_V(u),
    \qquad
    \xi\in V.
\]
The character associated with \(\xi\in V\) is
\(\chi_\xi(v):=e^{2\pi i\langle v,\xi\rangle}\), \(v\in V\).

For \(P\subset V\), define
\[
    \operatorname{sep}(P)
    :=
    \inf_{\substack{u_1,u_2\in P\\u_1\neq u_2}}
    \|u_1-u_2\|,
    \qquad
    \operatorname{covrad}(P)
    :=
    \sup_{v\in V}\operatorname{dist}(v,P).
\]
The set \(P\) is \emph{\((r,R)\)-Delone} if
\(\operatorname{sep}(P)\geq r\) and
\(\operatorname{covrad}(P)\leq R\), and it is \emph{Delone} if this
holds for some \(r,R>0\). A family of subsets of \(V\) is
\emph{uniformly Delone} if the same pair \((r,R)\) works for every
member of the family.


\subsection{Point processes and Euclidean symmetries}
\label{subsec:point-processes-symmetries}

Let \(\mathcal M_+(V)\) be the space of positive Radon measures on
\(V\), equipped with the vague topology, and let
\(\mathcal N_s(V)\subset\mathcal M_+(V)\) be the subspace of locally
finite simple counting measures. We identify a locally finite set
\(P\subset V\) with
\[
    \delta_P:=\sum_{u\in P}\delta_u.
\]

The action of \(G_K\) on \(V\) induces an action on
\(\mathcal M_+(V)\) by
\[
    \bigl((v,k).\omega\bigr)(f)
    :=
    \omega\bigl(f(k\,\cdot-v)\bigr),
    \qquad (v,k)\in G_K,\quad \omega\in\mathcal M_+(V),\quad f\in C_c(V),
\]
so that
\[
    (v,k).\delta_P=\delta_{kP-v}.
\]

A \emph{point process} on \(V\) is a Borel probability measure
\(\eta\) on \(\mathcal N_s(V)\). Let \(H\leq G_K\) be closed. We say
that \(\eta\) is \emph{\(H\)-invariant} if \(h_*\eta=\eta\) for every
\(h\in H\), \emph{\(H\)-ergodic} if the resulting \(H\)-action on
\((\mathcal N_s(V),\eta)\) is ergodic, and \emph{\(H\)-weakly
mixing} if its diagonal action on
\((\mathcal N_s(V)^2,\eta\otimes\eta)\) is ergodic. For the subgroups
\(V\times\{e\}\) and \(\{0\}\times K\), we abbreviate these terms as
\(V\)-invariance and \(K\)-invariance, and similarly for ergodicity
and weak mixing.

When \(d\geq2\), a \(V\)-invariant point process which is also
\(SO(V)\)-invariant is called \emph{isotropic}; if it is also
\(O(V)\)-invariant, it is called \emph{orthogonally invariant}.

\subsection{Linear statistics and the Bartlett spectrum}
\label{subsec:linear-statistics-Bartlett}

A \(V\)-invariant point process \(\eta\) has \emph{local second
moments} if
\[
    \int_{\mathcal N_s(V)}\omega(A)^2\,d\eta(\omega)<\infty
\]
for every bounded Borel set \(A\subset V\). Its first moment measure
is a multiple of Lebesgue measure; the corresponding constant
\(\rho_\eta\geq0\), determined by
\[
    \int_{\mathcal N_s(V)}\omega(f)\,d\eta(\omega)
    =
    \rho_\eta\int_V f\,d\lambda_V,
    \qquad f\in C_c(V),
\]
is the \emph{intensity} of \(\eta\).

For an integrable random variable \(Z\) on a probability space
\((\Omega,\mathbb P)\), write
\[
    Z^\circ:=Z-\int_\Omega Z\,d\mathbb P.
\]
In particular, for random variables on \((\mathcal N_s(V),\eta)\), the
centering is with respect to \(\eta\).
For a bounded Borel set \(A\subset V\), write
\[
    N_A(\omega):=\omega(A).
\]

Suppose that \(\eta\) has local second moments. For
\(f\in\mathcal S(V)\), the \emph{linear statistic}
\(\bS f(\omega):=\int_V f\,d\omega\) belongs to \(L^2(\eta)\), and
\[
    (\bS f)^\circ
    =
    \bS f-\rho_\eta\int_V f\,d\lambda_V.
\]
There is a unique positive Radon measure \(\sigma_\eta\) on \(V\)
such that
\[
    \int_{\mathcal N_s(V)}
    (\bS f)^\circ\overline{(\bS g)^\circ}\,d\eta
    =
    \int_V
    \widehat f(\xi)\overline{\widehat g(\xi)}
    \,d\sigma_\eta(\xi),
    \qquad f,g\in\mathcal S(V);
\]
see \cite[Proposition~8.2.I and Definition~8.2.II]{DVJ03}, with the
spectral variable rescaled to our Fourier convention. We call
\(\sigma_\eta\) the \emph{Bartlett spectrum} of \(\eta\). In
particular,
\[
    \operatorname{Var}_\eta(\bS f)
    =
    \int_V|\widehat f(\xi)|^2\,d\sigma_\eta(\xi).
\]
Moreover, \(\sigma_\eta\) is translation bounded
\cite[Proposition~8.2.II(iv)]{DVJ03}, that is,
\(\sup_{\xi\in V}\sigma_\eta(\xi+L)<\infty\) for every compact
\(L\subset V\).

We also use the covariance identity for indicators.

\begin{lemma}
\label{lem:Bartlett-indicators}
Let \(\eta\) be a \(V\)-invariant point process with local second moments,
and let \(A\subset V\) be bounded with \(\lambda_V(\partial A)=0\). Then
\[
    \operatorname{Var}_\eta(N_A)
    =
    \int_V|\widehat{\mathbf 1}_A(\xi)|^2\,d\sigma_\eta(\xi).
\]
More generally, the corresponding covariance identity holds for two
such bounded sets.
\end{lemma}

\begin{proof}
Choose \(f_n\in C_c^\infty(V)\) with \(0\leq f_n\leq1\), all
supported in one fixed compact set, such that
\(f_n\to\mathbf 1_A\) pointwise off \(\partial A\) and in
\(L^1(\lambda_V)\). Stationarity gives
\[
    \int\omega(\partial A)\,d\eta(\omega)
    =\rho_\eta\lambda_V(\partial A)=0,
\]
so \(\bS f_n\to N_A\) almost surely. The variables are dominated by
\(\omega(K)\) for a fixed compact \(K\), which belongs to
\(L^2(\eta)\) by the local second-moment assumption. Hence the
convergence also holds in \(L^2(\eta)\).

The Bartlett identity for Schwartz functions therefore shows that
\(\widehat f_n\) is Cauchy in \(L^2(\sigma_\eta)\), and hence converges
to some \(F\in L^2(\sigma_\eta)\). After passing to a subsequence,
\(\widehat f_n\to F\) almost everywhere. On the other hand,
\(L^1\)-convergence gives uniform convergence
\(\widehat f_n\to\widehat{\mathbf 1}_A\), so
\(F=\widehat{\mathbf 1}_A\) almost everywhere. Taking the limit gives the
stated identity. The covariance version follows by polarization.
\end{proof}

If \(\eta\) is \(K\)-invariant, then \(\sigma_\eta\) is
\(K\)-invariant as well. In particular, the
Bartlett spectrum of an isotropic point process is radial.

We call \(\eta\) \emph{hyperuniform} if
\[
    \operatorname{Var}_\eta(N_{B_R})=o(R^d)
    \qquad(R\to\infty).
\]
By
\cite[Proposition~3.3]{BH24}, this is equivalent to
\[
    \sigma_\eta(B_\varepsilon)
    =
    o(\varepsilon^d)
    \qquad
    (\varepsilon\downarrow0).
\]


\subsection{Rigidity}
\label{subsec:rigidity-definitions}

Let \(\eta\) be a \(V\)-invariant point process with local second
moments. Fix a bounded Borel set \(A\subset V\). We first make precise
what is meant by the information contained in the configuration outside
\(A\).

Let
\[
    \mathcal N_s(A^c)
    :=
    \{\nu\in\mathcal N_s(V):\nu(A)=0\},
\]
equipped with the Borel structure inherited from
\(\mathcal N_s(V)\), and let
\[
    r_{A^c}:\mathcal N_s(V)\longrightarrow\mathcal N_s(A^c),
    \qquad
    r_{A^c}(\omega):=\omega|_{A^c},
\]
be the restriction map. We denote by
\[
    \mathcal F_{A^c}
    :=
    r_{A^c}^{-1}\bigl(\mathcal B(\mathcal N_s(A^c))\bigr)
\]
the sigma-algebra generated by the outside configuration. Equivalently,
\(\mathcal F_{A^c}\) is the smallest sigma-algebra on
\(\mathcal N_s(V)\) for which all counting maps
\[
    \omega\longmapsto\omega(B),
    \qquad
    B\subset A^c\ \text{bounded Borel},
\]
are measurable. Thus an \(\mathcal F_{A^c}\)-measurable random
variable depends only on \(\omega|_{A^c}\), up to \(\eta\)-null sets.

We also need the linear information carried by the outside
configuration. Let \(\mathcal H_{A^c}\) be the closed subspace of
\(L^2(\eta)\) generated by the centered linear statistics
\[
    (\bS f)^\circ,
\]
where \(f\) ranges over bounded, compactly supported Borel functions
that vanish on \(A\). Every such statistic depends only on
\(\omega|_{A^c}\), and hence
\[
    \mathcal H_{A^c}
    \subset
    L^2(\mathcal F_{A^c},\eta).
\]
Membership in \(\mathcal H_{A^c}\) is stronger than mere
\(\mathcal F_{A^c}\)-measurability: it means that the random variable
can be approximated in \(L^2(\eta)\) by finite linear combinations of
centered linear statistics supported outside \(A\).

For \(\mathbf j\in\mathbb N_0^d\), define the polynomial moment inside
\(A\) by
\begin{equation}
\label{eq:interior-polynomial-moment}
    \mathfrak m_{\mathbf j}(\omega;A)
    :=
    \int_A u^{\mathbf j}\,d\omega(u).
\end{equation}

The term \emph{linear rigidity} for a closed-linear-span condition is used
by Bufetov--Dabrowski--Qiu \cite{BDQ18}; their point-process formulation is
motivated by the rigidity arguments of Ghosh and Peres \cite{GP17}.
Higher-order linear rigidity for polynomial moments of stationary random
measures is developed by Lachi\`eze-Rey \cite{LR24}. The definition below is
the corresponding local formulation for point processes. For an integer
\(\ell\geq0\),
we say that \(\eta\) is \emph{linearly \(\ell\)-rigid} if
\begin{equation}
\label{eq:linear-ell-rigidity-definition}
    \mathfrak m_{\mathbf j}(\,\cdot\,;A)^\circ
    \in\mathcal H_{A^c}
    \qquad
    (|\mathbf j|\leq\ell)
\end{equation}
for every bounded Borel set \(A\subset V\). Thus each centered
polynomial moment of degree at most \(\ell\) inside \(A\) lies in the
closed linear span of statistics observable outside \(A\).

We say that \(\eta\) is \emph{\(\ell\)-rigid} if
\[
    \mathfrak m_{\mathbf j}(\,\cdot\,;A)
\]
is \(\mathcal F_{A^c}\)-measurable for every bounded Borel set
\(A\subset V\) and every \(|\mathbf j|\leq\ell\). In other words, the
outside configuration determines all polynomial moments inside \(A\)
of degree at most \(\ell\), almost surely. The case \(\ell=0\) is
\emph{number rigidity}. Since
\(\mathcal H_{A^c}\subset L^2(\mathcal F_{A^c},\eta)\), linear
\(\ell\)-rigidity implies \(\ell\)-rigidity.

Finally, \(\eta\) is \emph{maximally rigid} if, for every bounded Borel
set \(A\subset V\), the outside configuration determines the entire
configuration almost surely. Equivalently, for every such \(A\) there
is a Borel map
\[
    \mathcal R_A:
    \mathcal N_s(A^c)\longrightarrow\mathcal N_s(V)
\]
such that
\[
    \mathcal R_A(\omega|_{A^c})=\omega
\]
for \(\eta\)-almost every \(\omega\). Maximal rigidity therefore
implies \(\ell\)-rigidity for every finite \(\ell\), but it does not by
itself imply linear rigidity.


\subsection{Translation cross-sections and configuration realizations}
\label{subsec:cross-sections-realizations}

Let \(G_K\curvearrowright(X,\mu)\) be a p.m.p.\ Borel action. For a
Borel set \(Y\subset X\) and \(x\in X\), its \emph{translation
return set} is
\[
    Y_x:=\{v\in V:(v,e).x\in Y\}.
\]
Throughout the paper, cross-sections for p.m.p.\ actions are understood
modulo invariant null sets. Thus \(Y\) is a \emph{translation cross-section}
if there is a \(G_K\)-invariant conull Borel set \(X_0\subset X\) such that
every \(V\times\{e\}\)-orbit in \(X_0\) meets \(Y\cap X_0\) and every
\(Y_x\), \(x\in X_0\), is locally finite. We henceforth replace \(Y\) by
\(Y\cap X_0\). It is \emph{\(K\)-invariant} if \((0,k).Y=Y\) for every
\(k\in K\).

A translation cross-section is \emph{separated},
\emph{cocompact}, or \emph{Delone} if, on such an invariant conull set,
the family \((Y_x)\) is, respectively, uniformly separated, uniformly
relatively dense, or uniformly Delone. For general background on separated
cross-sections and transverse measures, see \cite[Sections~3 and~4]{BHK25}.
For later use in dimension one, recall also that free Borel flows admit
separated cocompact cross-sections. Kechris proved the existence of complete
lacunary cross-sections for Borel actions of locally compact second countable
groups \cite[Corollary~1.2]{Kec92}; for flows, the regular cross-section theorem
of Slutsky gives the stronger bounded-gap conclusion
\cite[Theorem~9.1]{Slu19}. We use only the consequence that successive return
times are bounded both below and above; see \cite[Section~1]{Slu19} for the
historical background.

If \(Y\) is \(K\)-invariant, then
\[
    Y_{(v,k).x}=kY_x-v,
    \qquad
    (v,k)\in G_K,\quad x\in X.
\]
Hence the \emph{return-time map}
\[
    \kappa_Y:X\longrightarrow\mathcal N_s(V),
    \qquad
    \kappa_Y(x):=\delta_{Y_x},
\]
is \(G_K\)-equivariant. It is Borel by the Lusin--Novikov theorem
\cite[Theorem~18.10]{Kec95}, applied to the incidence set
\[
    \{(x,v)\in X\times V:(v,e).x\in Y\}.
\]
The point process
\[
    \eta_{Y,\mu}:=(\kappa_Y)_*\mu
\]
is called the \emph{return-time process}; when \(\mu\) is fixed,
we write \(\eta_Y\). It is \(G_K\)-invariant, and it is
\(V\)-ergodic whenever the restricted \(V\)-action on
\((X,\mu)\) is ergodic. If \(Y\) is separated, then \(\eta_Y\)
has local second moments. We define the \emph{intensity} of \(Y\) to be
\(\rho_{\eta_Y}\).

A \(K\)-invariant translation cross-section \(Y\) is
\emph{generating} if \(\kappa_Y\) is injective on a
\(G_K\)-invariant conull Borel subset of \(X\). By the
Lusin--Souslin theorem \cite[Theorem~15.1]{Kec95}, in this case
\(\kappa_Y\) is a measurable \(G_K\)-isomorphism onto its image.

When \(X=\mathcal N_s(V)\) carries the canonical \(G_K\)-action and \(\eta\)
is a \(G_K\)-invariant point process, the set
\[
    Y_{\rm can}:=\{\omega\in\mathcal N_s(V):\omega(\{0\})=1\}
\]
is the canonical translation cross-section on the nonempty configurations,
and its return-time map is the identity: \(\kappa_{Y_{\rm can}}(\omega)=\omega\).
More generally, if \(Y\) and \(Z\) are \(K\)-invariant translation
cross-sections for the same action and \(Y\) is generating, then
\[
    T_{Y,Z}:=\kappa_Z\circ\kappa_Y^{-1}
\]
is a Borel \(G_K\)-equivariant map from the point-process realization defined
by \(Y\) to that defined by \(Z\). If \(Z\) is also generating, this map is a
measurable \(G_K\)-isomorphism. Thus changing generating cross-sections amounts
to an equivariant recoding of configurations.

We emphasize that a \(K\)-invariant translation cross-section need
not be a cross-section for the full \(G_K\)-action. Its full return
set is
\[
    \{g\in G_K:g.x\in Y\}
    =
    \bigcup_{u\in Y_x}
    \{(0,k)(u,e):k\in K\},
\]
which is not locally finite in \(G_K\) when \(K\) is nondiscrete. In the
higher-dimensional construction we also use a different, auxiliary
cross-section \(C\) for the full \(G_K\)-action. Its full return set is locally
finite and is used only to build framed orbit coordinates; the final
realization cross-section is the \(K\)-invariant translation cross-section
obtained from the equivariant point configuration below.

The following elementary converse will be used repeatedly.

\begin{proposition}[Configuration realization]
\label{prop:configuration-realization}
Let \(X_0\subset X\) be a \(G_K\)-invariant conull Borel set and let
\[
    \kappa:X_0\longrightarrow\mathcal N_s(V),
    \qquad
    \kappa(x)=\delta_{P_x},
\]
be Borel and \(G_K\)-equivariant. Suppose that \(P_x\neq\varnothing\)
for every \(x\in X_0\). Then
\[
    Y:=\{x\in X_0:0\in P_x\}
\]
is a \(K\)-invariant translation cross-section (with conull domain
\(X_0\)), and for every \(x\in X_0\),
\[
    Y_x=P_x,
    \qquad
    \kappa_Y|_{X_0}=\kappa.
\]
If the family \((P_x)_{x\in X_0}\) is uniformly Delone, then \(Y\)
is a Delone translation cross-section. If \(\kappa\) is injective, then
\(Y\) is generating.
\end{proposition}

\begin{proof}
The set \(Y\) is Borel because
\[
    Y=\{x\in X_0:\kappa(x)(\{0\})=1\},
\]
and evaluation \(\omega\mapsto\omega(B)\) is Borel on
\(\mathcal N_s(V)\) for every Borel set \(B\subset V\). Equivariance gives
\(P_{(0,k).x}=kP_x\), so \(Y\) is \(K\)-invariant. Moreover, for
\(u\in V\),
\[
    u\in Y_x
    \quad\Longleftrightarrow\quad
    0\in P_{(u,e).x}=P_x-u
    \quad\Longleftrightarrow\quad
    u\in P_x.
\]
Thus \(Y_x=P_x\), and the remaining assertions follow immediately.
\end{proof}


\subsection{Borel and geometric preliminaries}
\label{subsec:Borel-geometric-preliminaries}

We record several elementary measurable facts that will be used repeatedly
when configurations, Voronoi cells, and finite subsets depend on a
parameter.

Let \(Z\) and \(E\) be standard Borel spaces and let
\(\mathscr A\subset Z\times E\) be Borel. For \(z\in Z\), write
\[
    \mathscr A[z]
    :=
    \{u\in E:(z,u)\in\mathscr A\}
\]
for the section over \(z\).

We will repeatedly use the Lusin--Novikov theorem
\cite[Theorem~18.10]{Kec95}. In the form needed here, it says that if
every section \(\mathscr A[z]\) is countable, then there are countably
many Borel partial maps
\[
    a_n:D_n\longrightarrow E,
    \qquad D_n\subset Z\ \text{Borel},
\]
whose graphs cover \(\mathscr A\). Thus the points of every countable
section can be enumerated in a manner that is Borel in the parameter:
\[
    \mathscr A[z]
    =
    \{a_n(z):z\in D_n\}.
\]
When the sections are locally finite subsets of \(V\), this allows us,
for example, to check a condition against every point of the
configuration by checking countably many Borel conditions involving
the functions \(a_n\).

We also use the following parameter-integration fact. If
\(\mathscr A\subset Z\times V\) and
\(h:Z\times V\to[0,\infty]\) are Borel, then
\[
    z\longmapsto
    \int_{\mathscr A[z]}h(z,u)\,d\lambda_V(u)
\]
is Borel. Indeed, the integrand
\[
    (z,u)\longmapsto
    \mathbf 1_{\mathscr A}(z,u)h(z,u)
\]
is Borel, and the assertion follows from the usual parameterized
integration theorem. The same conclusion holds for real- or
complex-valued \(h\) whenever the integral is absolutely convergent.

Finally, we use Borel linear orders to make finite choices canonically.
A \emph{Borel linear order} on a standard Borel space \(E\) is a
linear order \(\preccurlyeq\) for which
\[
    \{(u,v)\in E\times E:u\preccurlyeq v\}
\]
is Borel. Every standard Borel space admits such an order: choose a
Borel injection \(b:E\to\mathbb R\) and pull back the usual order on
\(\mathbb R\). If \(\mathscr A\subset Z\times E\) is Borel with
finite nonempty sections, then
\[
    z\longmapsto\min_{\preccurlyeq}\mathscr A[z]
\]
is Borel. Indeed, enumerate each section by Lusin--Novikov and select
the least of the finitely many enumerated points.

For a locally finite set \(P\subset V\) and \(u\in P\), let
\[
    \operatorname{Vor}(P,u)
    :=
    \{v\in V:\|v-u\|\leq\|v-w\|
        \text{ for every }w\in P\}
\]
denote the closed Voronoi cell of \(u\).

\begin{lemma}
\label{lem:Delone-Voronoi-geometry}
If \(P\) is \((r,R)\)-Delone and \(u\in P\), then
\[
    \overline B_{r/2}(u)
    \subset
    \operatorname{Vor}(P,u)
    \subset
    \overline B_R(u).
\]
Moreover, \(\operatorname{Vor}(P,u)\) is a compact convex polytope
and its boundary has \(\lambda_V\)-measure zero.
\end{lemma}

\begin{proof}
If \(\|v-u\|\leq r/2\) and \(w\in P\setminus\{u\}\), then
\[
    \|v-w\|
    \geq \|w-u\|-\|v-u\|
    \geq r-\|v-u\|
    \geq \|v-u\|,
\]
which proves the first inclusion. If \(\|v-u\|>R\), the covering
property gives \(w\in P\) with
\(\|v-w\|\leq R<\|v-u\|\), proving the second.

Only the finitely many points of
\(P\cap\overline B_{2R}(u)\) can contribute a defining inequality
on \(\overline B_R(u)\). Hence the cell is a bounded intersection
of finitely many closed half-spaces. Its boundary is contained in
a finite union of bisector hyperplanes and is therefore
\(\lambda_V\)-null.
\end{proof}

We repeatedly use the following packing bound. If \(P\) is
\(r\)-separated, then for every \(v\in V\) and \(t>0\),
\begin{equation}
\label{eq:Delone-packing-bound}
    \#\bigl(P\cap B_t(v)\bigr)
    \leq
    \frac{\lambda_V(B_{t+r/2})}{\lambda_V(B_{r/2})}
    =
    \left(1+\frac{2t}{r}\right)^d.
\end{equation}
Indeed, the balls \(B_{r/2}(u)\), \(u\in P\cap B_t(v)\), are
pairwise disjoint and contained in \(B_{t+r/2}(v)\).

Let \(Z\) be a standard Borel space. Suppose now that
\(z\mapsto\delta_{P_z}\) is a Borel family of locally finite configurations
and that \(a:Z\to V\) is Borel with \(a(z)\in P_z\). Then the incidence
set
\[
    \mathscr V
    :=
    \{(z,v)\in Z\times V:
        v\in\operatorname{Vor}(P_z,a(z))\}
\]
is Borel. To see this, choose a Lusin--Novikov enumeration \(u_n(z)\)
of the points of \(P_z\). Then
\[
    v\in\operatorname{Vor}(P_z,a(z))
    \quad\Longleftrightarrow\quad
    \|v-a(z)\|\leq\|v-u_n(z)\|
\]
for every enumerated point \(u_n(z)\), so membership in \(\mathscr V\)
is a countable conjunction of Borel conditions. It follows from the
parameter-integration statement above that
\[
    z\longmapsto
    \int_{\operatorname{Vor}(P_z,a(z))}
        h(z,v)\,d\lambda_V(v)
\]
is Borel whenever \(h\) is Borel and the integral is absolutely
convergent.

Suppose in addition that the configurations \(P_z\) are uniformly
\((r,R)\)-Delone. Translate each distinguished cell to the origin and put
\[
    D_z
    :=
    \operatorname{Vor}(P_z,a(z))-a(z).
\]
By Lemma~\ref{lem:Delone-Voronoi-geometry},
\[
    \overline B_{r/2}\subset D_z\subset\overline B_R.
\]
Thus \(D_z\) belongs to the space
\(\mathcal K(\overline B_R)\) of nonempty compact subsets of
\(\overline B_R\), equipped with the Hausdorff metric. The map
\[
    z\longmapsto D_z
\]
is Borel. Indeed, the corresponding incidence relation is Borel and
has compact sections; equivalently, for every open
\(U\subset\overline B_R\), the set
\[
    \{z:D_z\cap U\neq\varnothing\}
\]
is Borel. We shall use this Borel dependence when geometric
constructions are applied to Voronoi cells parameter by parameter. See also
\cite[Proposition~7.2 and the proof of Proposition~7.4]{ABC25} for
measurable Voronoi constructions for locally finite subsets of general
locally compact second countable groups.

Closed Voronoi cells cover \(V\), but neighboring cells meet along
their common bisectors. These overlaps are contained in the cell
boundaries and hence are \(\lambda_V\)-null. Later our sites will carry
distinct Borel labels. Fixing a Borel linear order on the label space
then gives an exact partition: whenever a point belongs to several
closed Voronoi cells, assign it to the site with least label. This
alters each closed cell only on its null boundary. The labelled cells
\(W_c\) used below are obtained in exactly this way; their Borel and
\(G_K\)-covariance properties will be checked when they are introduced.


\subsection{Measured graphings and one-ended forests}
\label{subsec:measured-graphings-forests}

Let \((Z,\nu)\) be a standard probability space. A \emph{Borel graph}
on \(Z\) is a symmetric, irreflexive Borel relation
\[
    \mathcal G\subset Z\times Z;
\]
we regard \((z,w)\in\mathcal G\) as an undirected edge between \(z\)
and \(w\). It is \emph{locally finite} if every vertex has finite
degree. In that case its connected components are countable, and we
write \(\mathcal R_{\mathcal G}\) for the resulting countable Borel
equivalence relation.

A Borel graph is a \emph{forest} if it contains no finite cycles, or
equivalently if every connected component is a tree. A
\emph{spanning subforest} of \(\mathcal G\) is a Borel forest
\(\mathcal T\subseteq\mathcal G\) with the same vertex set. We do not
require a component of \(\mathcal G\) to remain connected after edges
are removed; the components of \(\mathcal T\) are the trees produced
by the spanning forest.

The graph \(\mathcal G\) is \emph{p.m.p.} if
\(\mathcal R_{\mathcal G}\) preserves \(\nu\): every Borel partial
bijection
\[
    \theta:A\longrightarrow B
\]
whose graph is contained in \(\mathcal R_{\mathcal G}\) satisfies
\(\nu(A)=\nu(B)\). It is \emph{aperiodic} if the connected component of \(z\) is infinite
for \(\nu\)-almost every \(z\in Z\).

We next recall the notion of an end. If \(G\) is a connected locally
finite infinite graph and \(F\subset V(G)\) is finite, deleting \(F\)
may separate \(G\) into several infinite components. The number of
\emph{ends} of \(G\) is
\[
    \sup_{F\subset V(G)\ {\rm finite}}
    \#\{\text{infinite connected components of }G\setminus F\}.
\]
This number may be \(1,2,3,\ldots\), or infinite. For example, three
rays joined at one vertex give a three-ended tree, while an infinite regular tree of degree at least three has infinitely many ends. A graph is \emph{one-ended} if
the number above is one and \emph{two-ended} if it is two.

We single out the one- and two-ended cases because they play different
roles in the measured forest theorem used below. One-ended trees are
the desired output of the theorem. They have a canonical direction
toward infinity: after deleting a vertex, exactly one adjacent
component is infinite, so there is a unique neighboring vertex lying
toward the end. This will give the parent map used later for
integerization. Two-ended components, on the other hand, are precisely
the obstruction appearing in the theorem: under the p.m.p.\ hypotheses
below, they prevent a spanning subforest from having one-ended components at
almost every vertex. Components with three or more ends, or infinitely many
ends, are not excluded by the theorem.

Accordingly, a p.m.p.\ Borel graph \(\mathcal G\) is called
\emph{\(\nu\)-nowhere two-ended} if the set of vertices belonging to
two-ended components has \(\nu\)-measure zero. A Borel spanning subforest
\(\mathcal T\subseteq\mathcal G\) is \emph{\(\nu\)-almost-everywhere
one-ended} if
\[
    \nu\bigl(
        \{z\in Z:\text{the \(\mathcal T\)-component of \(z\) is one-ended}\}
    \bigr)=1.
\]

The following theorem is the principal nonstandard external
measured-combinatorial input in the higher-dimensional construction. We use
exactly the statement recorded here.

\begin{theorem}\cite[Theorem~2.1]{CGMTD26}
\label{thm:measured-one-ended-forest}
Let \(\mathcal G\) be an aperiodic locally finite p.m.p.\ Borel
graph on \((Z,\nu)\). Then \(\mathcal G\) admits a Borel
\(\nu\)-almost-everywhere one-ended spanning subforest if and only
if \(\mathcal G\) is \(\nu\)-nowhere two-ended.
\end{theorem}

If \(\mathcal T\) is the forest supplied by the theorem, let
\[
    N_{\mathcal T}
    :=
    \{z\in Z:\text{the \(\mathcal T\)-component of \(z\) is not one-ended}\}.
\]
Then \(\nu(N_{\mathcal T})=0\). For the pointwise constructions below we
remove the entire \(\mathcal R_{\mathcal G}\)-saturation of this exceptional
set. Since \(\mathcal R_{\mathcal G}\) is a countable p.m.p.\ equivalence
relation, the saturation of a null set is again null. Thus, after restricting
to an \(\mathcal R_{\mathcal G}\)-invariant conull Borel subset, every
remaining component of \(\mathcal T\) is one-ended.

We record the elementary structure of such forests.

\begin{lemma}
\label{lem:one-ended-forest-orientation}
Let \(\mathcal T\) be a locally finite Borel forest on a standard Borel
space \(Z\), all of whose components are one-ended. There is a Borel map
\[
    \operatorname{par}:Z\longrightarrow Z
\]
such that \(\operatorname{par}(z)\) is the unique neighbor of \(z\)
lying toward the end of its component. If
\[
    \operatorname{Ch}(z)
    :=
    \{w\in Z:\operatorname{par}(w)=z\},
    \qquad
    \operatorname{Desc}(z)
    :=
    \{w\in Z:\operatorname{par}^{\,n}(w)=z
        \text{ for some }n\geq0\},
\]
then the descendant relation is Borel and every
\(\operatorname{Desc}(z)\) is finite. If there is \(D<\infty\) such that
\(\deg_{\mathcal T}(z)\leq D\) for every \(z\in Z\), then
\(\#\operatorname{Ch}(z)\leq D-1\) for every \(z\in Z\).
\end{lemma}

\begin{proof}
For adjacent \(z,w\), the vertex \(w\) is
\(\operatorname{par}(z)\) precisely when the component containing
\(w\) after deleting \(z\) is infinite. In a one-ended tree there
is exactly one such neighbor. This condition is Borel: in a locally
finite Borel graph, that component is infinite if and only if for
every \(n\) there is a simple path of length \(n\) starting at
\(w\) and avoiding \(z\). Thus \(\operatorname{par}\) is Borel,
and so is the descendant relation.

Suppose that \(\operatorname{Desc}(z)\) is infinite. With the edges
oriented from each vertex to its parent, the descendant set is an infinite
rooted tree with root \(z\). Local finiteness of \(\mathcal T\) makes this
rooted tree finitely branching. K\"onig's lemma therefore gives an infinite
ray
\[
    z=z_0,z_1,z_2,\ldots,
    \qquad
    \operatorname{par}(z_{n+1})=z_n.
\]
This ray moves away from the unique end of the component. On the other hand,
the parent map is defined at every vertex, and acyclicity implies that its
successive iterates are distinct. Hence
\[
    z,\operatorname{par}(z),\operatorname{par}^{\,2}(z),\ldots
\]
is an infinite ray toward that end. After deleting \(z\), the two tails lie
in distinct infinite components, contradicting one-endedness. The child bound
follows since every vertex has exactly one parent.
\end{proof}


\subsection{Euclidean-motion actions with ergodic translations}
\label{subsec:Euclidean-actions}

Let \(G_K\curvearrowright(X,\mu)\) be a p.m.p.\ Borel action. We use
\(V\)-ergodic and \(V\)-weakly mixing, equivalently
\emph{translation-ergodic} and \emph{translation-weakly-mixing}, for the
corresponding properties of the restricted \(V\times\{e\}\)-action. The main results assume \(V\)-ergodicity, which implies
\(G_K\)-ergodicity. The next
proposition records a useful consequence of the additional
\(K\)-symmetry: if the nonzero \(K\)-orbits in \(V\) are
uncountable, then \(V\)-ergodicity already implies
\(V\)-weak mixing. In particular, this applies to \(K=SO(V)\) and
\(K=O(V)\) when \(d\geq2\).

Let \(\mathcal H_X:=L^2(X,\mu)\), and let \(\pi_X\) be the Koopman
\(G_K\)-representation on \(\mathcal H_X\), defined by
\[
    (\pi_X(g)F)(x):=F(g^{-1}.x),
    \qquad g\in G_K,\quad F\in\mathcal H_X.
\]
A vector \(\xi\in V\) is a
\emph{\(V\)-eigenvalue} if there is a nonzero
\(F\in\mathcal H_X\) such that
\[
    \pi_X((v,e))F=\chi_\xi(v)F
    \qquad (v\in V).
\]

\begin{proposition}[Rotational orbits of translation eigenvalues]
\label{prop:K-orbits-of-eigenvalues}
Suppose that the restricted \(V\)-action is ergodic. The
\(K\)-orbit of every \(V\)-eigenvalue is countable. Consequently,
if every nonzero \(K\)-orbit in \(V\) is uncountable, then the
restricted \(V\)-action is weakly mixing.
\end{proposition}

\begin{proof}
Let \(E_\xi\subset\mathcal H_X\) be the eigenspace associated with
\(\xi\). For \(v\in V\) and \(k\in K\),
\[
    \pi_X((v,e))\pi_X((0,k))
    =
    \pi_X((0,k))\pi_X((k^{-1}v,e)),
\]
so \(\pi_X((0,k))E_\xi=E_{k\xi}\). Eigenspaces corresponding
to distinct characters are orthogonal, so separability of
\(\mathcal H_X\) implies that there are at most countably many
\(V\)-eigenvalues. Hence every eigenvalue has a countable
\(K\)-orbit.

Under the additional hypothesis, \(0\) is the only
\(V\)-eigenvalue. Its eigenspace consists of the constants by
\(V\)-ergodicity, and the \(V\)-action is therefore weakly mixing;
see \cite[Proposition~1.2]{BR88}.
\end{proof}

\begin{corollary}[Isotropy and translation weak mixing]
\label{cor:isotropic-ergodic-weak-mixing}
Suppose \(d\geq2\) and \(K=SO(V)\) or \(K=O(V)\). If the restricted
\(V\)-action is ergodic, then it is weakly mixing. In particular,
every \(V\)-ergodic isotropic or orthogonally invariant point
process is \(V\)-weakly mixing.
\end{corollary}

\begin{proof}
If \(\xi,\zeta\in V\setminus\{0\}\) have the same norm, there is
\(k\in SO(V)\) with \(k\xi=\zeta\). Thus every nonzero
\(SO(V)\)-orbit, and hence every nonzero \(O(V)\)-orbit, is a
sphere and is uncountable.
\end{proof}



\section{The one-dimensional construction}
\label{sec:dimension-one}

The one-dimensional argument has three parts. We first choose a separated
cocompact cross-section for the flow, which gives a suspension whose roof
function is bounded above and below. We then construct nested return sets
whose orbit intervals form successively coarser return blocks: every block at
the next level is the union of two or three blocks at the preceding level.
This hierarchy is used to assign an integer count to each roof interval, with
logarithmic discrepancy on every finite orbit interval. Finally, these integer
counts are realized by uniformly Delone point configurations containing
recognizable three-point markers.

\subsection{Bounded suspension model}
\label{subsec:bounded-suspension}

Throughout this section, let
\(\mathbb R\curvearrowright(X,\mu)\) be an essentially free ergodic
p.m.p.\ Borel action. After replacing \(X\) by an invariant conull Borel set,
we assume that the action is free.

By the cross-section results recalled in
Subsection~\ref{subsec:cross-sections-realizations}, choose a separated
cocompact Borel cross-section \(Z\subset X\). Thus there are constants
\(r,R>0\) such that every return set
\[
    Z_x:=\{t\in\mathbb R:t.x\in Z\}
\]
is \(r\)-separated and \(R\)-relatively dense. For \(z\in Z\), let
\[
    \tau(z):=\min\{t>0:t.z\in Z\}
\]
be the first positive return time and put
\[
    Tz:=\tau(z).z.
\]
The return-time relation is Borel and locally finite, so \(\tau\) and \(T\)
are Borel. Separation gives \(\tau(z)\ge r\). If two consecutive returns had
gap greater than \(2R\), the midpoint of that gap would have distance greater
than \(R\) from \(Z_z\), contradicting relative denseness. Hence
\begin{equation}
\label{eq:bounded-roof}
    0<\tau_-\leq\tau(z)\leq\tau_+<\infty,
    \qquad z\in Z,
\end{equation}
for instance with \(\tau_-=r\) and \(\tau_+=2R\).

Let \(\nu\) be the normalized transverse invariant measure on \(Z\). Then
\(T:(Z,\nu)\to(Z,\nu)\) is an ergodic p.m.p.\ automorphism, and the standard
cross-section representation identifies the original flow, modulo the
invariant null set already removed, with the suspension
\begin{equation}
\label{eq:suspension-measure}
    X
    \cong
    \{(z,s):z\in Z,\ 0\leq s<\tau(z)\},
    \qquad
    d\mu(z,s)
    =
    \frac{d\nu(z)\,ds}{\int_Z\tau\,d\nu};
\end{equation}
see \cite{AK42}. We use this suspension model throughout the section.

For \(z\in Z\) and \(n\in\mathbb Z\), define
\[
    \tau_n(z)
    :=
    \begin{cases}
        \displaystyle\sum_{j=0}^{n-1}\tau(T^jz), & n>0,\\[6pt]
        0, & n=0,\\[4pt]
        \displaystyle-\sum_{j=n}^{-1}\tau(T^jz), & n<0.
    \end{cases}
\]
Then
\begin{equation}
\label{eq:suspension-cocycle}
    T^nz=\tau_n(z).z,
    \qquad
    \tau_{n+m}(z)
    =
    \tau_n(T^mz)+\tau_m(z),
    \qquad z\in Z,\quad m,n\in\mathbb Z.
\end{equation}

\subsection{Nested return partitions}
\label{subsec:nested-return-partitions}

We next construct a nested sequence
\[
    Z=A_0\supset A_1\supset A_2\supset\cdots
\]
of Borel return sets. The only multiple-tower result needed is the following
special case of Alpern's theorem.

\begin{lemma}[Two-height Rokhlin partition]
\label{lem:two-three-Rokhlin}
Let \(S\) be an aperiodic invertible p.m.p.\ transformation of a standard
probability space \((W,\lambda)\). There are measurable sets
\(B^{(2)},B^{(3)}\subset W\), each of measure \(1/5\), such that
\[
    B^{(2)},\ SB^{(2)},\
    B^{(3)},\ SB^{(3)},\ S^2B^{(3)}
\]
form a partition of \(W\) modulo null sets.
\end{lemma}

\begin{proof}
Apply the multiple Rokhlin tower theorem
\cite[Theorem~1.1]{EP97} with heights \(2\) and \(3\) and prescribed base
measures \(1/5\) and \(1/5\). The compatibility condition is
\(2/5+3/5=1\).
\end{proof}

We apply the lemma recursively. Set \(A_0:=Z\) and \(T_0:=T\). Suppose that
\(A_k\) has been defined with positive measure. Let \(T_k\) be the first-return
transformation of \(T\) to \(A_k\), and let
\[
    \nu_k:=\frac{\nu|_{A_k}}{\nu(A_k)}.
\]
The induced transformation \(T_k\) is an aperiodic p.m.p.\ automorphism of
\((A_k,\nu_k)\). Applying Lemma~\ref{lem:two-three-Rokhlin} to \(T_k\) gives
measurable sets
\(A_{k+1}^{(2)},A_{k+1}^{(3)}\subset A_k\) such that
\begin{equation}
\label{eq:23-multitower}
\begin{aligned}
    A_k
    ={}&
    A_{k+1}^{(2)}
    \sqcup T_kA_{k+1}^{(2)}\\
    &\sqcup A_{k+1}^{(3)}
    \sqcup T_kA_{k+1}^{(3)}
    \sqcup T_k^2A_{k+1}^{(3)}
\end{aligned}
\end{equation}
modulo null sets, with
\[
    \nu_k(A_{k+1}^{(2)})
    =
    \nu_k(A_{k+1}^{(3)})=\frac15.
\]
Put
\[
    A_{k+1}:=A_{k+1}^{(2)}\sqcup A_{k+1}^{(3)}.
\]

The partition \eqref{eq:23-multitower} determines the return time from
\(A_{k+1}\) to itself under \(T_k\). A point of
\(A_{k+1}^{(2)}\) returns after exactly two iterates of \(T_k\), while a point
of \(A_{k+1}^{(3)}\) returns after exactly three. Indeed, the intermediate
tower levels are disjoint from \(A_{k+1}\), and applying \(T_k\) to
\eqref{eq:23-multitower} shows, after cancelling the common non-top levels,
that
\[
    T_k^2A_{k+1}^{(2)}\sqcup T_k^3A_{k+1}^{(3)}
    =
    A_{k+1}^{(2)}\sqcup A_{k+1}^{(3)}
\]
modulo null sets. Consequently, between two successive visits to
\(A_{k+1}\) there are exactly two or three successive visits to \(A_k\).

The base measures satisfy
\begin{equation}
\label{eq:hierarchy-base-measures}
    \nu(A_{k+1})=\frac25\nu(A_k),
    \qquad
    \nu(A_k)=\left(\frac25\right)^k.
\end{equation}
Hence \(\nu(\bigcap_{k\geq0}A_k)=0\).

The preceding statements initially hold modulo null sets. Choose Borel
representatives of all tower bases, and for each \(k\) let \(D_k\) be the null
set on which the level-\(k\) partition or the asserted first-return description
fails. Remove the \(T\)-saturation of
\[
    \bigcap_{k\geq0}A_k
    \ \cup\
    \bigcup_{k\geq0}D_k.
\]
After this single invariant null-set removal, all tower partitions and
first-return statements hold pointwise, the inclusions
\(A_{k+1}\subset A_k\) are exact, and
\begin{equation}
\label{eq:empty-hierarchy-intersection}
    \bigcap_{k\geq0}A_k=\varnothing.
\end{equation}
We keep the same notation for the restricted objects.

For \(z\in A_k\), let
\[
    r_k(z):=\min\{n\geq1:T^nz\in A_k\}
\]
and define the \emph{level-\(k\) return block}
\[
    B_k(z):=\{z,Tz,\ldots,T^{r_k(z)-1}z\}.
\]
Thus the level-\(k\) return blocks are precisely the orbit intervals between
successive visits to \(A_k\). For each \(k\) they partition every \(T\)-orbit.
Moreover, the partitions are nested: every level-\((k+1)\) return block is,
in orbit order, the disjoint union of either two or three level-\(k\) return
blocks. It follows inductively that
\begin{equation}
\label{eq:canonical-block-size}
    2^k\leq r_k(z)\leq3^k,
    \qquad z\in A_k.
\end{equation}

We shall also need to know at which level a cut between two consecutive
level-zero blocks ceases to be a block boundary. For \(z\in Z\), consider the
cut between \(T^{-1}z\) and \(z\). Since the starting points of the level-\(k\)
return blocks are exactly the points of \(A_k\), this cut is a level-\(k\)
block boundary if and only if \(z\in A_k\). By
\eqref{eq:empty-hierarchy-intersection},
\[
    \ell(z):=\max\{k\geq0:z\in A_k\}<\infty.
\]
Hence the cut is a boundary at levels \(0,1,\ldots,\ell(z)\), but not at level
\(\ell(z)+1\). Equivalently, the level-\(\ell(z)\) return blocks immediately
on the two sides of the cut are two consecutive children of the same
level-\((\ell(z)+1)\) return block. The map
\(\ell:Z\to\mathbb N_0\) is Borel because
\[
    \{\ell=k\}=A_k\setminus A_{k+1}.
\]


\subsection{Hierarchical integer rounding}
\label{subsec:hierarchical-rounding}

Fix \(\rho>0\), and for \(s\in\mathbb R\) write
\[
    \{s\}:=s-\lfloor s\rfloor
\]
for its fractional part. The real mass assigned to the suspension interval
over \(z\in Z\) is \(\rho\tau(z)\). We shall define an integer
\(q_\rho(z)\) for each \(z\) so that sums of the \(q_\rho\)'s differ from the
corresponding real masses by at most a logarithmic term on every finite orbit
interval.

For a level-\(k\) return block \(B\), define
\[
    L_\rho(B)
    :=
    \rho\sum_{z\in B}\tau(z),
    \qquad
    n_\rho(B)
    :=
    \lfloor L_\rho(B)\rfloor.
\]
Suppose that a level-\((k+1)\) return block \(B\) is the union, in orbit order,
of its level-\(k\) children \(B_1,\ldots,B_m\), where \(m\in\{2,3\}\). Define
\[
    c_\rho(B)
    :=
    n_\rho(B)-\sum_{i=1}^m n_\rho(B_i).
\]
Since \(L_\rho(B)=\sum_iL_\rho(B_i)\),
\begin{equation}
\label{eq:one-dimensional-child-correction}
    c_\rho(B)
    =
    \left\lfloor
        \sum_{i=1}^m\{L_\rho(B_i)\}
    \right\rfloor,
    \qquad
    0\leq c_\rho(B)\leq m-1.
\end{equation}
There are \(m-1\) cuts between consecutive children of \(B\). Select the
first \(c_\rho(B)\) of these cuts in orbit order, and assign one unit to the
level-zero block immediately to the right of each selected cut.

This prescription associates each cut with exactly one level. To see this,
fix \(z\in Z\) and put \(k:=\ell(z)\). The cut immediately to the left of
\(z\) is a boundary between level-\(k\) return blocks, but not between
level-\((k+1)\) return blocks. Hence it is one of the child boundaries of the
unique level-\((k+1)\) return block containing it, and this is the only
transition at which the cut occurs as a boundary between children. Let
\(\iota_\rho(z)\in\{0,1\}\) be the indicator that this cut is among the
selected child boundaries of that block, and define
\begin{equation}
\label{eq:one-dimensional-local-count}
    q_\rho(z)
    :=
    \lfloor\rho\tau(z)\rfloor+\iota_\rho(z).
\end{equation}

The function \(q_\rho:Z\to\mathbb Z_{\geq0}\) is Borel. Indeed,
\(k=\ell(z)\) is Borel. Since the return time from \(A_{k+1}\) to itself under
\(T_k\) is either two or three, the initial point \(a\in A_{k+1}\) of the
unique level-\((k+1)\) return block containing the cut is the unique point of
\[
    \{T_k^{-1}z,T_k^{-2}z\}
\]
that belongs to \(A_{k+1}\). Thus \(a\), the number
\(m\in\{2,3\}\) of level-\(k\) children of \(B_{k+1}(a)\), and the position of
\(z\) among the starting points of the noninitial children are Borel functions
of \(z\). The child masses in \eqref{eq:one-dimensional-child-correction} are
finite sums of the Borel function \(\tau\). Therefore the condition that the
cut before \(z\) is one of the first \(c_\rho(B_{k+1}(a))\) child boundaries
is Borel, proving the claim.

The next lemma gives the exact sum of the local integer counts on a return
block.

\begin{lemma}[Return-block rounding]
\label{lem:canonical-block-rounding}
Let \(B=B_k(a)\) be a level-\(k\) return block. Then
\begin{equation}
\label{eq:canonical-block-rounding}
    \sum_{w\in B}q_\rho(w)
    =
    n_\rho(B)+\zeta_\rho(B),
    \qquad
    \zeta_\rho(B)\in\{0,1\},
\end{equation}
where \(\zeta_\rho(B)=\iota_\rho(a)\) is the possible contribution assigned
to the cut immediately to the left of \(B\). In particular,
\begin{equation}
\label{eq:canonical-block-discrepancy}
    \left|
        \sum_{w\in B}q_\rho(w)-L_\rho(B)
    \right|
    \leq1.
\end{equation}
\end{lemma}

\begin{proof}
We first sum only the corrections assigned to cuts strictly inside \(B\), so
that the correction \(\iota_\rho(a)\) at the left boundary is omitted. We
claim by induction on \(k\) that the resulting sum is exactly \(n_\rho(B)\).

For \(k=0\), the block is \(B=\{a\}\), there are no internal cuts, and the
claim is
\[
    \lfloor\rho\tau(a)\rfloor=n_\rho(B).
\]
Suppose now that \(k\geq1\) and that
\[
    B=B_1\sqcup\cdots\sqcup B_m,
    \qquad m\in\{2,3\},
\]
is the decomposition of \(B\) into its level-\((k-1)\) children. By the
induction hypothesis, the initial floors together with all corrections at
cuts strictly inside the children contribute
\[
    \sum_{i=1}^m n_\rho(B_i).
\]
The remaining cuts strictly inside \(B\) are precisely the \(m-1\) boundaries
between these children. By construction, exactly \(c_\rho(B)\) of those
boundaries are selected, so their contribution changes the preceding sum to
\[
    \sum_{i=1}^m n_\rho(B_i)+c_\rho(B)=n_\rho(B).
\]
This proves the induction claim.

The full sum \(\sum_{w\in B}q_\rho(w)\) contains one further possible
correction: \(\iota_\rho(a)\), attached to the cut immediately to the left of
the first level-zero block of \(B\). This gives
\eqref{eq:canonical-block-rounding}. Since
\(n_\rho(B)=\lfloor L_\rho(B)\rfloor\),
\eqref{eq:canonical-block-discrepancy} follows.
\end{proof}

We next decompose an arbitrary finite orbit interval into return blocks for
which Lemma~\ref{lem:canonical-block-rounding} applies.

\begin{lemma}[Logarithmic return-block decomposition]
\label{lem:canonical-decomposition}
Let \(z\in Z\) and let \(m<n\) be integers. Put
\[
    J=\{T^mz,\ldots,T^{n-1}z\},
    \qquad
    N:=n-m.
\]
Consider all return blocks \(B_k(a)\) that are contained in \(J\), and retain
those that are maximal under inclusion. These maximal return blocks form a
partition of \(J\). At each level \(k\) there are at most four of them, and
their total number is at most
\begin{equation}
\label{eq:number-maximal-canonical-blocks}
    4\bigl(1+\lfloor\log_2N\rfloor\bigr).
\end{equation}
\end{lemma}

\begin{proof}
Because every level-\((k+1)\) return block is a union of consecutive
level-\(k\) return blocks, any two return blocks, possibly from different
levels, are either disjoint or one contains the other. Every point of \(J\)
is itself a level-zero return block contained in \(J\). Moreover, a return
block contained in \(J\) can have level at most \(\lfloor\log_2N\rfloor\), by
\eqref{eq:canonical-block-size}. Hence, for each point of \(J\), there is a
largest return block containing that point and contained in \(J\). It follows
that the return blocks maximal under inclusion are pairwise disjoint and cover
\(J\).

Fix a level \(k\), and let \(B\) be a maximal level-\(k\) return block in
\(J\). Its unique level-\((k+1)\) parent is not contained in \(J\), for
otherwise \(B\) would not be maximal. Since the parent contains \(B\subset J\)
but is not contained in \(J\), it crosses at least one of the two boundary
cuts
\[
    (T^{m-1}z,T^mz),
    \qquad
    (T^{n-1}z,T^nz).
\]
For each boundary cut there is at most one level-\((k+1)\) return block that
crosses it, because the level-\((k+1)\) blocks partition the orbit. Such a
parent has at most three level-\(k\) children, and at least one of those
children is not contained in \(J\); hence at most two of its children can be
maximal level-\(k\) blocks contained in \(J\). With two boundary cuts, there
are therefore at most four maximal level-\(k\) return blocks.

Finally, every level-\(k\) return block contains at least \(2^k\) level-zero
blocks by \eqref{eq:canonical-block-size}. Therefore no level
\(k>\lfloor\log_2N\rfloor\) can occur, and summing the bound of four over the
possible levels gives \eqref{eq:number-maximal-canonical-blocks}.
\end{proof}

Combining the preceding two lemmas gives the required orbit discrepancy
estimate.

\begin{proposition}[Orbit discrepancy and exact mean]
\label{prop:one-dimensional-orbit-discrepancy}
For every \(z\in Z\) and integers \(m<n\),
\begin{equation}
\label{eq:one-dimensional-orbit-discrepancy}
    \left|
        \sum_{j=m}^{n-1}q_\rho(T^jz)
        -
        \rho\sum_{j=m}^{n-1}\tau(T^jz)
    \right|
    \leq
    4\bigl(1+\lfloor\log_2(n-m)\rfloor\bigr).
\end{equation}
Moreover,
\begin{equation}
\label{eq:one-dimensional-exact-mean}
    \int_Zq_\rho\,d\nu
    =
    \rho\int_Z\tau\,d\nu.
\end{equation}
\end{proposition}

\begin{proof}
Apply Lemma~\ref{lem:canonical-decomposition} to the orbit interval
\[
    J=\{T^mz,\ldots,T^{n-1}z\}.
\]
On each return block in that partition, the difference between the sum of the
\(q_\rho\)'s and the corresponding real mass has absolute value at most one
by \eqref{eq:canonical-block-discrepancy}. The number of blocks is bounded by
\eqref{eq:number-maximal-canonical-blocks}, which proves
\eqref{eq:one-dimensional-orbit-discrepancy}.

Taking \(m=0\) and \(n=N\), integrating
\eqref{eq:one-dimensional-orbit-discrepancy}, and using \(T\)-invariance of
\(\nu\), we obtain
\[
    N\left|
        \int_Z(q_\rho-\rho\tau)\,d\nu
    \right|
    \leq
    4\bigl(1+\lfloor\log_2N\rfloor\bigr).
\]
The function \(q_\rho\) is bounded because \(\tau\) is bounded. Dividing by
\(N\) and letting \(N\to\infty\) proves
\eqref{eq:one-dimensional-exact-mean}.
\end{proof}



\subsection{Delone placement, generation, and rigidity}
\label{subsec:one-dimensional-placement-generation}

\emph{Placement and local geometry.}
We realize the integer counts \(q_\rho(z)\) geometrically.
Assume
\begin{equation}
\label{eq:one-dimensional-intensity-threshold}
    \rho\tau_-\geq20
\end{equation}
and set \(h:=\rho^{-1}\). By construction,
\[
    \lfloor\rho\tau(z)\rfloor
    \leq q_\rho(z)
    \leq \lfloor\rho\tau(z)\rfloor+1,
\]
so in particular \(q_\rho(z)\geq20\).

Fix a Borel injection \(b:Z\to(0,1)\) and put
\(\varepsilon_0:=10^{-2}\). In the roof interval
\([0,\tau(z))\), place the following three-point marker, which
we call the \emph{marker}:
\[
    \operatorname{Mark}_\rho(z)
    :=
    h\left\{
        1,\,
        1+\varepsilon_0,\,
        1+(3+b(z))\varepsilon_0
    \right\}.
\]
The remaining \(q_\rho(z)-3\) points are placed at
\[
    5h+j\,d_\rho(z),
    \qquad
    j=0,1,\ldots,q_\rho(z)-4,
    \qquad
    d_\rho(z)
    :=
    \frac{\tau(z)-7h}{q_\rho(z)-4}.
\]
Denote the resulting \(q_\rho(z)\)-point subset of
\([0,\tau(z))\) by \(\mathcal A_\rho(z)\). For \(z\in Z\), define the
corresponding configuration on the whole orbit line by
\begin{equation}
\label{eq:one-dimensional-global-configuration}
    P_z
    :=
    \bigcup_{n\in\mathbb Z}
    \left(
        \tau_n(z)+\mathcal A_\rho(T^nz)
    \right).
\end{equation}

\begin{lemma}[One-dimensional Delone bounds]
\label{lem:one-dimensional-local-geometry}
For every \(z\in Z\),
\[
    \frac34h<d_\rho(z)<h,
\]
and
\[
    \operatorname{sep}(P_z)\geq\varepsilon_0h,
    \qquad
    \operatorname{covrad}(P_z)\leq2h.
\]
In particular, the family \((P_z)_{z\in Z}\) is uniformly Delone.
\end{lemma}

\begin{proof}
Write \(x:=\rho\tau(z)\geq20\). The pointwise bounds on
\(q_\rho\) give
\[
    x-5
    \leq q_\rho(z)-4
    \leq x-3,
\]
and therefore
\[
    \frac{x-7}{x-3}h
    \leq d_\rho(z)
    \leq
    \frac{x-7}{x-5}h.
\]
For \(x\geq20\), the left-hand side is greater than \(3h/4\)
and the right-hand side is less than \(h\).

The smallest distance inside a marker is
\(\varepsilon_0h\). The distance from the marker to the first remaining point is
greater than \(3h\), consecutive remaining points are separated by more than
\(3h/4\), and the last remaining point in the interval \([0,\tau(z))\) is
\(\tau(z)-2h\). The first point in the next roof interval is its marker point
at distance \(h\) from the next roof boundary, so the gap across two
successive roof intervals is \(3h\). These estimates apply to every translate
appearing in \eqref{eq:one-dimensional-global-configuration}. Hence
\(\operatorname{sep}(P_z)\geq\varepsilon_0h\), while every gap is smaller
than \(4h\), giving \(\operatorname{covrad}(P_z)\leq2h\).
\end{proof}

\emph{Global configuration and generation.}
The cocycle identity \eqref{eq:suspension-cocycle} gives
\begin{equation}
\label{eq:one-dimensional-base-covariance}
    P_{T^nz}=P_z-\tau_n(z),
    \qquad z\in Z,\quad n\in\mathbb Z.
\end{equation}
If \(x=s.z\), with \(0\leq s<\tau(z)\), set
\begin{equation}
\label{eq:one-dimensional-suspension-configuration}
    P_x:=P_z-s.
\end{equation}

\begin{proposition}[One-dimensional configuration map]
\label{prop:one-dimensional-configuration-map}
The map
\[
    \kappa_\rho:X\longrightarrow\mathcal N_s(\mathbb R),
    \qquad
    \kappa_\rho(x):=\delta_{P_x},
\]
is Borel and satisfies
\begin{equation}
\label{eq:one-dimensional-translation-covariance}
    P_{t.x}=P_x-t
    \qquad
    (t\in\mathbb R,\ x\in X).
\end{equation}
Moreover, \(\kappa_\rho\) is injective.
\end{proposition}

\begin{proof}
The maps \(z\mapsto q_\rho(z)\) and
\(z\mapsto\mathcal A_\rho(z)\) are Borel. Since
\(\tau\geq\tau_->0\), only uniformly finitely many terms in
\eqref{eq:one-dimensional-global-configuration} meet a fixed
compact interval. Hence \(z\mapsto\delta_{P_z}\), and therefore
\(x\mapsto\delta_{P_x}\), is Borel.

To verify covariance, write \(x=s.z\) and choose the unique
\(n\in\mathbb Z\) and
\(s'\in[0,\tau(T^nz))\) such that
\[
    s+t=\tau_n(z)+s'.
\]
Then \(t.x=s'.T^nz\), so
\[
    P_{t.x}
    =
    P_{T^nz}-s'
    =
    P_z-\tau_n(z)-s'
    =
    P_x-t.
\]

It remains to prove injectivity. Join two points of a configuration
by a short edge when their distance is less than \(h/10\).
Every pairwise distance inside a marker is less than
\(4\varepsilon_0h<h/10\), whereas every distance involving one of the remaining
points is greater than \(3h/4\). Thus the nontrivial connected
components of the short-edge graph are exactly the three-point
markers.

A marker can be decoded intrinsically. Its unique shortest edge
identifies its first two points; among their two endpoints, the
first marker point is the one farther from the third point. If
these points are denoted by \(u_1,u_2,u_3\) accordingly, then
\begin{equation}
\label{eq:one-dimensional-marker-decoding}
    b(z)
    =
    \frac{|u_3-u_1|}{\varepsilon_0h}-3,
    \qquad
    \text{left roof boundary}=u_1-h.
\end{equation}

Suppose \(z,w\in Z\), \(u\in\mathbb R\), and
\[
    P_w=P_z-u.
\]
The marker based at the zero roof boundary in \(P_w\) must
correspond to one of the markers in \(P_z-u\). Decoding its roof
boundary gives
\[
    u=\tau_n(z)
\]
for some \(n\in\mathbb Z\), while decoding its label gives
\(b(w)=b(T^nz)\). Since \(b\) is injective,
\[
    w=T^nz.
\]
Thus
\begin{equation}
\label{eq:one-dimensional-orbit-recovery}
    P_w=P_z-u
    \quad\Longrightarrow\quad
    u=\tau_n(z)
    \ \text{ and }\
    w=T^nz
\end{equation}
for some \(n\in\mathbb Z\).

Finally, if \(x=s.z\), \(y=t.w\), and \(P_x=P_y\), then
\(P_w=P_z-(s-t)\). By
\eqref{eq:one-dimensional-orbit-recovery},
\(w=T^nz\) and \(s-t=\tau_n(z)\) for some \(n\). Hence
\[
    y=t.w
    =(t+\tau_n(z)).z
    =s.z
    =x,
\]
which proves injectivity.
\end{proof}

Let
\[
    \eta_\rho:=(\kappa_\rho)_*\mu.
\]
The process \(\eta_\rho\) has the prescribed intensity. Since each
fundamental roof interval contains exactly \(q_\rho(z)\) points,
decomposing these local points into finitely many Borel graphs over
\(Z\) and applying the suspension formula gives
\[
    \rho_{\eta_\rho}
    =
    \frac{\displaystyle\int_Zq_\rho\,d\nu}
         {\displaystyle\int_Z\tau\,d\nu}
    =
    \rho.
\]
where the last equality is
\eqref{eq:one-dimensional-exact-mean}.

Set
\[
    Y_\rho:=\{x\in X:0\in P_x\}.
\]
By Proposition~\ref{prop:configuration-realization},
\(Y_\rho\) is a generating Delone translation cross-section,
its return set at \(x\) is \(P_x\), and its return-time process is
\(\eta_\rho\). In particular, \(Y_\rho\) has intensity \(\rho\).

\emph{Maximal rigidity.}
\begin{proposition}[Maximal rigidity from three-point markers]
\label{prop:one-dimensional-maximal-rigidity}
The point process \(\eta_\rho\) is maximally rigid. More precisely,
for every bounded Borel set \(A\subset\mathbb R\), there is a
Borel recovery map from configurations on \(A^c\) to
\(\mathcal N_s(\mathbb R)\) which recovers
\(\delta_{P_x}\) from \(P_x\cap A^c\) for every \(x\) in the
invariant conull set on which the suspension construction is defined.
\end{proposition}

\begin{proof}
Fix a bounded Borel set \(A\subset\mathbb R\), put
\[
    \theta:=\frac h{10},
    \qquad
    E_A
    :=
    \{u\in\mathbb R:\operatorname{dist}(u,A)>\theta\},
\]
and let \(x=s.z\), with \(0\leq s<\tau(z)\).

For \(n\in\mathbb Z\), the marker belonging to the roof
\(T^nz\) appears in \(P_x\) at
\[
    \tau_n(z)-s
    +
    h\left\{
        1,\,
        1+\varepsilon_0,\,
        1+(3+b(T^nz))\varepsilon_0
    \right\}.
\]
Since \(\tau\geq\tau_->0\),
\[
    \tau_n(z)-s\longrightarrow+\infty
    \quad(n\to+\infty),
    \qquad
    \tau_n(z)-s\longrightarrow-\infty
    \quad(n\to-\infty).
\]
Moreover, every marker has diameter less than
\(4\varepsilon_0h<\theta\). Hence infinitely many complete markers
are contained in \(E_A\).

From the outside configuration \(P_x\cap A^c\) we can recover
\(P_x\cap E_A\). Form on this set the graph whose edges join pairs
at distance less than \(\theta\). Every three-point nontrivial
component of this graph is a complete marker. Indeed, in the full
configuration the nontrivial components of the graph just defined
are exactly the markers; restricting to \(E_A\) can delete or split
a marker, but cannot create a new edge. At least one complete
three-point component exists by the preceding paragraph.

More explicitly, the relation consisting of pairs \((\omega,M)\), where
\(\omega\) is an outside configuration and \(M\) is a complete
three-point marker component of \(\omega|_{E_A}\), is Borel and has
countable sections. By the preceding paragraph, its projection contains every
outside configuration arising from the process. Lusin--Novikov therefore
shows that this projection is Borel and provides a Borel choice of one such
component there; extend the choice arbitrarily off the projection.
Decode
the chosen marker as in
\eqref{eq:one-dimensional-marker-decoding}, and write its
intrinsically ordered points as \(u_1,u_2,u_3\). Put
\[
    a:=u_1-h,
    \qquad
    \beta
    :=
    \frac{|u_3-u_1|}{\varepsilon_0h}-3.
\]
If the chosen marker belongs to the roof \(w=T^nz\), then
\[
    a=\tau_n(z)-s,
    \qquad
    \beta=b(w).
\]
Since \(b\) is a Borel injection between standard Borel spaces,
Lusin--Souslin gives a Borel inverse on \(b(Z)\), so the marker
determines \(w=b^{-1}(\beta)\). It then determines the entire
suspension state:
\[
    x=(-a).w.
\]
Finally, the Borel map \(\kappa_\rho\) recovers
\(\delta_{P_x}\). Thus the restriction to \(A^c\) determines the
whole configuration by a Borel map, as claimed.
\end{proof}


\subsection{Euclidean discrepancy and spectral consequences}
\label{subsec:one-dimensional-spectral-consequences}

We first extend
Proposition~\ref{prop:one-dimensional-orbit-discrepancy} from
unions of suspension intervals to arbitrary intervals in
\(\mathbb R\).

\begin{proposition}[Logarithmic growth of discrepancy]
\label{prop:one-dimensional-Euclidean-discrepancy}
There is a constant \(C_\rho<\infty\) such that, for every
\(x\in X\) and every bounded interval \(I\subset\mathbb R\),
\begin{equation}
\label{eq:one-dimensional-Euclidean-discrepancy}
    \left|
        \#(P_x\cap I)-\rho|I|
    \right|
    \leq
    C_\rho\log(2+|I|).
\end{equation}
\end{proposition}

\begin{proof}
By translation covariance it is enough to consider \(x=z\in Z\).
The interval \(I\) differs from a union
\[
    [\tau_m(z),\tau_n(z))
\]
of consecutive complete suspension intervals by at most two
partial suspension intervals. Each partial interval has length at
most \(\tau_+\) and contains at most
\(\lfloor\rho\tau_+\rfloor+1\) points. Its contribution to the
absolute discrepancy is therefore at most
\[
    \lfloor\rho\tau_+\rfloor+1+\rho\tau_+.
\]

If \(N:=n-m\) is the number of complete suspension intervals, then
\[
    N\leq\frac{|I|}{\tau_-}+2.
\]
Hence Proposition~\ref{prop:one-dimensional-orbit-discrepancy}
gives
\[
    \left|
        \#(P_z\cap I)-\rho|I|
    \right|
    \leq
    4\bigl(1+\lfloor\log_2\max\{1,N\}\rfloor\bigr)
    +
    2\bigl(\lfloor\rho\tau_+\rfloor+1+\rho\tau_+\bigr),
\]
which implies \eqref{eq:one-dimensional-Euclidean-discrepancy}
after enlarging the constant.
\end{proof}

The logarithmic bound above cannot in general be replaced by a uniform
bound: the next proposition shows that such a bound forces a nonzero Koopman
eigenvalue.

\begin{proposition}[Bounded discrepancy forces an eigenvalue]
\label{prop:bounded-discrepancy-eigenvalue}
Let \(\mathbb R\curvearrowright(X,\mu)\) be an ergodic p.m.p.\ Borel
action and let \(x\mapsto P_x\) be a Borel equivariant family of
locally finite configurations,
\[
    P_{t.x}=P_x-t
    \qquad(t\in\mathbb R,\ x\in X).
\]
Suppose that, for some \(\rho>0\), there are an invariant conull Borel
set \(X_0\subset X\) and \(C<\infty\) such that
\[
    \left|\#(P_x\cap I)-\rho|I|\right|\leq C
\]
for every \(x\in X_0\) and every bounded interval \(I\subset\mathbb R\).
Then \(\rho\) is a nonzero eigenvalue of the flow. In particular, a
weakly mixing flow admits no such realization.
\end{proposition}

\begin{proof}
For \(x\in X_0\), define the signed counting cocycle
\[
    n(t,x):=
    \begin{cases}
        \#(P_x\cap[0,t)),&t\geq0,\\
        -\#(P_x\cap[t,0)),&t<0,
    \end{cases}
    \qquad
    c(t,x):=n(t,x)-\rho t.
\]
Equivariance gives
\[
    c(s+t,x)=c(t,x)+c(s,t.x)
    \qquad(s,t\in\mathbb R,\ x\in X_0),
\]
and the discrepancy assumption gives \(|c(t,x)|\leq C\) for all
\(t\in\mathbb R\) and \(x\in X_0\). The cocycle is therefore bounded. We record the short
coboundary argument directly. For \(n\geq1\), put
\[
    \psi_n(x)
    :=
    \operatorname*{ess\,sup}_{|t|\leq n}c(t,x),
    \qquad
    \psi(x):=\sup_{n\geq1}\psi_n(x).
\]
Since \((t,x)\mapsto c(t,x)\) is Borel, each \(\psi_n\) is
measurable: for every \(a\in\mathbb R\), the set
\(\{x:\psi_n(x)>a\}\) is the set of \(x\) for which
\(\{t\in[-n,n]:c(t,x)>a\}\) has positive Lebesgue measure, hence is
measurable by Fubini. Thus \(\psi\) is measurable and, since
\(|c|\leq C\), bounded. Translation invariance of Lebesgue measure
together with the cocycle identity gives
\[
    \psi(s.x)=\psi(x)-c(s,x)
    \qquad(s\in\mathbb R,\ x\in X_0).
\]
Thus \(c(s,x)=\psi(x)-\psi(s.x)\). Since \(n(s,x)\in\mathbb Z\), the
function
\[
    F(x):=e^{-2\pi i\psi(x)}
\]
satisfies, for \(s\in\mathbb R\) and \(x\in X_0\),
\[
    F(s.x)=e^{-2\pi i\rho s}F(x).
\]
Equivalently,
\[
    F((-s).x)=e^{2\pi i\rho s}F(x),
\]
so \(F\) is a nonzero eigenfunction of frequency \(\rho\).
\end{proof}

\begin{corollary}[Logarithmic number variance]
\label{cor:one-dimensional-number-variance}
As \(R\to\infty\),
\begin{equation}
\label{eq:one-dimensional-number-variance}
    \operatorname{Var}_{\eta_\rho}(N_{[0,R)})
    =
    O(\log^2 R).
\end{equation}
\end{corollary}

\begin{proof}
The process \(\eta_\rho\) has intensity \(\rho\), so
Proposition~\ref{prop:one-dimensional-Euclidean-discrepancy}
applied to \([0,R)\) gives
\[
    \left|
        \omega([0,R))-\rho R
    \right|
    \leq
    C_\rho\log(2+R)
\]
for \(\eta_\rho\)-almost every \(\omega\). Squaring and integrating
proves \eqref{eq:one-dimensional-number-variance}.
\end{proof}

We finally estimate the Bartlett spectrum
\(\sigma_{\eta_\rho}\). For \(f_R:=\mathbf 1_{[0,R)}\), a direct
calculation gives
\[
    \left|\widehat f_R(\xi)\right|
    =
    \left|
        \frac{\sin(\pi R\xi)}{\pi\xi}
    \right|,
\]
with the value \(R\) at \(\xi=0\).

\begin{proposition}[Near-quadratic Bartlett suppression]
\label{prop:one-dimensional-Bartlett-decay}
As \(\varepsilon\downarrow0\),
\begin{equation}
\label{eq:one-dimensional-Bartlett-decay}
    \sigma_{\eta_\rho}([-\varepsilon,\varepsilon])
    =
    O\!\left(
        \varepsilon^2
        \log^2\!\frac{e}{\varepsilon}
    \right).
\end{equation}
In particular, \(\eta_\rho\) is hyperuniform.
\end{proposition}

\begin{proof}
For \(0<\varepsilon\leq1/2\), set
\[
    R:=\frac{1}{2\varepsilon}.
\]
If \(|\xi|\leq\varepsilon\), then
\(|\pi R\xi|\leq\pi/2\). Since
\(\sin t\geq2t/\pi\) for \(0\leq t\leq\pi/2\),
\[
    |\widehat f_R(\xi)|
    \geq
    \frac{2R}{\pi}.
\]
Lemma~\ref{lem:Bartlett-indicators} and the preceding bound give
\[
\begin{aligned}
    \operatorname{Var}_{\eta_\rho}(\bS f_R)
    &=
    \int_{\mathbb R}
        |\widehat f_R(\xi)|^2
        \,d\sigma_{\eta_\rho}(\xi) \\
    &\geq
    \frac{4R^2}{\pi^2}
    \sigma_{\eta_\rho}([-\varepsilon,\varepsilon]).
\end{aligned}
\]
By Corollary~\ref{cor:one-dimensional-number-variance},
\[
    \operatorname{Var}_{\eta_\rho}(\bS f_R)
    =
    O(\log^2 R).
\]
Since \(R=(2\varepsilon)^{-1}\), this proves
\eqref{eq:one-dimensional-Bartlett-decay}. The right-hand side is
\(o(\varepsilon)\), so hyperuniformity follows from
Subsection~\ref{subsec:linear-statistics-Bartlett}.
\end{proof}


\section{Orbit geometry and integerization}
\label{sec:orbit-geometry-integerization}

Throughout this section, let \(d\geq2\) and let
\(G_K\curvearrowright(X,\mu)\) be an essentially free p.m.p.\
Borel action whose restricted \(V\)-action is ergodic. After
restricting to a \(G_K\)-invariant conull Borel set, we assume that
the action is free. The cross-section \(C\) introduced below is an auxiliary
cross-section for the full \(G_K\)-action; it should not be confused with the
\(K\)-invariant translation cross-section produced from the final point
configuration.

All geometric constants introduced in this section are fixed independently of
\(c\) and \(\rho\). The threshold \(\rho_0\) may be increased finitely many
times after those constants are fixed; every such increase is uniform in
\(c\).

The goal of this section is to construct the bounded sets that will be
discretized in Section~\ref{sec:local-filling-assembly}. For every sufficiently
large \(\rho\) we shall define a Borel family
\[
    \Pi_\rho(c)\subset V,
    \qquad c\in C,
\]
with a constant \(r_*>0\), independent of \(c\) and \(\rho\), such that
\[
    \overline B_{r_*}\subset\Pi_\rho(c),
    \qquad
    \rho\lambda_V(\Pi_\rho(c))\in\mathbb N,
\]
and with all \(\Pi_\rho(c)\) contained in one fixed ball.
For each \(x\in X\), the copies of these sets placed at the returns of \(x\)
will form an exact partition of \(V\). The placement is compatible with the
\(G_K\)-action; the precise transformation law is
\eqref{eq:physical-packet-covariance} below.

The integer-mass condition is needed in Section~\ref{sec:local-filling-assembly},
where \(\Pi_\rho(c)\) is replaced by exactly
\(\rho\lambda_V(\Pi_\rho(c))\) points. Related constructions based on
invariant partitions appear in \cite{KLLY25,LK26}. Fair partitions have equal
cell volumes, so a fixed number of points can be assigned to every cell;
Lotz--Klatt \cite[Theorem~5.2 and Remark~5.3]{LK26} also treat
moment-preserving rearrangements for more general invariant partitions. Here
the initial Voronoi cells are determined by the prescribed action and need not
have equal volume. We change each cell by a set of \(\rho\)-mass bounded
independently of \(\rho\) so that its resulting \(\rho\)-mass is an integer.
The construction has four parts. Subsection~\ref{subsec:framed-return-geometry}
produces the labelled Delone scaffold and the placement maps.
Subsection~\ref{subsec:proximity-forest} constructs a bounded-range parent map
with finite descendant sets. Subsection~\ref{subsec:labelled-Voronoi-cells}
turns the closed Voronoi cells into an exact Borel partition. Finally,
Subsection~\ref{subsec:integerization-packet-transfer} modifies those cells
along the parent map so that their \(\rho\)-masses are integers.


\subsection{Framed return geometry}
\label{subsec:framed-return-geometry}

Fix \(r_{\rm sep}>0\) and put
\[
    L_{r_{\rm sep}}
    :=
    \{(u,k)\in G_K:\|u\|\leq r_{\rm sep}\}.
\]
Since \(K\) is compact, \(L_{r_{\rm sep}}\) is compact. By standard
cocompact cross-section theory, choose a Borel cross-section
\(C\subset X\) for the full \(G_K\)-action such that
\begin{equation}
\label{eq:full-cross-section-separation}
    L_{r_{\rm sep}}.c_0\cap L_{r_{\rm sep}}.c_1=\varnothing
    \qquad(c_0\neq c_1\in C);
\end{equation}
see, for instance, \cite[Theorem~A.3]{CGMTD26}. Choose a compact set
\(\Omega\subset G_K\) with \(\Omega.C=X\), and set
\begin{equation}
\label{eq:Rcov-definition}
    R_{\rm cov}
    :=
    \sup\{\|a\|:(a,k)\in\Omega\}<\infty.
\end{equation}

For \(x\in X\), let
\[
    C_x:=\{g\in G_K:g.x\in C\}.
\]
A return \(g=(a,k)\in C_x\) determines three pieces of data:
\begin{equation}
\label{eq:framed-return-coordinates}
\begin{aligned}
    \operatorname{pos}_x(g)&:=k^{-1}a\in V,\\
    \operatorname{fr}_x(g)&:=k^{-1}\in K,\\
    \operatorname{lab}_x(g)&:=g.x\in C.
\end{aligned}
\end{equation}
We call these, respectively, the \emph{position}, \emph{frame}, and
\emph{label} of the return. The position is the image of the coset \(Kg\)
under the identification \(K\backslash G_K\cong V\) from
Subsection~\ref{subsec:Euclidean-Fourier-conventions}. The frame records how a
local configuration written in the coordinates of its label is rotated when
placed in the orbit of \(x\), while the label specifies the cross-section
point from which that local configuration is taken.

The set of positions
\begin{equation}
\label{eq:framed-scaffold}
    \mathscr P_x
    :=
    \{\operatorname{pos}_x(g):g\in C_x\}
\end{equation}
is the \emph{scaffold} associated with \(x\). If \(c\in C\), then the
identity return \((0,e)\in C_c\) gives the distinguished site
\(0\in\mathscr P_c\), with frame \(e\) and label \(c\).

It will be convenient to package the position and frame into the affine
isometry
\begin{equation}
\label{eq:framed-placement-map}
    \Phi_{x,g}:V\longrightarrow V,
    \qquad
    \Phi_{x,g}(y)
    :=
    \operatorname{pos}_x(g)+\operatorname{fr}_x(g)y.
\end{equation}
Thus a set constructed in the local coordinates attached to
\(\operatorname{lab}_x(g)\) is placed in physical coordinates by
\(\Phi_{x,g}\).

\begin{proposition}[Framed Delone scaffold]
\label{prop:framed-Delone-scaffold}
For every \(x\in X\), the map
\(g\mapsto\operatorname{pos}_x(g)\) is injective on \(C_x\), and
\(\mathscr P_x\) is \((r_{\rm sep},R_{\rm cov})\)-Delone. The relation
\[
    \mathscr S
    :=
    \{(x,u)\in X\times V:u\in\mathscr P_x\}
\]
is Borel, and the unique return \(g\in C_x\) associated with
\((x,u)\in\mathscr S\) is a Borel function of \((x,u)\). Consequently, its
frame in \(K\) and label in \(C\) are Borel functions of \((x,u)\), and
\(x\mapsto\delta_{\mathscr P_x}\) is Borel.

If \(h=(v,\ell)\in G_K\), then
\begin{equation}
\label{eq:framed-return-covariance}
    C_{h.x}=C_xh^{-1},
\end{equation}
and, for \(g\in C_x\),
\begin{equation}
\label{eq:framed-coordinate-covariance}
\begin{aligned}
    \operatorname{pos}_{h.x}(gh^{-1})
    &=\ell\operatorname{pos}_x(g)-v,\\
    \operatorname{fr}_{h.x}(gh^{-1})
    &=\ell\operatorname{fr}_x(g),\\
    \operatorname{lab}_{h.x}(gh^{-1})
    &=\operatorname{lab}_x(g).
\end{aligned}
\end{equation}
Equivalently,
\begin{equation}
\label{eq:framed-transport-identity}
    \Phi_{h.x,gh^{-1}}(y)
    =
    \ell\Phi_{x,g}(y)-v,
    \qquad y\in V.
\end{equation}
In particular,
\[
    \mathscr P_{(v,\ell).x}=\ell\mathscr P_x-v.
\]
\end{proposition}

\begin{proof}
Suppose that \(g_i=(a_i,k_i)\in C_x\) have the same position \(u\). Then
\(a_i=k_i u\), and hence
\[
    g_2g_1^{-1}=(0,k_2k_1^{-1})\in L_{r_{\rm sep}}.
\]
Writing \(c_i:=g_i.x\in C\), we have
\(c_2=(g_2g_1^{-1}).c_1\). If \(c_1\neq c_2\), then
\(c_2\in L_{r_{\rm sep}}.c_1\cap L_{r_{\rm sep}}.c_2\), contrary to
\eqref{eq:full-cross-section-separation}. Thus \(c_1=c_2\), and freeness gives
\(g_1=g_2\). In particular, each scaffold site has a unique frame and label.

Let \(u_1,u_2\in\mathscr P_x\) correspond to distinct returns
\(g_i=(k_i u_i,k_i)\). Then
\[
    g_2g_1^{-1}
    =
    \bigl(k_2(u_2-u_1),k_2k_1^{-1}\bigr).
\]
If \(\|u_2-u_1\|\leq r_{\rm sep}\), the preceding argument again
contradicts \eqref{eq:full-cross-section-separation}. Hence
\(\operatorname{sep}(\mathscr P_x)\geq r_{\rm sep}\).

For the covering bound, fix \(w\in V\). Choose \(f=(a,k)\in\Omega\) and
\(c\in C\) such that
\[
    (w,e).x=f.c.
\]
Then \(g:=f^{-1}(w,e)\) belongs to \(C_x\), and
\[
    g=(k^{-1}(w-a),k^{-1}),
    \qquad
    \operatorname{pos}_x(g)=w-a.
\]
Therefore
\[
    \operatorname{dist}(w,\mathscr P_x)
    \leq\|a\|\leq R_{\rm cov},
\]
which proves the Delone assertion.

Let \(h=(v,\ell)\). Since \(g\in C_{h.x}\) if and only if \(gh\in C_x\),
\eqref{eq:framed-return-covariance} follows. If \(g=(a,k)\in C_x\), then
\[
    gh^{-1}
    =
    (a-k\ell^{-1}v,k\ell^{-1}),
\]
and \eqref{eq:framed-coordinate-covariance} follows directly from
\eqref{eq:framed-return-coordinates}; equation
\eqref{eq:framed-transport-identity} is the same identity written in terms of
\(\Phi\).

Finally, the incidence relation
\[
    \mathscr C
    :=
    \{(x,g)\in X\times G_K:g.x\in C\}
\]
is Borel and has countable sections. The map
\[
    \mathscr C\longrightarrow X\times V,
    \qquad
    (x,g)\longmapsto(x,\operatorname{pos}_x(g)),
\]
is Borel and injective. By the Lusin--Souslin theorem, its image
\(\mathscr S\) is Borel and its inverse is Borel. Composing this inverse with
\(\operatorname{fr}_x\) and \(\operatorname{lab}_x\) gives the asserted
Borel frame and label maps. The Borelness of
\(x\mapsto\delta_{\mathscr P_x}\) then follows from Lusin--Novikov
enumeration and local finiteness.
\end{proof}

For later use, if \(g\in C_x\) has position \(u\), frame \(r\), and label
\(c\), then
\begin{equation}
\label{eq:framed-recovery}
    x=(-u,r).c.
\end{equation}
Thus one labelled and framed scaffold site determines the original state.

The p.m.p.\ action induces the usual invariant transverse measure on the
cross-section \(C\). Since \(G_K\) is unimodular, we normalize this measure to
a probability measure and denote it by \(\nu_C\). For background on
transverse measures and cross-section formulas, see, for instance,
\cite[Sections~3 and~4]{BHK25}; the formulation used here is also recorded in
\cite[Proposition~A.8]{CGMTD26}.


\subsection{Proximity graph and one-ended forest}
\label{subsec:proximity-forest}

The formulas in Subsection~\ref{subsec:integerization-packet-transfer} use a
map \(\operatorname{par}:C\to C\) with three properties: the distance
\(d_C(c,\operatorname{par}(c))\) is uniformly bounded, every vertex has
uniformly finitely many children, and every descendant set is finite. We now
construct such a map from a one-ended spanning forest.

Let \(\mathcal R_C\) be the orbit equivalence relation on \(C\). For
\((c,c')\in\mathcal R_C\), freeness gives a unique
\(\gamma(c,c')=(a,k)\in G_K\) such that
\[
    c'=\gamma(c,c').c.
\]
The cocycle \(\gamma\) is Borel. Indeed, on the Borel set
\[
    \{(c,g)\in C\times G_K:g.c\in C\}
\]
the map \((c,g)\mapsto(c,g.c)\) is injective, with image
\(\mathcal R_C\), and the assertion follows from Lusin--Souslin.

Define
\begin{equation}
\label{eq:cross-section-metric}
    d_C(c,c'):=\|a\|,
    \qquad \gamma(c,c')=(a,k).
\end{equation}
If \(c_1,c_2\) belong to the same \(\mathcal R_C\)-class and
\(g_i=(a_i,k_i)\in C_c\) satisfy \(g_i.c=c_i\), write
\(u_i:=\operatorname{pos}_c(g_i)\). Then
\[
    \gamma(c_1,c_2)
    =
    \bigl(k_2(u_2-u_1),k_2k_1^{-1}\bigr),
\]
and therefore
\begin{equation}
\label{eq:cross-section-metric-scaffold}
    d_C(c_1,c_2)=\|u_2-u_1\|.
\end{equation}
Thus \(d_C(c_1,c_2)=\|u_2-u_1\|\); in particular, \(d_C\) is a Borel
metric on each \(\mathcal R_C\)-class.

Set
\begin{equation}
\label{eq:Rips-radius}
    D_{\rm Rips}:=3R_{\rm cov}
\end{equation}
and define a Borel graph \(\mathcal G\) on \(C\) by
\begin{equation}
\label{eq:proximity-graph}
    c\sim_{\mathcal G}c'
    \quad\Longleftrightarrow\quad
    (c,c')\in\mathcal R_C,
    \quad 0<d_C(c,c')\leq D_{\rm Rips}.
\end{equation}
Under the identification of an \(\mathcal R_C\)-class with its scaffold,
\(\mathcal G\) joins precisely the pairs of sites at distance at most
\(D_{\rm Rips}\).

\begin{proposition}[Proximity graph]
\label{prop:proximity-graph}
The graph \(\mathcal G\) is locally finite, p.m.p., and aperiodic, and
\[
    \mathcal R_{\mathcal G}=\mathcal R_C.
\]
Moreover,
\begin{equation}
\label{eq:proximity-degree-bound}
    \deg_{\mathcal G}(c)
    \leq
    \left(1+\frac{2D_{\rm Rips}}{r_{\rm sep}}\right)^d-1
    \qquad(c\in C),
\end{equation}
and every connected component of \(\mathcal G\) is one-ended.
\end{proposition}

\begin{proof}
Fix \(c\in C\). The map
\[
    c'\longmapsto
    \operatorname{pos}_c(\gamma(c,c'))
\]
identifies the \(\mathcal R_C\)-class of \(c\) with \(\mathscr P_c\), and
\eqref{eq:cross-section-metric-scaffold} identifies \(d_C\) with Euclidean
distance. Hence the neighbors of \(c\) correspond to the points of
\(\mathscr P_c\cap B_{D_{\rm Rips}}(0)\) other than the origin. The packing
bound \eqref{eq:Delone-packing-bound} gives
\eqref{eq:proximity-degree-bound}.

To prove connectedness, take two sites \(u,v\in\mathscr P_c\). Subdivide the
segment from \(u\) to \(v\) into pieces of length at most
\(R_{\rm cov}\). At each intermediate subdivision point choose a site of
\(\mathscr P_c\) within distance \(R_{\rm cov}\), and use \(u\) and \(v\)
as the endpoints. Consecutive chosen sites are at distance at most
\(3R_{\rm cov}=D_{\rm Rips}\), hence are adjacent in \(\mathcal G\). Thus
\(\mathcal G\) is connected on every \(\mathcal R_C\)-class, so
\(\mathcal R_{\mathcal G}=\mathcal R_C\).

We next prove one-endedness. Fix a finite set \(F\) in one component and
identify that component with \(\mathscr P_c\). Choose \(A<\infty\) so that
the sites corresponding to \(F\) lie in \(B_A(0)\). If
\(u,v\in\mathscr P_c\) lie outside \(B_{A+3R_{\rm cov}}(0)\), then, because
\(d\geq2\), they can be joined by a path in
\[
    V\setminus\overline B_{A+2R_{\rm cov}}(0).
\]
Subdivide this path into pieces of length at most \(R_{\rm cov}\) and
approximate the intermediate points by scaffold sites as above. All chosen
sites lie outside \(B_A(0)\), and consecutive sites are adjacent. Hence all
vertices outside \(B_{A+3R_{\rm cov}}(0)\) belong to a single component of
\(\mathcal G\setminus F\). Since bounded sets contain only finitely many
scaffold sites, every infinite component of \(\mathcal G\setminus F\) meets
this exterior component. Thus \(\mathcal G\setminus F\) has exactly one
infinite component.

The scaffold \(\mathscr P_c\) is relatively dense in the unbounded space
\(V\), so every component of \(\mathcal G\) is infinite; hence the graph is
aperiodic. Finally, the transverse measure \(\nu_C\) is invariant under the
full relation \(\mathcal R_C\). Since
\(\mathcal R_{\mathcal G}=\mathcal R_C\), the graph is p.m.p. in the sense of
Subsection~\ref{subsec:measured-graphings-forests}.
\end{proof}

The proof of one-endedness is the only place here where \(d\geq2\) is used:
the complement of a closed ball in \(V\) is path connected, so the
approximation argument can be carried out without meeting the prescribed
finite set.

Proposition~\ref{prop:proximity-graph} verifies that \(\mathcal G\) is locally
finite, p.m.p., aperiodic, and one-ended componentwise; in particular, it is
\(\nu_C\)-nowhere two-ended. Thus all hypotheses of
Theorem~\ref{thm:measured-one-ended-forest} hold, and \(\mathcal G\) has a Borel
\(\nu_C\)-almost-everywhere one-ended spanning subforest
\(\mathcal T\subseteq\mathcal G\). We remove the
\(\mathcal R_{\mathcal G}\)-saturation of the exceptional null set as in
Subsection~\ref{subsec:measured-graphings-forests}. The corresponding
\(G_K\)-saturation in \(X\) is conull by the cross-section measure formula;
see \cite[Sections~3 and~4]{BHK25} or
\cite[Proposition~A.8]{CGMTD26}. After this restriction, every component of
\(\mathcal T\) is one-ended.

Lemma~\ref{lem:one-ended-forest-orientation} now gives a Borel map
\[
    \operatorname{par}:C\longrightarrow C.
\]
Define
\[
    \operatorname{Ch}(c)
    :=
    \{w\in C:\operatorname{par}(w)=c\},
    \qquad
    \operatorname{Desc}(c)
    :=
    \{w\in C:\operatorname{par}^{\,n}(w)=c
        \text{ for some }n\geq0\}.
\]
The child and descendant relations are Borel, every descendant set is finite,
and \(d_C(c,\operatorname{par}(c))\leq D_{\rm Rips}\) for every \(c\). Set
\begin{equation}
\label{eq:forest-child-bound}
    b_{\mathcal T}
    :=
    \left\lceil
        \left(1+\frac{2D_{\rm Rips}}{r_{\rm sep}}\right)^d
    \right\rceil,
\end{equation}
so that
\[
    \#\operatorname{Ch}(c)< b_{\mathcal T}
    \qquad(c\in C).
\]
Thus the parent map has the three properties stated at the beginning of this
subsection.


\subsection{Labelled Voronoi cells}
\label{subsec:labelled-Voronoi-cells}

For \(c\in C\), let
\[
    \operatorname{Vor}(c)
    :=
    \operatorname{Vor}(\mathscr P_c,0)
\]
be the closed Voronoi cell of the distinguished site \(0\in\mathscr P_c\).
Set
\begin{equation}
\label{eq:Voronoi-radii}
    r_{\rm Vor}:=\frac{r_{\rm sep}}2,
    \qquad
    R_{\rm Vor}:=R_{\rm cov}.
\end{equation}
By Lemma~\ref{lem:Delone-Voronoi-geometry},
\[
    \overline B_{r_{\rm Vor}}
    \subset
    \operatorname{Vor}(c)
    \subset
    \overline B_{R_{\rm Vor}},
\]
and \(\operatorname{Vor}(c)\) is a compact convex polytope with
\(\lambda_V\)-null boundary. The parameterized Voronoi discussion in
Subsection~\ref{subsec:Borel-geometric-preliminaries} shows that
\(c\mapsto\operatorname{Vor}(c)\) is Borel as a compact-valued map.

The closed Voronoi cells overlap on their boundaries. To obtain disjoint
cells, fix a Borel linear order \(\preccurlyeq\) on \(C\), as in
Subsection~\ref{subsec:Borel-geometric-preliminaries}, and use the labels to
break ties between nearest scaffold sites. For \(c\in C\) and \(y\in V\), let
\[
    \mathcal M_c(y)
    :=
    \left\{
        g\in C_c:
        \|y-\operatorname{pos}_c(g)\|
        =
        \operatorname{dist}(y,\mathscr P_c)
    \right\}.
\]
This set is nonempty and finite, and its elements have distinct labels. Among
its elements choose the one with least label. Define \(W_c\) to be the set of
\(y\) for which this choice is the identity return \((0,e)\in C_c\).
Equivalently,
\begin{equation}
\label{eq:labelled-cell-definition}
    W_c
    :=
    \left\{
        y\in V:
        (\|y\|,c)
        \leq_{\rm lex}
        \bigl(
            \|y-\operatorname{pos}_c(g)\|,
            \operatorname{lab}_c(g)
        \bigr)
        \text{ for every }g\in C_c
    \right\},
\end{equation}
where the first coordinate is ordered in the usual way and the second by
\(\preccurlyeq\). Thus \(W_c\subset\operatorname{Vor}(c)\), and the difference can occur only
on \(\partial\operatorname{Vor}(c)\).

\begin{proposition}[Labelled Voronoi partition]
\label{prop:labelled-Voronoi-partition}
The set
\[
    \mathscr W:=\{(c,y)\in C\times V:y\in W_c\}
\]
is Borel. For every \(c\in C\),
\begin{equation}
\label{eq:labelled-cell-geometry}
    B_{r_{\rm Vor}}
    \subset
    W_c
    \subset
    \operatorname{Vor}(c)
    \subset
    \overline B_{R_{\rm Vor}},
\end{equation}
and
\[
    \lambda_V\bigl(\operatorname{Vor}(c)\setminus W_c\bigr)=0,
    \qquad
    \lambda_V(W_c)=\lambda_V(\operatorname{Vor}(c)).
\]
In particular,
\[
    0<\lambda_V(B_{r_{\rm Vor}})
    \leq \lambda_V(W_c)
    \leq \lambda_V(B_{R_{\rm Vor}})<\infty.
\]

For every \(x\in X\), the sets
\begin{equation}
\label{eq:physical-cell-partition}
    \left\{
        \Phi_{x,g}
        \bigl(W_{\operatorname{lab}_x(g)}\bigr):
        g\in C_x
    \right\}
\end{equation}
form an exact Borel partition of \(V\). If \(h=(v,\ell)\in G_K\), then
\begin{equation}
\label{eq:physical-cell-covariance}
    \Phi_{h.x,gh^{-1}}
        \bigl(W_{\operatorname{lab}_{h.x}(gh^{-1})}\bigr)
    =
    \ell\,
    \Phi_{x,g}
        \bigl(W_{\operatorname{lab}_x(g)}\bigr)-v.
\end{equation}
Finally, for every polynomial \(P\) on \(V\), the maps
\[
    c\longmapsto\lambda_V(W_c),
    \qquad
    c\longmapsto\int_{W_c}P(y)\,d\lambda_V(y)
\]
are Borel.
\end{proposition}

\begin{proof}
The lexicographic description \eqref{eq:labelled-cell-definition} and a Borel
Lusin--Novikov enumeration of \(C_c\) show that \(\mathscr W\) is Borel: the
defining condition is a countable intersection of Borel inequalities.

If \(y\in W_c\), then the origin realizes the minimum distance from \(y\) to
\(\mathscr P_c\), so \(y\in\operatorname{Vor}(c)\). Conversely, a point of
\(\operatorname{Vor}(c)\) can fail to belong to \(W_c\) only when it is
equidistant from the origin and another scaffold site. Hence
\[
    \operatorname{Vor}(c)\setminus W_c
    \subset
    \partial\operatorname{Vor}(c),
\]
which proves the null-boundary assertion. If \(y\in B_{r_{\rm Vor}}\) and
\(u\in\mathscr P_c\setminus\{0\}\), then \(\|u\|\geq r_{\rm sep}\), and
therefore
\[
    \|y-u\|
    \geq \|u\|-\|y\|
    > r_{\rm sep}-r_{\rm Vor}
    =r_{\rm Vor}
    >\|y\|.
\]
Thus the origin is the unique closest site, and
\eqref{eq:labelled-cell-geometry} follows.

To identify the cells in the orbit of an arbitrary \(x\), fix
\(g=(a,k)\in C_x\) and put
\[
    p:=\operatorname{pos}_x(g)=k^{-1}a,
    \qquad
    r:=\operatorname{fr}_x(g)=k^{-1},
    \qquad
    c:=\operatorname{lab}_x(g)=g.x\in C.
\]
The scaffold covariance gives
\[
    \mathscr P_c=k\mathscr P_x-a,
    \qquad\text{hence}\qquad
    \mathscr P_x=p+r\mathscr P_c.
\]
Thus \(\Phi_{x,g}(y)=p+ry\) sends the distinguished site
\(0\in\mathscr P_c\) to \(p\). If \(g'\in C_x\), then
\(g'g^{-1}\in C_c\) and
\[
    \operatorname{lab}_c(g'g^{-1})
    =
    \operatorname{lab}_x(g').
\]
Therefore \(\Phi_{x,g}\) preserves the nearest-site comparison and the label
used to break ties. It sends \(W_c\) onto the cell assigned to \(p\) in
\(\mathscr P_x\), which proves \eqref{eq:physical-cell-partition}. Equation
\eqref{eq:physical-cell-covariance} follows from
\eqref{eq:framed-transport-identity} and
\eqref{eq:framed-coordinate-covariance}.

Finally, \(\mathscr W\) is Borel and all its sections lie in the fixed ball
\(\overline B_{R_{\rm Vor}}\). The parameter-integration statement in
Subsection~\ref{subsec:Borel-geometric-preliminaries} gives the asserted Borel
dependence of volumes and polynomial moments.
\end{proof}


\subsection{Integerization and packet transfer}
\label{subsec:integerization-packet-transfer}

We now modify the cells \(W_c\) so that their \(\rho\)-masses are integers.
By \emph{integerization} we mean the following two steps: first define integers
\(q_\rho(c)\) whose difference from \(\rho\lambda_V(W_c)\) is expressed along
the parent map; then change \(W_c\) by sets of the corresponding masses so
that the resulting set has \(\rho\)-mass exactly \(q_\rho(c)\). In
Section~\ref{sec:local-filling-assembly}, that set will be replaced by exactly
\(q_\rho(c)\) points.

\emph{Integer masses.}
Fix \(\rho>0\), and write \(\{s\}:=s-\lfloor s\rfloor\) for the fractional
part of \(s\in\mathbb R\). For \(c\in C\), define the real mass of its
labelled Voronoi cell and the total real mass of its finite descendant subtree by
\begin{equation}
\label{eq:cell-mass-integerization}
    m_\rho(c):=\rho\lambda_V(W_c),
    \qquad
    M_\rho(c)
    :=
    \sum_{w\in\operatorname{Desc}(c)}m_\rho(w).
\end{equation}
The sum defining \(M_\rho(c)\) is finite by
Subsection~\ref{subsec:proximity-forest}. Define the target integer mass
\begin{equation}
\label{eq:integerized-cell-mass}
    q_\rho(c)
    :=
    \lfloor M_\rho(c)\rfloor
    -
    \sum_{w\in\operatorname{Ch}(c)}
        \lfloor M_\rho(w)\rfloor.
\end{equation}
All functions in
\eqref{eq:cell-mass-integerization}--\eqref{eq:integerized-cell-mass} are
Borel, because the child and descendant relations are Borel with finite
sections.

\begin{lemma}[Forest integerization]
\label{lem:forest-integerization}
For every \(c\in C\),
\begin{equation}
\label{eq:forest-divergence-identity}
    q_\rho(c)-m_\rho(c)
    =
    \sum_{w\in\operatorname{Ch}(c)}
        \{M_\rho(w)\}
    -
    \{M_\rho(c)\}.
\end{equation}
Consequently,
\begin{equation}
\label{eq:integerization-error-bounds}
    -1<q_\rho(c)-m_\rho(c)<b_{\mathcal T}.
\end{equation}
In particular, there is \(\rho_0<\infty\), depending only on the fixed
geometric constants, such that
\[
    q_\rho(c)\geq1
    \qquad(c\in C,\ \rho\geq\rho_0).
\]
\end{lemma}

\begin{proof}
The descendant decomposition
\[
    \operatorname{Desc}(c)
    =
    \{c\}
    \sqcup
    \bigsqcup_{w\in\operatorname{Ch}(c)}
        \operatorname{Desc}(w)
\]
gives
\[
    M_\rho(c)
    =
    m_\rho(c)
    +
    \sum_{w\in\operatorname{Ch}(c)}M_\rho(w).
\]
Subtracting integer parts gives
\eqref{eq:forest-divergence-identity}. Since every fractional part belongs to
\([0,1)\),
\[
    q_\rho(c)-m_\rho(c)
    \leq\#\operatorname{Ch}(c)
    <b_{\mathcal T},
\]
and the lower bound is strict because \(\{M_\rho(c)\}<1\). This proves
\eqref{eq:integerization-error-bounds}.

By \eqref{eq:labelled-cell-geometry},
\[
    m_\rho(c)
    \geq
    \rho\lambda_V(B_{r_{\rm Vor}}).
\]
Thus any \(\rho_0\) with
\(\rho_0\lambda_V(B_{r_{\rm Vor}})\geq1\) makes \(q_\rho(c)>0\); since
\(q_\rho(c)\) is an integer, it is then at least one.
\end{proof}

Put
\[
    \theta_\rho(c):=\{M_\rho(c)\}\in[0,1).
\]
Then \eqref{eq:forest-divergence-identity} becomes
\[
    q_\rho(c)
    =
    m_\rho(c)-\theta_\rho(c)
    +\sum_{w\in\operatorname{Ch}(c)}\theta_\rho(w).
\]
Accordingly, we remove from \(W_c\) a set of \(\rho\)-mass
\(\theta_\rho(c)\) and add, for each child \(w\), the set of \(\rho\)-mass
\(\theta_\rho(w)\) removed from \(W_w\).

\emph{Modification of the cells.}
From now on assume \(\rho\geq\rho_0\). For \(c\in C\), define
\begin{equation}
\label{eq:packet-homothety-factor}
    \alpha_\rho(c)
    :=
    \left(
        1-
        \frac{\theta_\rho(c)}{m_\rho(c)}
    \right)^{1/d}
\end{equation}
and, with dilations taken about the origin,
\begin{equation}
\label{eq:outgoing-shell-definition}
    \mathscr O_\rho(c)
    :=
    W_c
    \setminus
    \alpha_\rho(c)\operatorname{Vor}(c).
\end{equation}
If \(\alpha_\rho(c)<1\), then
\(\alpha_\rho(c)\operatorname{Vor}(c)\) is contained in the interior of
\(\operatorname{Vor}(c)\), hence in \(W_c\); when
\(\alpha_\rho(c)=1\), the outgoing shell is empty. Therefore
\begin{equation}
\label{eq:outgoing-shell-mass}
    \rho\lambda_V(\mathscr O_\rho(c))
    =
    \theta_\rho(c).
\end{equation}
Moreover,
\begin{equation}
\label{eq:homothety-thickness}
    1-\alpha_\rho(c)
    \leq
    \frac{1}{
        \rho\lambda_V(B_{r_{\rm Vor}})
    },
\end{equation}
and hence
\begin{equation}
\label{eq:outgoing-shell-thickness}
    \sup_{y\in\mathscr O_\rho(c)}
    \operatorname{dist}
        \bigl(y,\partial\operatorname{Vor}(c)\bigr)
    \leq
    \frac{R_{\rm Vor}}{
        \rho\lambda_V(B_{r_{\rm Vor}})
    }.
\end{equation}
Indeed, if \(y=r\theta\in\mathscr O_\rho(c)\) and
\(s\theta\in\partial\operatorname{Vor}(c)\) lies on the same ray, then
\(r>\alpha_\rho(c)s\). Hence
\[
    \operatorname{dist}
        \bigl(y,\partial\operatorname{Vor}(c)\bigr)
    \leq
    s-r
    \leq
    R_{\rm Vor}\bigl(1-\alpha_\rho(c)\bigr),
\]
and \eqref{eq:outgoing-shell-thickness} follows from
\eqref{eq:homothety-thickness}.

The set removed from a child \(w\) is expressed in the local coordinates of
\(w\). To include the same subset of \(V\) in the set indexed by its parent
\(c\), write
\(\gamma(c,w)=(a,k)\), so that \(w=(a,k).c\), and define the affine isometry
\begin{equation}
\label{eq:packet-coordinate-transition}
    \operatorname{Iso}_{w\to c}(y)
    :=
    k^{-1}a+k^{-1}y,
    \qquad y\in V.
\end{equation}
Equivalently,
\[
    \operatorname{Iso}_{w\to c}(y)
    =
    \operatorname{pos}_c(\gamma(c,w))
    +
    \operatorname{fr}_c(\gamma(c,w))y.
\]

\begin{lemma}[Coordinate transfer along the forest]
\label{lem:packet-coordinate-transition}
Let \(x\in X\), let \(g\in C_x\) have label \(c\), and let
\(w\mathrel{\mathcal R_C}c\). Then \(\gamma(c,w)g\in C_x\) has label \(w\),
and
\begin{equation}
\label{eq:same-physical-point-coordinate-change}
    \Phi_{x,\gamma(c,w)g}(y)
    =
    \Phi_{x,g}
        \bigl(\operatorname{Iso}_{w\to c}(y)\bigr),
    \qquad y\in V.
\end{equation}
Thus the two sides of \eqref{eq:same-physical-point-coordinate-change}
represent the same point of \(V\) in the two local coordinate systems.
\end{lemma}

\begin{proof}
Write \(g=(b,\ell)\) and \(\gamma(c,w)=(a,k)\). Then
\[
    \gamma(c,w)g=(a+kb,k\ell),
\]
and therefore
\[
\begin{aligned}
    \operatorname{pos}_x(\gamma(c,w)g)
    &=
    \ell^{-1}b+\ell^{-1}k^{-1}a,
    \\
    \operatorname{fr}_x(\gamma(c,w)g)
    &=
    \ell^{-1}k^{-1}.
\end{aligned}
\]
Since \(\operatorname{pos}_x(g)=\ell^{-1}b\) and
\(\operatorname{fr}_x(g)=\ell^{-1}\), this is exactly
\eqref{eq:same-physical-point-coordinate-change}.
\end{proof}

For \(\rho\geq\rho_0\) and \(c\in C\), define
\begin{equation}
\label{eq:packet-definition}
    \Pi_\rho(c)
    :=
    \bigl(W_c\setminus\mathscr O_\rho(c)\bigr)
    \cup
    \bigcup_{w\in\operatorname{Ch}(c)}
        \operatorname{Iso}_{w\to c}
            \bigl(\mathscr O_\rho(w)\bigr).
\end{equation}
The first term is the part retained from \(W_c\); the remaining terms are the
sets removed from the children, written in the coordinates of \(c\).

\begin{proposition}[Integer-mass packets]
\label{prop:integer-mass-packets}
After increasing \(\rho_0\) if necessary, the following hold for all
\(\rho\geq\rho_0\).

The family \(c\mapsto\Pi_\rho(c)\) is Borel, and for every \(x\in X\),
\begin{equation}
\label{eq:physical-packet-partition}
    V
    =
    \bigsqcup_{g\in C_x}
        \Phi_{x,g}
        \bigl(\Pi_\rho(\operatorname{lab}_x(g))\bigr).
\end{equation}
Moreover, for every \(c\in C\),
\begin{equation}
\label{eq:integer-packet-mass}
    \rho\lambda_V(\Pi_\rho(c))
    =
    q_\rho(c)
    \in\mathbb N,
\end{equation}
and
\begin{equation}
\label{eq:packet-symmetric-difference}
    \rho\lambda_V
        \bigl(\Pi_\rho(c)\triangle W_c\bigr)
    <
    1+b_{\mathcal T}.
\end{equation}
There is a common radius
\begin{equation}
\label{eq:packet-core-radius}
    r_*:=\frac{r_{\rm Vor}}2
\end{equation}
such that
\begin{equation}
\label{eq:packet-core-and-support}
    \overline B_{r_*}
    \subset
    W_c\setminus\mathscr O_\rho(c)
    \subset
    \Pi_\rho(c)
    \subset
    \overline B_{D_{\rm Rips}+R_{\rm Vor}}.
\end{equation}
Finally, for \(x\in X\), \(g\in C_x\), and \(h=(v,\ell)\in G_K\),
\begin{equation}
\label{eq:physical-packet-covariance}
    \Phi_{h.x,gh^{-1}}
        \bigl(\Pi_\rho(\operatorname{lab}_{h.x}(gh^{-1}))\bigr)
    =
    \ell\,
    \Phi_{x,g}
        \bigl(\Pi_\rho(\operatorname{lab}_x(g))\bigr)-v.
\end{equation}
\end{proposition}

Since the placement maps \(\Phi_{x,g}\) are isometries, every physical packet
has diameter at most \(2(D_{\rm Rips}+R_{\rm Vor})\). Hence any physical
packet selected according to the stationary transverse law has diameter
bounded by the same constant almost surely, and therefore has polynomial and
exponential moments of every order.

\begin{proof}
By Proposition~\ref{prop:labelled-Voronoi-partition}, the sets in
\eqref{eq:physical-cell-partition} partition \(V\). For each \(c\),
\(\mathscr O_\rho(c)\subset W_c\). It is removed from the term indexed by
\(c\), and because \(c\) has exactly one parent it is added exactly once to
the term indexed by \(\operatorname{par}(c)\). Lemma~\ref{lem:packet-coordinate-transition}
shows that this addition is the same subset of \(V\), expressed in the
parent's coordinates. Therefore \eqref{eq:physical-packet-partition} is an
exact partition.

The pieces in \eqref{eq:packet-definition} are therefore disjoint. Using
\eqref{eq:outgoing-shell-mass} and
\eqref{eq:forest-divergence-identity}, we obtain
\[
\begin{aligned}
    \rho\lambda_V(\Pi_\rho(c))
    &=
    m_\rho(c)-\theta_\rho(c)
    +
    \sum_{w\in\operatorname{Ch}(c)}\theta_\rho(w)
    \\
    &=
    q_\rho(c),
\end{aligned}
\]
which proves \eqref{eq:integer-packet-mass}. The same disjointness gives
\[
    \rho\lambda_V
        \bigl(\Pi_\rho(c)\triangle W_c\bigr)
    =
    \theta_\rho(c)
    +
    \sum_{w\in\operatorname{Ch}(c)}\theta_\rho(w)
    <
    1+b_{\mathcal T},
\]
proving \eqref{eq:packet-symmetric-difference}.

Increase \(\rho_0\) so that \(\alpha_\rho(c)\geq1/2\) for every \(c\) and
\(\rho\geq\rho_0\). Since
\(\overline B_{r_{\rm Vor}}\subset\operatorname{Vor}(c)\), the set
\(W_c\setminus\mathscr O_\rho(c)\) contains
\(\overline B_{r_{\rm Vor}/2}\). This proves the first
two inclusions in \eqref{eq:packet-core-and-support}.

If \(w\in\operatorname{Ch}(c)\), then \(d_C(c,w)\leq D_{\rm Rips}\). Hence
\(\operatorname{Iso}_{w\to c}(W_w)\) is contained in the ball of radius
\(R_{\rm Vor}\) about a point at distance at most \(D_{\rm Rips}\) from the
origin. The last inclusion in \eqref{eq:packet-core-and-support} follows.

For Borelness, the functions \(m_\rho,M_\rho,q_\rho,\theta_\rho\) are Borel,
the child relation has finite sections, and
\[
    (c,w,y)
    \longmapsto
    \operatorname{Iso}_{w\to c}(y)
\]
is Borel by the Borelness of \(\gamma\). Together with the Borel cell family
from Subsection~\ref{subsec:labelled-Voronoi-cells}, this shows that
\[
    \{(c,y)\in C\times V:y\in\Pi_\rho(c)\}
\]
is Borel. Finally, \eqref{eq:physical-packet-covariance} follows from
\eqref{eq:framed-transport-identity} and the invariance of labels in
\eqref{eq:framed-coordinate-covariance}.
\end{proof}


\section{Local discretization and global realization}
\label{sec:local-filling-assembly}

We retain the standing higher-dimensional setup of
Section~\ref{sec:orbit-geometry-integerization}. Fix an integer \(p\geq1\) and,
for \(\rho\geq\rho_0\), put
\[
    h:=\rho^{-1/d}.
\]
Constants introduced in this section may depend on \(d\), \(p\), and the
fixed geometric constants from Section~\ref{sec:orbit-geometry-integerization},
but not on \(c\) or \(\rho\). We may increase \(\rho_0\) finitely many
times; every such increase is uniform in \(c\).

For every \(c\in C\), Proposition~\ref{prop:integer-mass-packets} gives a
bounded set \(\Pi_\rho(c)\) with
\[
    \rho\lambda_V(\Pi_\rho(c))=q_\rho(c)\in\mathbb N.
\]

By Proposition~\ref{prop:configuration-realization}, it suffices to construct a
Borel \(G_K\)-equivariant map
\[
    \kappa_\rho:X\longrightarrow\mathcal N_s(V),
    \qquad
    \kappa_\rho(x)=\delta_{P_x},
\]
whose configurations \(P_x\) are uniformly Delone and for which
\(\kappa_\rho\) is injective. Then
\[
    Y_\rho:=\{x\in X:0\in P_x\}
\]
is a \(K\)-invariant generating Delone translation cross-section with return
sets \(P_x\).
We construct \(P_x\) from local sets \(Q_\rho(c)\). Each \(Q_\rho(c)\) must
contain exactly \(q_\rho(c)\) points, all lying in
\(W_c\setminus\mathscr O_\rho(c)\), with separation, boundary-distance, and
covering bounds of order \(h\). It must also satisfy
\[
    \sum_{u\in Q_\rho(c)}P(u)
    =
    \rho\int_{\Pi_\rho(c)}P(y)\,d\lambda_V(y)
\]
for every polynomial \(P\) of degree less than \(p\). These identities are the
input for Section~\ref{sec:spectral-decay-rigidity}.

To prove injectivity, fix a
Borel injection \(b:C\to(0,1)\) and construct finite sets \(S_s\),
\(s\in(0,1)\), such that an unlabelled placed copy
\[
    a+r hS_s,
    \qquad a\in V,\quad r\in K,
\]
determines \(a\), \(r\), and \(s\). We arrange
\(hS_{b(c)}\subset Q_\rho(c)\) and separate this subset from all other short
interpoint distances. After global placement, one such copy determines a
position, frame, and label; equation~\eqref{eq:framed-recovery} then
determines the original state.

The proof proceeds in three local steps: quadrature on
\(\operatorname{Vor}(c)\), finite insertion of \(hS_{b(c)}\), and an
\(O(h)\) correction inside \(B_{r_*}\) that makes the moments of degree less
than \(p\) exact.

Ignoring the additional geometric and measurable requirements, the exact
moment identities are equal-weight quadrature formulas: since
\(q_\rho(c)=\rho\lambda_V(\Pi_\rho(c))\), they say that the points of
\(Q_\rho(c)\) form an averaging set for normalized Lebesgue measure on
\(\Pi_\rho(c)\) through degree \(p-1\). Averaging sets were introduced by
Seymour--Zaslavsky \cite{SZ84}; see also Wagner--Volkmann \cite{WV91} and,
for the closely related theory of spherical designs,
\cite{BV10,BRV13}. Moment-preserving point placements were used in
hyperuniform constructions by Gabrielli--Joyce--Torquato \cite{GJT08}.
More recently, Lotz--Klatt \cite[Appendix~C]{LK26} develop averaging sets for
precisely this type of application; their Theorem~C.8 gives existence under
uniform geometric control, and Remark~C.9 discusses the additional
measurability issue. The theorem below also requires Borel dependence on
\(c\), uniform separation, boundary-distance and covering estimates, and the
prescribed finite subset \(S_{b(c)}\) used to recover the original state.

We do not optimize the threshold in \(p\): although the correction space has
dimension \(\binom{d+p-1}{d}-1=O_d(p^d)\), the \(p\)-dependence of the
conditioning and derivative constants is not tracked, so the proof gives no
useful bound on \(\rho_0\) or the required local cardinalities.

\subsection{Local discretization theorem}
\label{subsec:local-filling}

\begin{theorem}[Uniform local discretization with exact moments]
\label{thm:local-filling}
There exist \(\rho_0<\infty\) and constants
\(c_{\rm fill},c_{\rm bd},C_{\rm fill},c_{\rm rec}>0\) such that, for every
\(\rho\geq\rho_0\), there is an assignment
\[
    c\longmapsto Q_\rho(c)\subset V
\]
for which \(c\mapsto\delta_{Q_\rho(c)}\) is Borel and the following
properties hold.

For every \(c\in C\),
\begin{equation}
\label{eq:local-filling-cardinality}
    Q_\rho(c)
    \subset
    W_c\setminus\mathscr O_\rho(c)
    \subset
    \Pi_\rho(c),
    \qquad
    \#Q_\rho(c)=q_\rho(c).
\end{equation}
Moreover, for every \(c\in C\),
\begin{equation}
\label{eq:local-filling-geometry}
\begin{aligned}
    \operatorname{sep}(Q_\rho(c))
    &\geq c_{\rm fill}h,\\
    \operatorname{dist}
        \bigl(Q_\rho(c),\partial\operatorname{Vor}(c)\bigr)
    &\geq c_{\rm bd}h,\\
    \sup_{y\in\operatorname{Vor}(c)}
        \operatorname{dist}(y,Q_\rho(c))
    &\leq C_{\rm fill}h.
\end{aligned}
\end{equation}
For every \(c\in C\) and every \(\mathbf j\in\mathbb N_0^d\) with
\(|\mathbf j|<p\),
\begin{equation}
\label{eq:local-filling-moments}
    \sum_{u\in Q_\rho(c)}u^{\mathbf j}
    =
    \rho\int_{\Pi_\rho(c)}
        y^{\mathbf j}\,d\lambda_V(y).
\end{equation}

For every \(c\in C\) and every Lipschitz function \(f\) on
\(\overline B_{D_{\rm Rips}+R_{\rm Vor}}\),
\begin{equation}
\label{eq:local-filling-Lipschitz}
    \left|
        \sum_{u\in Q_\rho(c)}f(u)
        -
        \rho\int_{\Pi_\rho(c)}
            f\,d\lambda_V
    \right|
    \leq
    C_{\rm fill}\rho h\,\operatorname{Lip}(f).
\end{equation}

There is also a family of sets
\[
    S_s\subset V,
    \qquad s\in(0,1),
    \qquad \#S_s=d+1,
\]
such that \(s\mapsto\delta_{S_s}\) is Borel and, for every \(c\in C\),
\begin{equation}
\label{eq:local-filling-marker}
    hS_{b(c)}\subset Q_\rho(c).
\end{equation}
In addition,
for every \(c\in C\),
\begin{equation}
\label{eq:local-filling-recognition-gap}
    \operatorname{diam}(hS_{b(c)})<\frac{c_{\rm rec}h}{5},
    \qquad
    \|u-v\|\geq c_{\rm rec}h
\end{equation}
whenever \(u\neq v\) belong to \(Q_\rho(c)\) and
\(\{u,v\}\not\subset hS_{b(c)}\). From any set
\[
    a+rhS_s,
    \qquad
    a\in V,\quad r\in K,\quad s\in(0,1),
\]
the map \((a,r,s)\mapsto\delta_{a+rhS_s}\) is injective and has a Borel inverse on its image.
\end{theorem}

Throughout the proof, \(\rho_0\) denotes the threshold inherited from
Proposition~\ref{prop:integer-mass-packets} and may be increased finitely
many times. All subsequent assertions are for \(\rho\geq\rho_0\).



\subsection{Uniform discretization of convex cells}
\label{subsec:quantization}

We first construct the required number of points in the uniformly controlled
convex cells \(\operatorname{Vor}(c)\).

\begin{lemma}[Equal-mass rectangular partition]
\label{lem:equal-mass-rectangles}
Let \(\lambda_n\) denote Lebesgue measure on \([0,1]^n\). Let
\(n\geq1\) and \(0<m\leq M<\infty\). There are constants
\(a_{n,m,M},A_{n,m,M}>0\) such that the following holds. Suppose that \(\nu\)
is a probability measure on \([0,1]^n\) of the form
\[
    d\nu=f\,d\lambda_n,
    \qquad
    m\leq f\leq M
    \quad \lambda_n\text{-a.e.}
\]
Then, for every integer \(q\geq1\), there is a partition into axis-parallel
rectangles
\[
    [0,1]^n=\bigsqcup_{j=1}^q R_j,
    \qquad
    \nu(R_j)=\frac1q,
\]
with a fixed half-open boundary convention, such that every side length
\(\ell\) of every \(R_j\) satisfies
\begin{equation}
\label{eq:equal-mass-rectangle-sides}
    a_{n,m,M}q^{-1/n}\leq \ell\leq A_{n,m,M}q^{-1/n}.
\end{equation}

Suppose additionally that \(Z\) is a standard Borel space and
\(z\mapsto\nu_z\) is a Borel family of probability measures, meaning that
\(z\mapsto\nu_z(B)\) is Borel for every Borel \(B\subset[0,1]^n\), with
\(d\nu_z=f_z\,d\lambda_n\) and \(m\leq f_z\leq M\) almost everywhere. Then
the rectangles can be indexed as
\[
    R_j(z)=I_{j,1}(z)\times\cdots\times I_{j,n}(z),
    \qquad 1\leq j\leq q,
\]
so that the lower and upper coordinate bounds of every interval
\(I_{j,k}(z)\) are Borel functions of \(z\).
\end{lemma}

\begin{proof}
We argue by induction on \(n\). For \(n=1\), let
\[
    F(t):=\nu([0,t]),
    \qquad 0\leq t\leq1.
\]
The density bounds give
\[
    m(t-s)\leq F(t)-F(s)\leq M(t-s)
    \qquad(0\leq s<t\leq1).
\]
Thus \(F\) is continuous and strictly increasing. If
\(F(t_j)=j/q\), then
\[
    \frac1{Mq}
    \leq t_j-t_{j-1}
    \leq\frac1{mq},
\]
and the intervals determined by the \(t_j\)'s give the required
partition.

Assume the result in dimension \(n-1\), with \(n\geq2\). For
\(q>1\), let \(N:=\lfloor q^{1/n}\rfloor\) and choose positive
integers \(q_1,\ldots,q_N\), differing by at most one, with
\[
    q_1+\cdots+q_N=q.
\]
Then
\begin{equation}
\label{eq:qi-comparable}
    c_n q^{(n-1)/n}
    \leq q_i\leq
    C_n q^{(n-1)/n}
\end{equation}
for constants depending only on \(n\). The case \(q=1\) is trivial. If
\(N=1\), then \(q<2^n\) and \(q_1=q\), so
\eqref{eq:qi-comparable} still holds after adjusting the constants depending
only on \(n\); the single strip constructed below is the whole cube in the
first coordinate, and the induction continues in dimension \(n-1\).

Let
\[
    F(t):=\nu([0,t]\times[0,1]^{n-1}).
\]
Again,
\[
    m(t-s)\leq F(t)-F(s)\leq M(t-s).
\]
Choose
\[
    0=t_0<t_1<\cdots<t_N=1
\]
so that
\[
    F(t_i)-F(t_{i-1})=\frac{q_i}{q},
\]
and let \(I_i\) be the corresponding half-open intervals. Writing
\(w_i:=|I_i|\), we obtain
\begin{equation}
\label{eq:strip-widths}
    \frac{q_i}{Mq}
    \leq w_i\leq
    \frac{q_i}{mq}.
\end{equation}
By \eqref{eq:qi-comparable}, these widths are comparable to
\(q^{-1/n}\).

Define a probability measure on \([0,1]^{n-1}\) by
\[
    \nu_i(A)
    :=
    \frac{q}{q_i}\,
    \nu(I_i\times A).
\]
It has density with respect to Lebesgue measure on \([0,1]^{n-1}\),
\[
    f_i(y)
    =
    \frac{q}{q_i}
    \int_{I_i}f(t,y)\,dt.
\]
Using \eqref{eq:strip-widths},
\[
    \frac{m}{M}
    \leq f_i(y)\leq
    \frac{M}{m}
    \qquad\text{a.e.}
\]
The induction hypothesis partitions \([0,1]^{n-1}\) into
\(q_i\) rectangles \(S_{ij}\), each of \(\nu_i\)-mass
\(1/q_i\), whose side lengths are comparable to
\(q_i^{-1/(n-1)}\). Hence
\[
    R_{ij}:=I_i\times S_{ij}
\]
has \(\nu\)-mass \(1/q\), and all its side lengths are comparable
to \(q^{-1/n}\) by \eqref{eq:qi-comparable} and
\eqref{eq:strip-widths}.

It remains to verify the parameterized statement. For \(n=1\), define
\[
    F_z(t):=\nu_z([0,t]).
\]
For \(1\leq j<q\), the coordinate \(t_j(z)\) is the unique solution of
\(F_z(t)=j/q\). The map \((z,t)\mapsto F_z(t)\) is Borel, and each
\(t\mapsto F_z(t)\) is continuous and strictly increasing. Hence \(t_j\) is
Borel; for rational \(r\),
\[
    \{z:t_j(z)<r\}=\{z:F_z(r)>j/q\}.
\]

For \(n\geq2\), use
\[
    F_z(t):=\nu_z([0,t]\times[0,1]^{n-1}),
    \qquad
    \alpha_i:=\frac{q_1+\cdots+q_i}{q}.
\]
The first-coordinate cut \(t_i(z)\) is the unique solution of
\(F_z(t)=\alpha_i\), and the same argument shows that \(t_i\) is Borel. Thus
the strip intervals \(I_i(z)\) have Borel coordinate bounds. For every Borel
set \(A\subset[0,1]^{n-1}\),
\[
    \nu_{z,i}(A):=\frac{q}{q_i}\nu_z(I_i(z)\times A)
\]
defines a Borel family of probability measures. Indeed, the indicator of
\(\{(z,t,y):t\in I_i(z),\ y\in A\}\) is jointly Borel, so the
parameter-integration statement in
Subsection~\ref{subsec:Borel-geometric-preliminaries} gives the Borel
dependence on \(z\). The induction hypothesis, applied to each family
\(z\mapsto\nu_{z,i}\), gives Borel lower and upper coordinate bounds for
the remaining coordinates of every rectangle.
\end{proof}

\begin{lemma}[Uniform \(q\)-point discretization of convex sets]
\label{lem:uniform-convex-quantization}
Fix \(0<r<R<\infty\). There are constants \(a,A>0\) and a Borel
assignment
\[
    (D,q)\longmapsto A_q(D)\subset\operatorname{int}D,
\]
defined for compact convex sets \(D\subset V\), equipped with the Borel
structure induced by the Hausdorff topology, satisfying
\[
    \overline B_r\subset D\subset\overline B_R
\]
and integers \(q\geq1\), such that
\begin{equation}
\label{eq:convex-quantization-geometry}
\begin{aligned}
    \#A_q(D)&=q,\\
    \operatorname{sep}(A_q(D))&\geq aq^{-1/d},\\
    \operatorname{dist}(A_q(D),\partial D)&\geq aq^{-1/d},\\
    \sup_{y\in D}\operatorname{dist}(y,A_q(D))&\leq Aq^{-1/d}.
\end{aligned}
\end{equation}
For every Lipschitz function \(f\) on \(\overline B_R\),
\begin{equation}
\label{eq:convex-quantization-Lipschitz}
    \left|
        \sum_{u\in A_q(D)}f(u)
        -
        \frac{q}{\lambda_V(D)}
        \int_D f\,d\lambda_V
    \right|
    \leq
    Aq^{1-1/d}\operatorname{Lip}(f).
\end{equation}
\end{lemma}

\begin{proof}
\emph{Uniform parametrization.}
Let \(Q=[0,1]^d\), and fix once and for all a bi-Lipschitz homeomorphism
\(G:Q\to\overline B_1\). For a convex set \(D\) as above, its
Minkowski functional is
\[
    p_D(x):=\inf\{t>0:x\in tD\},
    \qquad
    tD:=\{ty:y\in D\}.
\]
The inclusions \(\overline B_r\subset D\subset\overline B_R\)
give
\[
    \frac{\|x\|}{R}
    \leq p_D(x)
    \leq\frac{\|x\|}{r}.
\]
Subadditivity gives
\[
    p_D(x)-p_D(y)\leq p_D(x-y),
    \qquad
    p_D(y)-p_D(x)\leq p_D(y-x),
\]
and therefore
\[
    |p_D(x)-p_D(y)|\leq\frac{\|x-y\|}{r}.
\]
For \(\theta\in S^{d-1}\), put
\[
    \operatorname{rad}_D(\theta):=p_D(\theta)^{-1}.
\]
Then
\[
    r\leq\operatorname{rad}_D(\theta)\leq R,
    \qquad
    |\operatorname{rad}_D(\theta)
      -\operatorname{rad}_D(\varphi)|
    \leq
    \frac{R^2}{r}\|\theta-\varphi\|.
\]
Define \(H_D(0)=0\) and
\[
    H_D(s\theta)
    :=s\operatorname{rad}_D(\theta)\theta,
    \qquad 0<s\leq1.
\]
The maps \(H_D:\overline B_1\to D\) and \(H_D^{-1}\) have
Lipschitz constants bounded in terms of \(r,R\). To see this uniformly,
write \(x=s\theta\) and \(y=t\varphi\). Then
\(|s-t|\leq\|x-y\|\) and
\(\min\{s,t\}\,\|\theta-\varphi\|\leq2\|x-y\|\); together with
the preceding Lipschitz bound for \(\operatorname{rad}_D\), this gives a
uniform Lipschitz bound for \(H_D\). The inverse is
\[
    H_D^{-1}(y)=p_D(y)\frac{y}{\|y\|}
    \qquad(y\neq0),
\]
and
\[
    \left|\frac1{\operatorname{rad}_D(\theta)}
      -\frac1{\operatorname{rad}_D(\varphi)}\right|
    \leq \frac{R^2}{r^3}\|\theta-\varphi\|.
\]
The same polar estimate therefore gives a uniform Lipschitz bound for
\(H_D^{-1}\). Thus
\[
    F_D:=H_D\circ G:Q\longrightarrow D
\]
is \(L\)-bi-Lipschitz for a constant depending only on \(d,r,R\).
The dependence \(D\mapsto F_D\) is Borel: if
\(d_H(D,E)=\delta<r/2\), then
\(D\subset E+\delta B_1\subset(1+\delta/r)E\) and conversely. Hence
the Minkowski functionals, and therefore the radial functions, converge
uniformly on the sphere under Hausdorff convergence in this class. Hence
\((D,z)\mapsto F_D(z)\) is jointly Borel, and the explicit formula for
\(H_D^{-1}\) shows that \((D,u)\mapsto F_D^{-1}(u)\) is jointly Borel on
\(\{(D,u):u\in D\}\).

\emph{Equal-mass partition.}
Pull normalized Lebesgue measure on \(D\) back to \(Q\):
\[
    \nu_D
    :=
    (F_D^{-1})_*
    \left(
        \frac{\lambda_V|_D}{\lambda_V(D)}
    \right).
\]
By Rademacher's theorem and the area formula for Lipschitz maps; see, for
example, \cite[Chapters~3 and~6]{EG15}, the \(L\)-bi-Lipschitz map \(F_D\) is
differentiable almost everywhere, and its Jacobian satisfies
\[
    L^{-d}\leq |\det DF_D|\leq L^d
    \qquad\text{a.e. on }Q.
\]
The change-of-variables formula for the Lipschitz injection \(F_D\), together
with
\(\lambda_V(B_r)\leq\lambda_V(D)\leq\lambda_V(B_R)\), therefore shows that
\(\nu_D\) has a density with respect to Lebesgue measure on \(Q\), bounded above and below by positive constants depending
only on \(d,r,R\). The family \(D\mapsto\nu_D\) is Borel: for every bounded
Borel \(f:Q\to\mathbb R\),
\(
\int_Q f\,d\nu_D=\lambda_V(D)^{-1}
\int_D f(F_D^{-1}(u))\,d\lambda_V(u)
\), which depends Borel measurably on \(D\) by the joint Borelness above and
the parameter-integration statement in
Subsection~\ref{subsec:Borel-geometric-preliminaries}.

Apply Lemma~\ref{lem:equal-mass-rectangles} to \(\nu_D\). Let
\[
    Q=\bigsqcup_{j=1}^q R_j,
    \qquad
    \nu_D(R_j)=\frac1q,
\]
and let \(z_j\) be the center of \(R_j\). Set
\[
    A_q(D):=\{F_D(z_1),\ldots,F_D(z_q)\}.
\]

\emph{Geometry, quadrature estimate, and Borel dependence.}
Distinct rectangles have disjoint interiors, so in some coordinate
their centers differ by at least half the sum of the corresponding
side lengths. Also
\[
    \operatorname{dist}(z_j,\partial Q)
    \geq \frac12\min\{\text{side lengths of }R_j\}.
\]
Together with the diameter bound for \(R_j\), the bi-Lipschitz bounds
for \(F_D\) give all three geometric estimates in
\eqref{eq:convex-quantization-geometry}.

The measure
\[
    \frac{q}{\lambda_V(D)}\lambda_V|_D
\]
has mass one on each \(F_D(R_j)\). Hence
\[
\begin{aligned}
    &\left|
        \sum_{j=1}^q f(F_D(z_j))
        -
        \frac{q}{\lambda_V(D)}
        \int_D f\,d\lambda_V
    \right|\\
    &\qquad\leq
    Aq^{-1/d}\operatorname{Lip}(f)
    \sum_{j=1}^q
    \frac{q}{\lambda_V(D)}
    \lambda_V(F_D(R_j))
    =Aq^{1-1/d}\operatorname{Lip}(f).
\end{aligned}
\]
The parameterized part of Lemma~\ref{lem:equal-mass-rectangles}
and the Borel dependence of \(F_D\) show that
\((D,q)\mapsto\delta_{A_q(D)}\) is Borel.
\end{proof}

Apply the lemma with
\[
    D=\operatorname{Vor}(c),
    \qquad q=q_\rho(c),
\]
and put
\begin{equation}
\label{eq:initial-quantizer}
    A_\rho(c):=
    A_{q_\rho(c)}(\operatorname{Vor}(c)).
\end{equation}
By \eqref{eq:integerization-error-bounds} and
\eqref{eq:labelled-cell-geometry}, after increasing \(\rho_0\) there
are constants \(a_1,A_1>0\) such that
\begin{equation}
\label{eq:q-rho-comparable}
    a_1\rho\leq q_\rho(c)\leq A_1\rho
    \qquad(c\in C,\ \rho\geq\rho_0).
\end{equation}
Thus, with \(h=\rho^{-1/d}\),
\begin{equation}
\label{eq:initial-quantizer-geometry}
\begin{aligned}
    \#A_\rho(c)&=q_\rho(c),\\
    \operatorname{sep}(A_\rho(c))&\geq a_2h,\\
    \operatorname{dist}
        (A_\rho(c),\partial\operatorname{Vor}(c))
        &\geq a_2h,\\
    \sup_{y\in\operatorname{Vor}(c)}
        \operatorname{dist}(y,A_\rho(c))
        &\leq A_2h
\end{aligned}
\end{equation}
for constants \(a_2,A_2>0\).

Since \(A_\rho(c)\subset\operatorname{int}\operatorname{Vor}(c)\),
we have \(A_\rho(c)\subset W_c\). By
\eqref{eq:outgoing-shell-thickness}, every point of
\(\mathscr O_\rho(c)\) is within \(O(\rho^{-1})\) of
\(\partial\operatorname{Vor}(c)\). Since \(d\geq2\),
\(\rho^{-1}=o(h)\); after increasing \(\rho_0\),
\begin{equation}
\label{eq:initial-quantizer-confinement}
    A_\rho(c)
    \subset
    W_c\setminus\mathscr O_\rho(c)
    \subset
    \Pi_\rho(c).
\end{equation}

Let \(v_c:=\lambda_V(\operatorname{Vor}(c))\), and define
\[
    \mu_c^{\rm ref}
    :=
    \frac{q_\rho(c)}{v_c}
        \lambda_V|_{\operatorname{Vor}(c)},
    \qquad
    \mu_c^{\rm pkt}
    :=
    \rho\lambda_V|_{\Pi_\rho(c)}.
\]
Both measures have mass \(q_\rho(c)\). Since \(W_c\) and
\(\operatorname{Vor}(c)\) differ by a null set,
\eqref{eq:integerization-error-bounds} and
\eqref{eq:packet-symmetric-difference} give
\[
    \|\mu_c^{\rm ref}-\mu_c^{\rm pkt}\|_{\rm TV}
    \leq C
\]
with \(C\) independent of \(c\) and \(\rho\). Thus the two measures differ by uniformly bounded total variation. Both are
supported in
\(\overline B_{D_{\rm Rips}+R_{\rm Vor}}\). Subtracting
\(f(0)\), there is \(A_{\rm ref}<\infty\), independent of \(c\) and \(\rho\),
such that, for every Lipschitz \(f\) on that ball,
\begin{equation}
\label{eq:reference-packet-Lipschitz}
    \left|
        \frac{q_\rho(c)}{v_c}
        \int_{\operatorname{Vor}(c)}f\,d\lambda_V
        -
        \rho\int_{\Pi_\rho(c)}f\,d\lambda_V
    \right|
    \leq
    A_{\rm ref}\operatorname{Lip}(f).
\end{equation}
Combining
\eqref{eq:convex-quantization-Lipschitz},
\eqref{eq:q-rho-comparable}, and
\eqref{eq:reference-packet-Lipschitz} yields, for some
\(A_4<\infty\), independent of \(c\) and \(\rho\),
\begin{equation}
\label{eq:initial-quantizer-Lipschitz}
    \left|
        \sum_{u\in A_\rho(c)}f(u)
        -
        \rho\int_{\Pi_\rho(c)}f\,d\lambda_V
    \right|
    \leq
    A_4\rho h\operatorname{Lip}(f).
\end{equation}


\subsection{Finite sets determining position, frame, and label}
\label{subsec:finite-recovery-sets}

Equation~\eqref{eq:framed-recovery} shows that a position, its frame,
and its label determine the state \(x\). Fix the Borel injection
\(b:C\to(0,1)\) from Subsection~\ref{subsec:local-filling}. We construct sets
\(S_s\), \(s\in(0,1)\), so that every unlabelled copy \(a+r hS_s\)
determines \(a\), \(r\), and \(s\). After taking \(s=b(c)\), the
injectivity of \(b\) then also determines the label \(c\).

Use the orthonormal basis \(e_1,\ldots,e_d\) fixed in
Subsection~\ref{subsec:Euclidean-Fourier-conventions}. Put
\[
    \lambda_1(s):=1+\frac{s}{4},
    \qquad
    \lambda_i:=4^{i-1}\quad(2\leq i\leq d),
\]
and, for \(\delta>0\), define
\begin{equation}
\label{eq:full-frame-marker}
    S_s
    :=
    \{0,\delta\lambda_1(s)e_1,
        \delta\lambda_2e_2,\ldots,
        \delta\lambda_de_d\}.
\end{equation}

\begin{lemma}[Recovery of parameters from \(S_s\)]
\label{lem:full-frame-marker}
For every \(s\in(0,1)\), the pairwise distances in \(S_s\) are distinct. For
fixed \(h>0\), the map
\[
    (a,r,s)\longmapsto\delta_{a+r hS_s},
    \qquad
    a\in V,\quad r\in K,\quad s\in(0,1),
\]
is Borel and injective, and its inverse on its image is Borel.
\end{lemma}

\begin{proof}
Put \(v_1=\lambda_1(s)e_1\) and
\(v_i=\lambda_i e_i\) for \(i\geq2\). Apart from the factor
\(\delta\), the nonzero distances are
\[
    \lambda_i
    \quad\text{and}\quad
    (\lambda_i^2+\lambda_j^2)^{1/2}
    \qquad(i<j).
\]
For \(j\geq2\) and \(i<j\),
\begin{equation}
\label{eq:marker-distance-levels}
    \lambda_j
    <
    (\lambda_i^2+\lambda_j^2)^{1/2}
    <
    \frac{17}{16}\lambda_j.
\end{equation}
Since \(\lambda_{j+1}=4\lambda_j\), these intervals are disjoint for distinct
\(j\); for fixed \(j\), the middle term is strictly increasing in
\(\lambda_i\). Thus all pairwise distances are distinct.

The unique shortest edge joins \(0\) to \(v_1\). Of its two endpoints, \(0\)
is characterized by having distances \(\lambda_2,\ldots,\lambda_d\) to the
remaining vertices. Hence the metric structure determines \(0\) and then the
ordered vertices \(v_1,\ldots,v_d\). The shortest edge has length
\(\delta\lambda_1(s)\), so it determines \(s\).

For a placed copy \(a+r hS_s\), the vertex corresponding to \(0\) gives
\(a\). If \(z_i\) corresponds to \(\delta\lambda_i(s)e_i\), with
\(\lambda_i(s)=\lambda_i\) for \(i\geq2\), then
\[
    re_i=\frac{z_i-a}{h\delta\lambda_i(s)}.
\]
Thus \(r\) is determined. The parameter map is Borel; since it is injective,
Lusin--Souslin gives a Borel image and a Borel inverse on that image.
\end{proof}

\subsection{Insertion while preserving cardinality}
\label{subsec:marker-reservoir}

We now modify \(A_\rho(c)\) in two small regions contained in \(B_{r_*}\). Near the origin we insert \(hS_{b(c)}\). The number of points of
\(A_\rho(c)\) removed from these regions can depend on \(c\); a second fixed
region near a point \(x_{\rm aux}\) is used to insert exactly enough auxiliary
points to preserve the cardinality \(q_\rho(c)\). The moment correction in the
next subsection will move points only in an open set \(U_0\) disjoint from both
regions.

Choose bounded open sets
\[
    U_0\Subset U_1\Subset B_{r_*}
\]
with \(0\notin\overline U_1\), and choose
\[
    x_{\rm aux}\in
    B_{r_*}\setminus(\overline U_1\cup\{0\}).
\]

Let \(a_2,A_2\) be as in
\eqref{eq:initial-quantizer-geometry}. Choose \(L>A_2\) such that
\begin{equation}
\label{eq:reservoir-supply-choice}
    \left(\frac{L-A_2}{A_2}\right)^d\geq d+1.
\end{equation}
For \(\rho\) sufficiently large, the balls
\(\overline B_h(0)\) and
\(\overline B_{Lh}(x_{\rm aux})\) are disjoint subsets of
\(B_{r_*}\setminus\overline U_1\). Put
\begin{equation}
\label{eq:marker-reservoir-deletion-region}
    D_\rho
    :=
    \overline B_h(0)
    \cup
    \overline B_{Lh}(x_{\rm aux}).
\end{equation}

After increasing \(\rho_0\) if necessary, there is \(N_*<\infty\),
independent of \(c\) and \(\rho\), such that
\begin{equation}
\label{eq:marker-reservoir-count-bounds}
    d+1
    \leq
    \#\bigl(A_\rho(c)\cap D_\rho\bigr)
    \leq N_*
    \qquad(c\in C,\ \rho\geq\rho_0).
\end{equation}
The upper bound follows from the separation estimate in
\eqref{eq:initial-quantizer-geometry}. For the lower bound, for
large \(\rho\),
\[
    B_{(L-A_2)h}(x_{\rm aux})
    \subset
    \bigcup_{
        u\in A_\rho(c)\cap B_{Lh}(x_{\rm aux})
    }
    B_{A_2h}(u).
\]
Comparison of volumes and
\eqref{eq:reservoir-supply-choice} give
\[
    \#\bigl(
        A_\rho(c)\cap B_{Lh}(x_{\rm aux})
    \bigr)
    \geq d+1.
\]

Choose distinct
\(t_1,\ldots,t_{N_*}\in B_{L/4}\), let
\[
    T_n:=\{t_1,\ldots,t_n\}
    \quad(1\leq n\leq N_*),
    \qquad
    T_0:=\varnothing,
\]
and put
\[
    s_T:=\min_{i\neq j}\|t_i-t_j\|>0,
    \qquad
    c_0:=\min\{a_2,s_T,1/2,L/2\}.
\]
Choose \(\delta>0\) in \eqref{eq:full-frame-marker} so that
\begin{equation}
\label{eq:marker-size-choice}
    \sup_{s\in(0,1)}\operatorname{diam}(S_s)
    <\frac{c_0}{10},
    \qquad
    \sup_{s\in(0,1)}\sup_{u\in S_s}\|u\|
    <\frac14.
\end{equation}

For this choice of \(\delta\) and the fixed scale \(h\), put
\[
    \mathscr M_h
    :=
    \{\delta_{a+r hS_s}:a\in V,\ r\in K,\ s\in(0,1)\}.
\]
By Lemma~\ref{lem:full-frame-marker}, \(\mathscr M_h\) is Borel and the
parameter map has a Borel inverse on \(\mathscr M_h\).

For \(c\in C\), let
\[
    N_\rho(c)
    :=
    \#\bigl(A_\rho(c)\cap D_\rho\bigr),
    \qquad
    n_\rho(c):=N_\rho(c)-(d+1),
\]
and define
\begin{equation}
\label{eq:marked-quantizer}
\begin{aligned}
    \widetilde A_\rho(c)
    &:=
    \bigl(A_\rho(c)\setminus D_\rho\bigr)
    \cup hS_{b(c)}\\
    &\qquad\cup
    \bigl(x_{\rm aux}+hT_{n_\rho(c)}\bigr).
\end{aligned}
\end{equation}

\begin{proposition}[Finite modification of the initial discretization]
\label{prop:marker-reservoir}
After increasing \(\rho_0\), the map
\(c\mapsto\delta_{\widetilde A_\rho(c)}\) is Borel. There are
constants \(a_3,A_3,A_5>0\) such that, for every \(c\in C\),
\begin{equation}
\label{eq:marked-quantizer-geometry}
\begin{aligned}
    \widetilde A_\rho(c)
    &\subset W_c\setminus\mathscr O_\rho(c),
    &
    \#\widetilde A_\rho(c)
    &=q_\rho(c),\\
    \operatorname{sep}(\widetilde A_\rho(c))
    &\geq a_3h,
    &
    \operatorname{dist}
        \bigl(
            \widetilde A_\rho(c),
            \partial\operatorname{Vor}(c)
        \bigr)
    &\geq a_3h,\\
    \sup_{y\in\operatorname{Vor}(c)}
        \operatorname{dist}
            (y,\widetilde A_\rho(c))
    &\leq A_3h.
\end{aligned}
\end{equation}
Furthermore,
\begin{equation}
\label{eq:marker-correction-region-unchanged}
    \widetilde A_\rho(c)\cap U_1
    =
    A_\rho(c)\cap U_1,
\end{equation}
and, for every \(c\in C\) and every Lipschitz \(f\) on
\(\overline B_{D_{\rm Rips}+R_{\rm Vor}}\),
\begin{equation}
\label{eq:marked-quantizer-Lipschitz}
    \left|
        \sum_{u\in \widetilde A_\rho(c)}f(u)
        -
        \rho\int_{\Pi_\rho(c)}f\,d\lambda_V
    \right|
    \leq
    A_5\rho h\operatorname{Lip}(f).
\end{equation}

For every \(c\in C\), every pair of distinct points
\(u,v\in \widetilde A_\rho(c)\) with
\(\{u,v\}\not\subset hS_{b(c)}\) satisfies
\[
    \|u-v\|\geq c_0h.
\]
Consequently, the graph on \(\widetilde A_\rho(c)\) with edges
\[
    \|u-v\|<\frac{c_0h}{5}
\]
has exactly one component with more than one vertex, namely
\(hS_{b(c)}\).
\end{proposition}

\begin{proof}
In \eqref{eq:marked-quantizer},
\(N_\rho(c)\) points are removed and
\[
    (d+1)+n_\rho(c)=N_\rho(c)
\]
points are inserted. Hence the cardinality is unchanged. All
inserted points belong to \(B_{r_*}\), while
\[
    \overline B_{r_*}
    \subset
    W_c\setminus\mathscr O_\rho(c)
\]
by \eqref{eq:packet-core-and-support}. This proves the first line of
\eqref{eq:marked-quantizer-geometry}.

Points of \(A_\rho(c)\setminus D_\rho\) remain \(a_2h\)-separated.
The points \(x_{\rm aux}+hT_{n_\rho(c)}\) are \(s_Th\)-separated
and have distance at least \(3Lh/4\) from
\(A_\rho(c)\setminus D_\rho\). By
\eqref{eq:marker-size-choice}, \(hS_{b(c)}\subset B_{h/4}(0)\), so
its distance from \(A_\rho(c)\setminus D_\rho\) is at least
\(3h/4\). Moreover,
\[
    \operatorname{dist}
        \bigl(hS_{b(c)},x_{\rm aux}+hT_{n_\rho(c)}\bigr)
    \geq
    \|x_{\rm aux}\|-\frac{L+1}{4}h
    \geq c_0h
\]
for sufficiently large \(\rho\). Since the
minimum distance between distinct points of \(S_s\) is at least
\(\delta\), uniformly in \(s\), the separation estimate follows.

For \(u\in B_{r_*}\),
\[
    \operatorname{dist}
        (u,\partial\operatorname{Vor}(c))
    \geq r_{\rm Vor}-r_*=r_*.
\]
Together with
\eqref{eq:initial-quantizer-geometry}, this proves the second
estimate in the second line of
\eqref{eq:marked-quantizer-geometry}, after increasing \(\rho_0\).

It remains to check the covering estimate. If the
\(A_2h\)-neighborhood of a point does not meet \(D_\rho\), the
estimate follows from
\eqref{eq:initial-quantizer-geometry}. Points within \(A_2h\) of
\(\overline B_h(0)\) have distance at most
\((A_2+1)h\) from \(0\in hS_{b(c)}\).

For the second ball, fix a unit vector \(e\). For sufficiently large
\(\rho\), the point
\[
    z_\rho
    :=
    x_{\rm aux}+(L+2A_2)he
\]
belongs to \(B_{r_*}\). Choose
\(u_\rho\in A_\rho(c)\) with
\(\|u_\rho-z_\rho\|\leq A_2h\). Then
\(u_\rho\notin D_\rho\). Thus every point within \(A_2h\) of
\(B_{Lh}(x_{\rm aux})\) has distance at most
\((2L+4A_2)h\) from \(u_\rho\). This proves the last line of
\eqref{eq:marked-quantizer-geometry}.

Equation
\eqref{eq:marker-correction-region-unchanged} follows from
\(D_\rho\cap U_1=\varnothing\) and from the fact that the inserted
points lie outside \(U_1\). Since at most \(N_*\) points are changed
and all changed points lie in \(B_{r_*}\),
\[
    \left|
        \sum_{u\in \widetilde A_\rho(c)}f(u)
        -
        \sum_{u\in A_\rho(c)}f(u)
    \right|
    \leq
    2r_*N_*\operatorname{Lip}(f).
\]
Together with \eqref{eq:initial-quantizer-Lipschitz}, this gives
\eqref{eq:marked-quantizer-Lipschitz}.

Finally, by \eqref{eq:marker-size-choice}, every two points of
\(hS_{b(c)}\) have distance less than \(c_0h/10\). Every other
pair of distinct points has distance at least \(c_0h\), after
increasing \(\rho_0\). This proves the last assertion.

The Borel assertion follows from the Borelness of
\(c\mapsto\delta_{A_\rho(c)}\), of \(b\), and of
\(c\mapsto N_\rho(c)\).
\end{proof}



\subsection{Moment correction}
\label{subsec:moment-correction}

For \(p=1\), the only moment condition in
Theorem~\ref{thm:local-filling} is the cardinality identity, so we put
\(Q_\rho(c):=\widetilde A_\rho(c)\). Assume from now on that \(p\geq2\).
The set \(hS_{b(c)}\) and all auxiliary points lie outside \(U_0\), and
\(\widetilde A_\rho(c)\cap U_1=A_\rho(c)\cap U_1\). We perturb only the
points in \(U_0\). The next argument shows that a perturbation of size
\(O(h)\) can make all nonconstant polynomial moments of degree at most
\(p-1\) exact.

Let \(\mathcal P_{p-1}^0\) be the vector space of real polynomial
functions on \(V\) of degree at most \(p-1\) that vanish at the
origin. Its dimension is
\[
    m:=\dim\mathcal P_{p-1}^0
      =\binom{d+p-1}{d}-1.
\]
Fix a basis \(P_1,\ldots,P_m\) of \(\mathcal P_{p-1}^0\), and define
\(\mathbf P:V\to\mathbb R^m\) by
\[
    \mathbf P(u):=(P_1(u),\ldots,P_m(u)).
\]
Choose \(\chi\in C_c^\infty(U_0)\) with \(\chi\geq0\) and
\(\chi>0\) on a nonempty open set, and define
\begin{equation}
\label{eq:correction-vector-fields}
    \mathcal X_a(u):=\chi(u)\nabla P_a(u),
    \qquad 1\leq a\leq m.
\end{equation}
For \(t=(t_1,\ldots,t_m)\in\mathbb R^m\), define
\(F_t:V\to V\) by
\begin{equation}
\label{eq:correction-map}
    F_t(u):=u+\sum_{a=1}^m t_a\mathcal X_a(u).
\end{equation}
Thus \(F_t(u)=u\) for \(u\notin U_0\).

Define
\begin{equation}
\label{eq:continuous-moment-Jacobian}
    J_{ba}
    :=
    \int_V
        \chi(u)\,
        \nabla P_b(u)\cdot\nabla P_a(u)
        \,d\lambda_V(u).
\end{equation}

\begin{lemma}[Positive definiteness of the moment Gram matrix]
\label{lem:continuous-moment-Jacobian}
The matrix \(J\) is positive definite.
\end{lemma}

\begin{proof}
For \(\xi=(\xi_1,\ldots,\xi_m)\in\mathbb R^m\), let
\(P_\xi:=\sum_a\xi_aP_a\). Then
\[
    \xi^{\mathsf T}J\xi
    =
    \int_V
        \chi(u)\|\nabla P_\xi(u)\|^2
        \,d\lambda_V(u).
\]
If this vanishes, then \(\nabla P_\xi\) vanishes on a nonempty open
set, hence everywhere. Thus \(P_\xi\) is constant. Since
\(P_\xi(0)=0\), it is zero, and therefore \(\xi=0\).
\end{proof}

The packet relation
\[
    \{(c,y)\in C\times V:y\in\Pi_\rho(c)\}
\]
is Borel, and all packets lie in the fixed ball
\(\overline B_{D_{\rm Rips}+R_{\rm Vor}}\). Hence parameter
integration shows that
\[
    c\longmapsto
    \int_{\Pi_\rho(c)}P(y)\,d\lambda_V(y)
\]
is Borel for every polynomial \(P\).

For \(c\in C\), define \(G_{\rho,c}:\mathbb R^m\to\mathbb R^m\) by
\begin{equation}
\label{eq:discrete-moment-map}
    G_{\rho,c}(t)
    :=
    \frac1\rho
    \left(
        \sum_{u\in \widetilde A_\rho(c)}
            \mathbf P(F_t(u))
        -
        \rho\int_{\Pi_\rho(c)}
            \mathbf P(y)\,d\lambda_V(y)
    \right).
\end{equation}
Thus \(G_{\rho,c}(t)\) is the normalized discrepancy between the polynomial
moments of the perturbed point set and those of the packet measure. By
\eqref{eq:marked-quantizer-Lipschitz}, applied to the components of
\(\mathbf P\), there is \(C_0<\infty\) such that
\begin{equation}
\label{eq:moment-initial-residual}
    \|G_{\rho,c}(0)\|\leq C_0h.
\end{equation}
Moreover,
\[
    DG_{\rho,c}(0)_{ba}
    =
    \frac1\rho
    \sum_{u\in \widetilde A_\rho(c)}
        \nabla P_b(u)\cdot\mathcal X_a(u).
\]
The functions
\(u\mapsto\nabla P_b(u)\cdot\mathcal X_a(u)\) are Lipschitz and
supported in \(U_0\). Since
\[
    U_0\subset B_{r_*}\subset\Pi_\rho(c),
\]
another application of
\eqref{eq:marked-quantizer-Lipschitz} gives
\begin{equation}
\label{eq:discrete-Jacobian-approximation}
    \|DG_{\rho,c}(0)-J\|\leq C_1h
\end{equation}
with \(C_1\) independent of \(c\) and \(\rho\).

Let \(\lambda_*>0\) be the least eigenvalue of \(J\). After
increasing \(\rho_0\),
\begin{equation}
\label{eq:discrete-Jacobian-inverse}
    \|DG_{\rho,c}(0)^{-1}\|
    \leq\frac{2}{\lambda_*}
    =:L_*.
\end{equation}
There are also \(\tau,C_2>0\), independent of \(c\) and \(\rho\),
such that
\begin{equation}
\label{eq:moment-second-derivative-bound}
    \sup_{\|t\|\leq\tau}
        \|D^2G_{\rho,c}(t)\|
    \leq C_2.
\end{equation}
Indeed, \(F_t\) is affine in \(t\), the Hessians of the \(P_a\) are
bounded on a fixed ball for \(\|t\|\leq\tau\), the fields
\(\mathcal X_a\) are bounded, and
\(q_\rho(c)\leq A_1\rho\) by \eqref{eq:q-rho-comparable}.

\begin{proposition}[Exact finite-dimensional moment correction]
\label{prop:exact-moment-correction}
After increasing \(\rho_0\), there are a constant \(C_3<\infty\) and a
Borel map \(c\mapsto t_\rho(c)\in\mathbb R^m\) such that
\begin{equation}
\label{eq:correction-parameter-bound}
    G_{\rho,c}(t_\rho(c))=0,
    \qquad
    \|t_\rho(c)\|\leq C_3h.
\end{equation}
If
\begin{equation}
\label{eq:corrected-local-configuration}
    Q_\rho(c)
    :=
    F_{t_\rho(c)}
        \bigl(\widetilde A_\rho(c)\bigr),
\end{equation}
then the conclusions of Theorem~\ref{thm:local-filling} hold.
\end{proposition}

\begin{proof}
\emph{Fixed point.}
Put \(J_{\rho,c}:=DG_{\rho,c}(0)\) and define
\[
    \mathcal N_{\rho,c}(t)
    :=
    t-J_{\rho,c}^{-1}G_{\rho,c}(t).
\]
Let \(C_3:=2L_*C_0\). For \(\|t\|\leq C_3h\),
\eqref{eq:moment-second-derivative-bound} gives
\[
    \|D\mathcal N_{\rho,c}(t)\|
    \leq L_*C_2C_3h.
\]
After increasing \(\rho_0\), this is at most \(1/2\). Also,
\eqref{eq:moment-initial-residual} gives
\[
    \|\mathcal N_{\rho,c}(0)\|
    \leq L_*C_0h
    =\frac{C_3h}{2}.
\]
Thus \(\mathcal N_{\rho,c}\) maps
\(\overline B_{C_3h}(0)\) into itself and is a contraction there.
Its unique fixed point \(t_\rho(c)\) satisfies
\eqref{eq:correction-parameter-bound}.

\emph{Borel dependence.}
The map \((c,t)\mapsto G_{\rho,c}(t)\) is Borel in \(c\) and
continuous in \(t\). The matrix \(J_{\rho,c}\) depends Borel
measurably on \(c\), and matrix inversion is continuous on the set of
invertible matrices. Hence \((c,t)\mapsto\mathcal N_{\rho,c}(t)\)
is Borel in \(c\) and continuous in \(t\). Starting with
\(t_0(c)=0\) and defining
\[
    t_{n+1}(c):=\mathcal N_{\rho,c}(t_n(c))
\]
gives Borel maps converging pointwise to \(t_\rho(c)\). Thus
\(c\mapsto t_\rho(c)\) is Borel.

\emph{Moment identities and geometry.}
There are constants \(L_{\mathcal X},M_{\mathcal X}<\infty\) such
that
\[
    \left\|D\sum_a t_a\mathcal X_a\right\|_\infty
    \leq L_{\mathcal X}\|t\|,
    \qquad
    \left\|\sum_a t_a\mathcal X_a\right\|_\infty
    \leq M_{\mathcal X}\|t\|.
\]
For sufficiently large \(\rho\),
\(L_{\mathcal X}\|t_\rho(c)\|\leq1/2\), and therefore
\begin{equation}
\label{eq:correction-bi-Lipschitz}
    \frac12\|u-v\|
    \leq
    \|F_{t_\rho(c)}(u)-F_{t_\rho(c)}(v)\|
    \leq
    \frac32\|u-v\|.
\end{equation}
After increasing \(\rho_0\) once more,
\[
    M_{\mathcal X}\|t_\rho(c)\|
    <
    \operatorname{dist}(\overline U_0,U_1^c),
\]
so \(F_{t_\rho(c)}(U_0)\subset U_1\). Points outside \(U_0\) are
fixed.

It follows from \eqref{eq:correction-bi-Lipschitz} that
\(F_{t_\rho(c)}\) is injective, and hence
\(\#Q_\rho(c)=q_\rho(c)\). Since
\(G_{\rho,c}(t_\rho(c))=0\),
\begin{equation}
\label{eq:corrected-basis-moments}
    \sum_{u\in Q_\rho(c)}P_a(u)
    =
    \rho\int_{\Pi_\rho(c)}
        P_a(y)\,d\lambda_V(y),
    \qquad 1\leq a\leq m.
\end{equation}
Together with the cardinality identity, this gives
\eqref{eq:local-filling-moments} for every
\(|\mathbf j|<p\).

Since
\[
    U_1\subset B_{r_*}
    \subset W_c\setminus\mathscr O_\rho(c)
\]
and points outside \(U_0\) are fixed,
Proposition~\ref{prop:marker-reservoir} gives
\[
    Q_\rho(c)\subset
    W_c\setminus\mathscr O_\rho(c).
\]
Equation \eqref{eq:correction-bi-Lipschitz} also gives
\[
    \operatorname{sep}(Q_\rho(c))
    \geq\frac{a_3}{2}h.
\]

Set
\[
    d_1:=
    \operatorname{dist}
        \bigl(\overline U_1,V\setminus B_{r_{\rm Vor}}\bigr)>0.
\]
Points moved by \(F_{t_\rho(c)}\) have distance at least \(d_1\)
from \(\partial\operatorname{Vor}(c)\). Every unchanged point inherited
from \(A_\rho(c)\) has boundary distance at least \(a_2h\) by
\eqref{eq:initial-quantizer-geometry}, while every inserted point lies in the fixed ball \(B_{r_*}\) and hence has boundary
distance at least \(r_*\). After increasing \(\rho_0\) so that
\(a_2h\leq\min\{d_1,r_*\}\), we therefore have the uniform estimate
\[
    \operatorname{dist}
        \bigl(Q_\rho(c),
              \partial\operatorname{Vor}(c)\bigr)
    \geq a_2h.
\]

Every point moves by at most \(M_{\mathcal X}C_3h\). Therefore
\eqref{eq:marked-quantizer-geometry} gives
\[
    \sup_{y\in\operatorname{Vor}(c)}
        \operatorname{dist}(y,Q_\rho(c))
    \leq
    (A_3+M_{\mathcal X}C_3)h.
\]
For every Lipschitz \(f\),
\[
    \left|
        \sum_{u\in Q_\rho(c)}f(u)
        -
        \sum_{u\in \widetilde A_\rho(c)}f(u)
    \right|
    \leq
    A_1M_{\mathcal X}C_3
        \rho h\operatorname{Lip}(f).
\]
Together with \eqref{eq:marked-quantizer-Lipschitz}, this proves
\eqref{eq:local-filling-Lipschitz}.

\emph{Isolation of \(hS_{b(c)}\) and Borelness.}
The set \(hS_{b(c)}\) is disjoint from \(U_0\) for sufficiently
large \(\rho\), and is therefore fixed pointwise. By
Proposition~\ref{prop:marker-reservoir}, every pair
\(u,v\in \widetilde A_\rho(c)\) with
\(\{u,v\}\not\subset hS_{b(c)}\) satisfies
\[
    \|u-v\|\geq c_0h.
\]
Hence \eqref{eq:correction-bi-Lipschitz} gives
\[
    \|F_{t_\rho(c)}(u)-F_{t_\rho(c)}(v)\|
    \geq\frac{c_0h}{2}.
\]
Since \(hS_{b(c)}\) is fixed and has diameter less than \(c_0h/10\),
the graph on \(Q_\rho(c)\) with edges of length less than
\(c_0h/5\) has \(hS_{b(c)}\) as its unique nontrivial component.
Thus \eqref{eq:local-filling-recognition-gap} holds with
\(c_{\rm rec}:=c_0/2\), and Lemma~\ref{lem:full-frame-marker} gives the
Borel recovery of \(a,r,s\).

For \(p\geq2\), we may therefore take, for example,
\[
    c_{\rm fill}:=a_3/2,
    \qquad
    c_{\rm bd}:=a_2,
    \qquad
    C_{\rm fill}:=
    \max\{A_3+M_{\mathcal X}C_3,
              A_5+A_1M_{\mathcal X}C_3\},
    \qquad
    c_{\rm rec}:=c_0/2.
\]
For \(p=1\), the same theorem holds directly from
Proposition~\ref{prop:marker-reservoir}. Indeed, the unchanged points of \(A_\rho(c)\) retain boundary distance \(a_2h\), while all inserted points have a
fixed positive boundary distance; after increasing \(\rho_0\), one may take
\[
    c_{\rm fill}=a_3,
    \qquad
    c_{\rm bd}=a_2,
    \qquad
    C_{\rm fill}=\max\{A_3,A_5\},
    \qquad
    c_{\rm rec}=c_0/2.
\]

Finally, \eqref{eq:corrected-local-configuration} and the Borelness
of \(c\mapsto t_\rho(c)\) show that
\(c\mapsto\delta_{Q_\rho(c)}\) is Borel.
\end{proof}

Proposition~\ref{prop:exact-moment-correction} completes the proof of
Theorem~\ref{thm:local-filling}.

\paragraph{Borel dependence.}
The measurable dependence in Theorem~\ref{thm:local-filling} follows through
the same steps as the construction. Proposition~\ref{prop:integer-mass-packets}
and its proof give Borel maps \(c\mapsto\Pi_\rho(c)\) and
\(c\mapsto q_\rho(c)\), while
Proposition~\ref{prop:labelled-Voronoi-partition} gives Borel dependence of
\(\operatorname{Vor}(c)\). The parameterized part of
Lemma~\ref{lem:equal-mass-rectangles} and
Lemma~\ref{lem:uniform-convex-quantization} show that
\((D,q)\mapsto\delta_{A_q(D)}\) is jointly Borel for the admissible compact
convex sets \(D\) and integers \(q\geq1\). Hence
\[
    c\longmapsto
    \delta_{A_{q_\rho(c)}(\operatorname{Vor}(c))}
\]
is Borel even though \(q_\rho(c)\) varies with \(c\).
Proposition~\ref{prop:marker-reservoir} preserves Borel dependence under the
finite modification inserting \(hS_{b(c)}\). Finally,
Proposition~\ref{prop:exact-moment-correction} obtains \(t_\rho(c)\) as the
pointwise limit of Borel Picard iterates and therefore gives Borel dependence
of \(Q_\rho(c)\). Thus every parameter-dependent step in the local
construction is Borel before the configurations are assembled globally.
For fixed \(p\), Sections~\ref{sec:orbit-geometry-integerization} and
\ref{sec:local-filling-assembly} impose only finitely many uniform lower
bounds on \(\rho\); their maximum is the final threshold \(\rho_0\) used
below.



\subsection{Global realization, intensity, and maximal rigidity}
\label{subsec:assembly-generation}

Fix \(\rho\geq\rho_0\), and let \(Q_\rho(c)\) be given by
Theorem~\ref{thm:local-filling}. We now verify the hypotheses of
Proposition~\ref{prop:configuration-realization}. For \(x\in X\), define
\begin{equation}
\label{eq:global-configuration}
    P_x
    :=
    \bigcup_{g\in C_x}
        \Phi_{x,g}
        \bigl(Q_\rho(\operatorname{lab}_x(g))\bigr).
\end{equation}
For brevity, when \(g\in C_x\), write
\[
    p_g:=\operatorname{pos}_x(g),
    \qquad
    r_g:=\operatorname{fr}_x(g),
    \qquad
    c_g:=\operatorname{lab}_x(g).
\]

\begin{proposition}[Global Delone realization and generation]
\label{prop:global-assembly}
The map
\[
    \kappa_\rho:X\longrightarrow\mathcal N_s(V),
    \qquad
    \kappa_\rho(x):=\delta_{P_x},
\]
is Borel and \(G_K\)-equivariant. There are constants
\(a,A>0\), independent of \(x\), such that
\begin{equation}
\label{eq:global-Delone-bounds}
    \operatorname{sep}(P_x)\geq ah,
    \qquad
    \sup_{y\in V}\operatorname{dist}(y,P_x)\leq Ah.
\end{equation}
Moreover, \(\kappa_\rho\) is injective.
\end{proposition}

\begin{proof}
\emph{Delone geometry.}
By Proposition~\ref{prop:labelled-Voronoi-partition}, the sets
\[
    \Phi_{x,g}(W_{c_g}),
    \qquad g\in C_x,
\]
form an exact partition of \(V\). Theorem~\ref{thm:local-filling} gives
\[
    \Phi_{x,g}(Q_\rho(c_g))
    \subset
    \Phi_{x,g}(W_{c_g}),
\]
so the union in \eqref{eq:global-configuration} is disjoint.

Two points in the same local configuration are separated by at least
\(c_{\rm fill}h\). If \(u\) and \(v\) lie in distinct physical Voronoi
cells, the segment \([u,v]\) meets the boundary of each cell. By
\eqref{eq:local-filling-geometry} and invariance under the isometries
\(\Phi_{x,g}\), both \(u\) and \(v\) have distance at least
\(c_{\rm bd}h\) from the corresponding boundary. Hence
\[
    \|u-v\|\geq2c_{\rm bd}h.
\]
Thus the separation estimate in \eqref{eq:global-Delone-bounds} holds with
\(a:=\min\{c_{\rm fill},2c_{\rm bd}\}\).

For \(y\in V\), choose \(g\in C_x\) such that
\[
    y\in\Phi_{x,g}(\operatorname{Vor}(c_g)).
\]
The covering estimate in \eqref{eq:local-filling-geometry}, transported by
\(\Phi_{x,g}\), gives
\[
    \operatorname{dist}(y,P_x)\leq C_{\rm fill}h.
\]
This proves \eqref{eq:global-Delone-bounds}.

\emph{Borelness and equivariance.}
The return relation may be enumerated by Borel maps
\(g_n\) on Borel subsets of \(X\). Since
\(c\mapsto\delta_{Q_\rho(c)}\) and the frame and position maps are
Borel, each corresponding summand in
\eqref{eq:global-configuration} depends Borel measurably on \(x\).
For a compact set \(D\subset V\), only returns satisfying
\[
    p_g\in D+\overline B_{R_{\rm Vor}}
\]
can contribute points to \(D\), and their number is uniformly
finite by the separation of \(\mathscr P_x\). Hence
\(x\mapsto\delta_{P_x}\) is Borel.

Let \(h_0=(v,\ell)\in G_K\). By
\eqref{eq:framed-return-covariance}, the map
\(g\mapsto gh_0^{-1}\) is a bijection from \(C_x\) onto
\(C_{h_0.x}\). The label is unchanged under this bijection, and
\eqref{eq:framed-transport-identity} gives
\[
    \Phi_{h_0.x,gh_0^{-1}}(y)
    =\ell\Phi_{x,g}(y)-v.
\]
Applying this identity to \eqref{eq:global-configuration} gives
\begin{equation}
\label{eq:global-configuration-covariance}
    P_{(v,\ell).x}=\ell P_x-v
    \qquad(x\in X).
\end{equation}

\emph{Injectivity.}
Let \(c_{\rm rec}\) be the constant from
Theorem~\ref{thm:local-filling}. By
\eqref{eq:local-filling-recognition-gap}, the set \(hS_{b(c)}\) has diameter less than \(c_{\rm rec}h/5\), and
\begin{equation}
\label{eq:nonmarker-distance}
    \|u-v\|\geq c_{\rm rec}h
\end{equation}
whenever \(u,v\in Q_\rho(c)\) and
\(\{u,v\}\not\subset hS_{b(c)}\).

By \eqref{eq:local-filling-geometry}, points belonging to distinct
Voronoi cells have distance at least \(2c_{\rm bd}h\). With the choices at
the end of Proposition~\ref{prop:exact-moment-correction},
\(c_{\rm bd}=a_2\) and
\(c_{\rm rec}=c_0/2\leq a_2/2\), so in particular
\[
    2c_{\rm bd}h\geq4c_{\rm rec}h.
\]
It follows that the connected
components with more than one vertex of the graph on \(P_x\) with
edges
\begin{equation}
\label{eq:global-marker-threshold}
    \|u-v\|<\frac{c_{\rm rec}h}{2}
\end{equation}
are precisely the sets
\begin{equation}
\label{eq:global-markers}
    p_g+r_ghS_{b(c_g)},
    \qquad g\in C_x.
\end{equation}

By Lemma~\ref{lem:full-frame-marker}, an unlabelled set in
\eqref{eq:global-markers} determines \(p_g\), \(r_g\), and \(b(c_g)\) by
a Borel map. Since \(b\) is injective, the set also determines \(c_g\).
Equation~\eqref{eq:framed-recovery} then gives
\[
    x=(-p_g,r_g).c_g.
\]
Thus \(P_x\) determines \(x\), and \(\kappa_\rho\) is injective.
\end{proof}

Proposition~\ref{prop:global-assembly} verifies the hypotheses of
Proposition~\ref{prop:configuration-realization}, including injectivity. Hence,
with
\begin{equation}
\label{eq:higher-dimensional-cross-section}
    Y_\rho:=\{x\in X:0\in P_x\},
    \qquad
    \eta_\rho:=(\kappa_\rho)_*\mu,
\end{equation}
Proposition~\ref{prop:configuration-realization} shows that \(Y_\rho\) is a
\(K\)-invariant generating Delone translation cross-section,
\begin{equation}
\label{eq:higher-dimensional-return-set}
    (Y_\rho)_x=P_x,
\end{equation}
and its return-time process is \(\eta_\rho\).

\begin{proposition}[Exact prescribed intensity]
\label{prop:higher-dimensional-exact-intensity}
The point process \(\eta_\rho\) has intensity \(\rho\).
\end{proposition}

\begin{proof}
Let
\[
    R_{\rm pkt}:=D_{\rm Rips}+R_{\rm Vor}.
\]
By \eqref{eq:packet-core-and-support}, the physical packet corresponding to
\(g\in C_x\) is contained in
\[
    p_g+\overline B_{R_{\rm pkt}}.
\]
Moreover,
\[
    \#Q_\rho(c_g)
    =
    \rho\lambda_V(\Pi_\rho(c_g))
    =
    q_\rho(c_g),
\]
and \eqref{eq:q-rho-comparable} gives
\(q_\rho(c_g)\leq A_1\rho\).

Consider \(B_R\), with \(R\geq2R_{\rm pkt}\). For each \(g\in C_x\),
the signed measure
\[
    \delta_{p_g+r_gQ_\rho(c_g)}
    -\rho\lambda_V|_{p_g+r_g\Pi_\rho(c_g)}
\]
has total mass zero and is supported in
\(p_g+\overline B_{R_{\rm pkt}}\). Hence it contributes zero to
\[
    \#(P_x\cap B_R)-\rho\lambda_V(B_R)
\]
whenever \(p_g+\overline B_{R_{\rm pkt}}\) is contained in \(B_R\) or
disjoint from it. Every remaining packet therefore has
\[
    \bigl|\|p_g\|-R\bigr|\leq R_{\rm pkt}.
\]
Since the set of positions is \(r_{\rm sep}\)-separated, a packing
estimate in this annulus gives
\[
    \#\left\{
        g\in C_x:
        \bigl|\|p_g\|-R\bigr|\leq R_{\rm pkt}
    \right\}
    \leq C R^{d-1},
\]
where \(C\) is independent of \(x\) and \(R\). Each such packet
contributes at most \(2A_1\rho\) in absolute value. Therefore
\begin{equation}
\label{eq:higher-dimensional-boundary-discrepancy}
    \left|
        \#(P_x\cap B_R)-\rho\lambda_V(B_R)
    \right|
    \leq C'\rho R^{d-1}.
\end{equation}
This ball estimate is used only to identify the mean intensity; the
stronger uniform boundary-discrepancy statement is proved in
Subsection~\ref{subsec:boundary-discrepancy-bd} from the first-order derivative
representation.

Let \(\rho_{\eta_\rho}\) be the intensity of \(\eta_\rho\).
Taking expectations in
\eqref{eq:higher-dimensional-boundary-discrepancy} gives
\[
    |\rho_{\eta_\rho}-\rho|\lambda_V(B_R)
    \leq C'\rho R^{d-1}.
\]
Dividing by \(\lambda_V(B_R)\) and letting \(R\to\infty\) proves
\(\rho_{\eta_\rho}=\rho\).
\end{proof}



\begin{proposition}[Maximal rigidity]
\label{prop:higher-dimensional-maximal-rigidity}
The point process \(\eta_\rho\) is maximally rigid. More precisely,
for every bounded Borel set \(A\subset V\), there is a Borel recovery
map from configurations on \(A^c\) to \(\mathcal N_s(V)\) which
recovers \(\delta_{P_x}\) from \(P_x\cap A^c\) for every \(x\) in the
invariant conull set on which the construction is defined.
\end{proposition}

\begin{proof}
Fix a bounded Borel set \(A\subset V\), let \(c_{\rm rec}\) be the
constant from Theorem~\ref{thm:local-filling}, and put
\[
    \theta:=c_{\rm rec}h,
    \qquad
    E_A:=\{u\in V:\operatorname{dist}(u,A)>\theta\}.
\]
From the outside configuration \(P_x\cap A^c\) one can recover
\(P_x\cap E_A\). Form on this set the graph whose edges join pairs at
distance less than \(\theta/2\).

In the full configuration, Proposition~\ref{prop:global-assembly}
shows that the nontrivial components of this graph are exactly the sets
\[
    p_g+r_ghS_{b(c_g)},
    \qquad g\in C_x,
\]
and every such set has exactly \(d+1\) vertices and diameter less than
\(\theta/5\). Restriction to \(E_A\) can delete vertices and edges but
cannot create an edge. Consequently, every component of cardinality
\(d+1\) in the restricted graph is one of these complete \(d+1\)-point sets.

Such a component always exists. Indeed, the set of positions
\(\{p_g:g\in C_x\}\) is relatively dense and therefore has points
arbitrarily far from the bounded set \(A\). Choose \(g\in C_x\) with
\[
    \operatorname{dist}(p_g,A)>\theta+h.
\]
By \eqref{eq:marker-size-choice}, the set \(p_g+r_ghS_{b(c_g)}\) lies in the
\(h/4\)-neighborhood of \(p_g\), so it is contained in
\(E_A\); in fact infinitely many such complete sets exist.

The relation consisting of pairs \((\omega,M)\), where \(\omega\) is an
outside configuration, \(M\) is a \((d+1)\)-point component of the graph just
defined on \(\omega|_{E_A}\), and \(\delta_M\in\mathscr M_h\), is Borel and has
countable sections. Its projection contains every
outside configuration arising from the process. Lusin--Novikov therefore
shows that the projection is Borel and provides a Borel choice of one such
component there; extend the choice arbitrarily off the projection.
Lemma~\ref{lem:full-frame-marker} recovers from the chosen component the
position \(p_g\), the frame \(r_g\), and the parameter \(b(c_g)\). By
Lusin--Souslin, \(b(C)\) is Borel and the inverse of the Borel injection
\(b\) on its image is Borel. Hence \(c_g\) is recovered by a Borel map, and
\eqref{eq:framed-recovery} gives
\[
    x=(-p_g,r_g).c_g.
\]
Applying the Borel map \(\kappa_\rho\) to the recovered state gives the
entire configuration \(\delta_{P_x}\). This is the required Borel
recovery map.
\end{proof}

\section{Consequences of the local moment identities}
\label{sec:spectral-decay-rigidity}

We retain the integer \(p\geq1\) and the standing higher-dimensional setup of
Sections~\ref{sec:orbit-geometry-integerization} and
\ref{sec:local-filling-assembly}. Fix \(\rho\geq\rho_0\), and let \(\eta=\eta_\rho\) and \(P_x\) be the point
process and configurations constructed in
Section~\ref{sec:local-filling-assembly}. Let \(X_0\subset X\) be the
\(G_K\)-invariant conull Borel set on which the construction is defined.

For \(c\in C\), set
\[
    \Delta_c
    :=
    \delta_{Q_\rho(c)}
    -
    \rho\lambda_V|_{\Pi_\rho(c)}.
\]
By Theorem~\ref{thm:local-filling}, \(\Delta_c\) is supported in a fixed ball
and
\[
    \int_V P\,d\Delta_c=0
    \qquad
    (\deg P<p).
\]
Placed in the orbit coordinates from Section~\ref{sec:orbit-geometry-integerization},
these measures sum to \(\delta_{P_x}-\rho\lambda_V\).

The order-one consequence of this cancellation gives the boundary discrepancy
needed for bounded displacement to a lattice. The full cancellation through
degree \(p-1\) gives an order-\(p\) distributional derivative identity, from
which we prove
\[
    \sigma_\eta(B_\varepsilon)=o(\varepsilon^{2p})
    \qquad(\varepsilon\downarrow0).
\]
We then deduce surface-order ball variance and finite-order linear rigidity.

Recall that \(C_x=\{g\in G_K:g.x\in C\}\). Random variables on
\((\mathcal N_s(V),\eta)\) will be identified with their
pullbacks under \(\kappa_\rho\) to \((X,\mu)\).


\subsection{Distributional derivative representation}
\label{subsec:packet-errors-primitives}

Put
\[
    R_{\rm pkt}:=D_{\rm Rips}+R_{\rm Vor}.
\]
Theorem~\ref{thm:local-filling} gives
\begin{equation}
\label{eq:canonical-packet-cancellation}
    \int_V P(u)\,d\Delta_c(u)=0
\end{equation}
for every polynomial \(P\) of degree less than \(p\), and
\begin{equation}
\label{eq:canonical-packet-error-bounds}
    \operatorname{supp}\Delta_c
    \subset\overline B_{R_{\rm pkt}},
    \qquad
    \|\Delta_c\|_{\rm TV}
    \leq2A_1\rho.
\end{equation}

For \(x\in X\) and \(g\in C_x\), write
\[
    p_g:=\operatorname{pos}_x(g),
    \qquad
    r_g:=\operatorname{fr}_x(g),
    \qquad
    c_g:=\operatorname{lab}_x(g),
\]
and let
\begin{equation}
\label{eq:physical-packet-error}
    \Delta_{x,g}
    :=
    (u\mapsto p_g+r_gu)_*\Delta_{c_g}.
\end{equation}
Since the physical packets partition \(V\),
\begin{equation}
\label{eq:global-error-packet-sum}
    \Delta_x
    :=
    \delta_{P_x}-\rho\lambda_V
    =
    \sum_{g\in C_x}\Delta_{x,g}
\end{equation}
as locally finite signed measures.

Taylor's formula gives the distributional identity used below. For a multi-index
\(\mathbf j=(j_1,\ldots,j_d)\), write
\[
    \mathbf j!:=j_1!\cdots j_d!.
\]

\begin{lemma}[Derivative representation for moment-canceling measures]
\label{lem:Taylor-measure-primitives}
Let \(p\geq1\), \(R>0\), and let \(\nu\) be a finite signed measure supported
in \(\overline B_R\). Suppose that
\[
    \int_V P\,d\nu=0
\]
for every polynomial \(P\) of degree less than \(p\). For
\(|\mathbf j|=p\), define a finite signed measure
\(\Theta_{\mathbf j}(\nu)\) by
\begin{equation}
\label{eq:Taylor-primitive-definition}
\begin{aligned}
    \int_V\psi(y)\,d\Theta_{\mathbf j}(\nu)(y)
    &:=
    (-1)^p\frac{p}{\mathbf j!}
    \int_V\int_0^1
        (1-t)^{p-1}u^{\mathbf j}\psi(tu)
        \,dt\,d\nu(u)
\end{aligned}
\end{equation}
for \(\psi\in C_c(V)\). Then
\begin{equation}
\label{eq:Taylor-primitive-identity}
    \nu
    =
    \sum_{|\mathbf j|=p}
        \partial^{\mathbf j}\Theta_{\mathbf j}(\nu)
\end{equation}
as distributions. Furthermore,
\begin{equation}
\label{eq:Taylor-primitive-bounds}
    \operatorname{supp}\Theta_{\mathbf j}(\nu)
    \subset\overline B_R,
    \qquad
    \|\Theta_{\mathbf j}(\nu)\|_{\rm TV}
    \leq
    \frac{R^p}{\mathbf j!}\|\nu\|_{\rm TV}.
\end{equation}
\end{lemma}

\begin{proof}
For \(\psi\in C_c^\infty(V)\), Taylor's formula along the segment
\(t\mapsto tu\) gives
\begin{equation}
\label{eq:Taylor-integral-remainder}
\begin{aligned}
    \psi(u)
    &=
    \sum_{|\boldsymbol\beta|<p}
        \frac{
            \partial^{\boldsymbol\beta}\psi(0)
        }{
            \boldsymbol\beta!
        }
        u^{\boldsymbol\beta}\\
    &\quad+
    p\sum_{|\mathbf j|=p}
        \frac{u^{\mathbf j}}{\mathbf j!}
        \int_0^1
            (1-t)^{p-1}
            \partial^{\mathbf j}\psi(tu)
        \,dt .
\end{aligned}
\end{equation}
After integration against \(\nu\), the first sum vanishes. By the
definition of distributional derivatives,
\[
    \left\langle
        \partial^{\mathbf j}\Theta_{\mathbf j}(\nu),
        \psi
    \right\rangle
    =
    (-1)^p
    \int_V
        \partial^{\mathbf j}\psi
        \,d\Theta_{\mathbf j}(\nu),
\]
and \eqref{eq:Taylor-primitive-identity} follows from
\eqref{eq:Taylor-primitive-definition} and
\eqref{eq:Taylor-integral-remainder}.

Since \(tu\in\overline B_R\) whenever
\(u\in\overline B_R\) and \(0\leq t\leq1\), the support assertion
follows. If \(\|\psi\|_\infty\leq1\), then
\[
\begin{aligned}
    \left|
        \int_V\psi\,d\Theta_{\mathbf j}(\nu)
    \right|
    &\leq
    \frac{p}{\mathbf j!}
    \int_V
        |u^{\mathbf j}|
        \int_0^1(1-t)^{p-1}\,dt
        \,d|\nu|(u)\\
    &\leq
    \frac{R^p}{\mathbf j!}\|\nu\|_{\rm TV},
\end{aligned}
\]
which proves \eqref{eq:Taylor-primitive-bounds}.
\end{proof}

For \(x\in X\) and \(g\in C_x\), translate \(\Delta_{x,g}\) by \(-p_g\):
\begin{equation}
\label{eq:centered-physical-packet-error}
    \widetilde\Delta_{x,g}
    :=
    (u\mapsto u-p_g)_*\Delta_{x,g}
    =
    (r_g)_*\Delta_{c_g}.
\end{equation}
The measure \(\widetilde\Delta_{x,g}\) is supported in
\(\overline B_{R_{\rm pkt}}\) and annihilates every polynomial of degree less
than \(p\). For \(|\mathbf j|=p\), define
\begin{equation}
\label{eq:packet-primitive}
    \Theta_{\mathbf j,x,g}
    :=
    (u\mapsto p_g+u)_*
    \Theta_{\mathbf j}
        (\widetilde\Delta_{x,g}).
\end{equation}
Lemma~\ref{lem:Taylor-measure-primitives} gives
\begin{equation}
\label{eq:physical-packet-primitive-identity}
    \Delta_{x,g}
    =
    \sum_{|\mathbf j|=p}
        \partial^{\mathbf j}\Theta_{\mathbf j,x,g}.
\end{equation}
It also gives a constant \(C_\Theta<\infty\), independent of \(x\), \(g\), and
\(\rho\geq\rho_0\), such that
\begin{equation}
\label{eq:physical-packet-primitive-bounds}
    \operatorname{supp}\Theta_{\mathbf j,x,g}
    \subset
    p_g+\overline B_{R_{\rm pkt}},
    \qquad
    \|\Theta_{\mathbf j,x,g}\|_{\rm TV}
    \leq C_{\Theta}\rho.
\end{equation}

For \(x\in X\) and \(|\mathbf j|=p\), define
\begin{equation}
\label{eq:global-primitives}
    \Theta_{\mathbf j,x}
    :=
    \sum_{g\in C_x}\Theta_{\mathbf j,x,g}.
\end{equation}
The sum is locally finite. Indeed, a summand can meet a bounded set
\(D\subset V\) only if
\[
    p_g\in D+\overline B_{R_{\rm pkt}},
\]
and the positions \(p_g\) form an \(r_{\rm sep}\)-separated set.
Consequently, for every \(|\mathbf j|=p\), there is a constant
\(C'_\Theta<\infty\), independent of \(\rho\geq\rho_0\), such that
\begin{equation}
\label{eq:global-primitive-local-bound}
    \sup_{x\in X}\sup_{z\in V}
        |\Theta_{\mathbf j,x}|(B_1(z))
    \leq C'_\Theta\rho.
\end{equation}

Summing
\eqref{eq:physical-packet-primitive-identity} and using
\eqref{eq:global-error-packet-sum} gives
\begin{equation}
\label{eq:global-primitive-identity}
    \Delta_x
    =
    \sum_{|\mathbf j|=p}
        \partial^{\mathbf j}\Theta_{\mathbf j,x}
\end{equation}
as distributions.

For a signed Radon measure \(\nu\), write
\(\tau_v\nu:=(u\mapsto u-v)_*\nu\). If \(x\in X\) and \(g\in C_x\), then
\(g(v,e)^{-1}\in C_{(v,e).x}\), with
\[
    p_{g(v,e)^{-1}}=p_g-v,
    \qquad
    r_{g(v,e)^{-1}}=r_g,
    \qquad
    c_{g(v,e)^{-1}}=c_g.
\]
It follows from the definitions that
\begin{equation}
\label{eq:primitive-translation-covariance}
    \Theta_{\mathbf j,(v,e).x}
    =
    \tau_v\Theta_{\mathbf j,x},
    \qquad v\in V,\quad x\in X,\quad |\mathbf j|=p.
\end{equation}
The maps \(x\mapsto\Theta_{\mathbf j,x}\) are Borel: the measures
\(\Delta_{x,g}\) depend Borel measurably on \((x,g)\),
\eqref{eq:Taylor-primitive-definition} preserves Borel dependence, and the
sums in \eqref{eq:global-primitives} are locally finite.


\subsection{Boundary discrepancy and bounded displacement}
\label{subsec:boundary-discrepancy-bd}

Only the identity \(\Delta_c(V)=0\) is needed here. Using the orthonormal
basis fixed in
Subsection~\ref{subsec:Euclidean-Fourier-conventions}, identify
\(V\cong\mathbb R^d\) and use the integer grid in these coordinates.

\begin{proposition}[Uniform boundary discrepancy]
\label{prop:uniform-boundary-discrepancy}
There is \(C_\rho<\infty\) such that, for every \(x\in X_0\) and every
finite union \(U\) of unit cubes,
\begin{equation}
\label{eq:Laczkovich-discrepancy}
    \left|
        \#(P_x\cap U)-\rho\lambda_V(U)
    \right|
    \leq
    C_\rho\,\mathcal H^{d-1}(\partial U).
\end{equation}
\end{proposition}

\begin{proof}
Constants in this proof may depend on the fixed intensity \(\rho\). Applying
Lemma~\ref{lem:Taylor-measure-primitives} with order \(1\) to the measures
\(\widetilde\Delta_{x,g}\) and summing gives signed Radon measures
\(\Xi_{1,x},\ldots,\Xi_{d,x}\) such that
\begin{equation}
\label{eq:first-order-balance}
    \delta_{P_x}-\rho\lambda_V
    =
    \sum_{i=1}^d\partial_i\Xi_{i,x}
\end{equation}
as distributions and
\begin{equation}
\label{eq:first-order-balance-local-bound}
    \sup_{x\in X}\sup_{z\in V}|\Xi_{i,x}|(B_1(z))<\infty
    \qquad(1\leq i\leq d).
\end{equation}

Fix a nonnegative mollifier \(\varphi\in C_c^\infty(B_{1/10})\) of
integral one and put \(f_U:=\mathbf 1_U*\varphi\). Then
\(f_U-\mathbf 1_U\) and \(\nabla f_U\) are supported in
\[
    \mathcal N_{1/10}(\partial U)
    :=\{y:\operatorname{dist}(y,\partial U)<1/10\},
\]
and \(\|\nabla f_U\|_\infty\) is bounded independently of \(U\).
The distributional identity \eqref{eq:first-order-balance} and
\eqref{eq:first-order-balance-local-bound} therefore give
\[
    |(\delta_{P_x}-\rho\lambda_V)(f_U)|
    \leq C\,\mathcal H^{d-1}(\partial U).
\]
Indeed, up to lower-dimensional intersections, \(\partial U\) is the
union of its exposed unit faces in the integer grid. Each such face has
\(\mathcal H^{d-1}\)-measure one, and its \(1/10\)-neighborhood can be
covered by a number of unit balls depending only on \(d\). The same covering,
together with the uniform separation of \(P_x\), gives
\[
    |(\delta_{P_x}-\rho\lambda_V)(\mathbf 1_U-f_U)|
    \leq C'\,\mathcal H^{d-1}(\partial U).
\]
Combining the two estimates proves
\eqref{eq:Laczkovich-discrepancy}.
\end{proof}

Put \(\Lambda_\rho:=\rho^{-1/d}\mathbb Z^d\subset V\).

\begin{proposition}[Bounded displacement to a lattice]
\label{prop:bounded-displacement-lattice}
For every \(x\in X_0\), the Delone set \(P_x\) is bounded-displacement
equivalent
to \(\Lambda_\rho\). In particular, \(P_x\) is bi-Lipschitz equivalent
to \(\Lambda_\rho\).
\end{proposition}

\begin{proof}
Proposition~\ref{prop:uniform-boundary-discrepancy} is precisely
condition~(ii) in Laczkovich's bounded-displacement criterion, in the form
recorded in \cite[Theorem~1.4]{Sol14} from \cite{Lac92}. The criterion tests
finite unions of unit cubes in the integer grid and, with \(\beta=\rho\),
yields bounded-displacement equivalence to \(\rho^{-1/d}\mathbb Z^d\). Hence
there is a bijection
\[
    \phi:P_x\longrightarrow\Lambda_\rho
\]
and a constant \(D<\infty\) such that
\(\|\phi(u)-u\|\leq D\) for all \(u\in P_x\).

If \(r_P\) and \(r_\Lambda\) are separation constants for \(P_x\) and
\(\Lambda_\rho\), then the bound \(\|\phi(u)-u\|\leq D\) gives Lipschitz
constants at most \(1+2D/r_P\) for \(\phi\) and
\(1+2D/r_\Lambda\) for \(\phi^{-1}\). Thus \(\phi\) is bi-Lipschitz. No
measurable or translation-equivariant choice of \(\phi\) is asserted.
\end{proof}



\subsection{Bartlett spectral decay}
\label{subsec:Bartlett-spectral-decay}

Let
\[
    \mathcal J_p:=\{\mathbf j\in\mathbb N_0^d:|\mathbf j|=p\},
    \qquad
    m_p:=\#\mathcal J_p.
\]
A translation-covariant Borel family of signed Radon measures with a uniform
local variation bound has a covariance spectral measure as follows.

\begin{lemma}[Scalar spectrum of a covariant signed random measure]
\label{lem:signed-random-measure-spectrum}
Let \(x\mapsto\Psi_x\) be a Borel family of signed Radon measures on
\(V\) satisfying
\[
    \Psi_{(v,e).x}=\tau_v\Psi_x
    \qquad(v\in V,\ x\in X),
\]
and suppose that
\[
    \sup_{x\in X}\sup_{z\in V}|\Psi_x|(B_1(z))<\infty.
\]
Then there is a unique positive tempered Radon measure
\(\Sigma_\Psi\) such that
\begin{equation}
\label{eq:signed-measure-spectral-covariance}
    \int_X
        \Psi_x(f)^\circ
        \overline{\Psi_x(g)^\circ}
        \,d\mu(x)
    =
    \int_V
        \widehat f(\xi)\overline{\widehat g(\xi)}
        \,d\Sigma_\Psi(\xi)
\end{equation}
for all \(f,g\in\mathcal S(V)\). If the restricted \(V\)-action is
ergodic, then
\begin{equation}
\label{eq:signed-spectrum-no-zero-atom}
    \Sigma_\Psi(\{0\})=0.
\end{equation}
\end{lemma}

\begin{proof}
Let \(M:=\sup_{x\in X}\sup_{z\in V}|\Psi_x|(B_1(z))\). A lattice cover of \(V\)
gives, for every \(N>d\), a constant \(C_N<\infty\) such that
\[
    \sup_{x\in X}|\Psi_x(f)|
    \leq
    C_N M\sup_{u\in V}(1+\|u\|)^N|f(u)|,
    \qquad f\in\mathcal S(V).
\]
Write \(\Psi(f)(x):=\Psi_x(f)\). Then
\(f\mapsto\Psi(f)^\circ\) is continuous from
\(\mathcal S(V)\) to the mean-zero subspace of \(L^2(X,\mu)\). For \(f,g\in\mathcal S(V)\), put
\[
    B(f,g)
    :=
    \int_X\Psi_x(f)^\circ
        \overline{\Psi_x(g)^\circ}\,d\mu(x).
\]
The form \(B\) is continuous, positive semidefinite, and sesquilinear. For
\(v,u\in V\) and \(f\in\mathcal S(V)\), write
\[
    T_vf(u):=f(u-v).
\]
Translation covariance gives \(B(T_vf,T_vg)=B(f,g)\). By the Schwartz
kernel theorem, the covariance form is represented by a tempered distribution
on \(V\times V\). Diagonal translation invariance
makes this distribution invariant under
\((u,w)\mapsto(u+v,w+v)\). Passing to the coordinates
\((r,w)=(u-w,w)\), the kernel is invariant under every translation in the
second variable. Hence each of its distributional derivatives in the
\(w\)-variables vanishes. The standard distributional fact that a distribution on \(V\) annihilated
by all first derivatives is constant, applied in the \(w\)-variable, shows
that the kernel is of the form \(C_\Psi(r)\otimes 1(w)\) for a tempered
distribution \(C_\Psi\) on \(V\). Substituting this form
into the kernel pairing gives the convolution factorization
\[
    B(f,g)=C_\Psi(f*\widetilde g),
    \qquad
    \widetilde g(u):=\overline{g(-u)}.
\]
The positivity of \(B\) says that \(C_\Psi\) is positive definite.
The Bochner--Schwartz theorem therefore gives a unique positive
tempered Radon measure \(\Sigma_\Psi\) with
\[
    C_\Psi(h)=\int_V\widehat h(\xi)\,d\Sigma_\Psi(\xi),
    \qquad h\in\mathcal S(V).
\]
Since
\(\widehat{f*\widetilde g}=\widehat f\,\overline{\widehat g}\),
this is exactly \eqref{eq:signed-measure-spectral-covariance}.

If the \(V\)-action is ergodic, choose \(f\in\mathcal S(V)\) with
\(\widehat f(0)\neq0\). With the Koopman convention of
Subsection~\ref{subsec:Euclidean-actions}, translation covariance and
\eqref{eq:signed-measure-spectral-covariance} give, for \(v\in V\),
\[
\begin{aligned}
    \bigl\langle
        \pi_X((v,e))\Psi(f)^\circ,
        \Psi(f)^\circ
    \bigr\rangle
    &=B(T_{-v}f,f)\\
    &=\int_V e^{2\pi i\langle v,\xi\rangle}
        |\widehat f(\xi)|^2\,d\Sigma_\Psi(\xi).
\end{aligned}
\]
Hence the Koopman spectral measure of \(\Psi(f)^\circ\) is
\(|\widehat f|^2\Sigma_\Psi\). Its atom at zero is the squared norm of
the projection onto the invariant subspace. Ergodicity makes that
subspace the constants, while \(\Psi(f)^\circ\) is centered; hence the
atom vanishes. Since \(\widehat f(0)\neq0\),
\eqref{eq:signed-spectrum-no-zero-atom} follows.
\end{proof}

For \(\mathbf j\in\mathcal J_p\), apply
Lemma~\ref{lem:signed-random-measure-spectrum} to
\(x\mapsto\Theta_{\mathbf j,x}\), and denote the resulting measure by
\(\Sigma_{\mathbf j}\). For \(f\in\mathcal S(V)\), write
\(\Theta_{\mathbf j}(f)(x):=\Theta_{\mathbf j,x}(f)\). The mean measure
\(\overline\Theta_{\mathbf j}\), defined by
\[
    \overline\Theta_{\mathbf j}(A)
    :=
    \int_X\Theta_{\mathbf j,x}(A)\,d\mu(x),
\]
is translation invariant and locally finite, and hence is a constant
multiple of \(\lambda_V\). Since \(|\mathbf j|=p\geq1\),
\(\int_V\partial^{\mathbf j}f\,d\lambda_V=0\) for
\(f\in\mathcal S(V)\). Therefore
\begin{equation}
\label{eq:derivative-statistics-centered}
    \Theta_{\mathbf j}(\partial^{\mathbf j}f)^\circ
    =
    \Theta_{\mathbf j}(\partial^{\mathbf j}f)
    \qquad(f\in\mathcal S(V)).
\end{equation}

\begin{proposition}[Bartlett decay]
\label{prop:Bartlett-spectral-decay}
The Bartlett spectrum of \(\eta\) satisfies
\begin{equation}
\label{eq:Bartlett-big-O}
    \sigma_\eta(B_\varepsilon)
    =
    O(\varepsilon^{2p})
    \qquad(\varepsilon\downarrow0).
\end{equation}
Under the standing assumption that the restricted \(V\)-action on
\((X,\mu)\) is ergodic,
\begin{equation}
\label{eq:Bartlett-little-o}
    \sigma_\eta(B_\varepsilon)
    =
    o(\varepsilon^{2p})
    \qquad(\varepsilon\downarrow0).
\end{equation}
\end{proposition}

\begin{proof}
For \(f\in\mathcal S(V)\) and \(x\in X\),
\eqref{eq:global-primitive-identity} and
\eqref{eq:derivative-statistics-centered} give
\[
    (\bS f)^\circ(x)
    =
    \Delta_x(f)
    =
    (-1)^p
    \sum_{\mathbf j\in\mathcal J_p}
        \Theta_{\mathbf j,x}(\partial^{\mathbf j}f).
\]
Hence, by Cauchy--Schwarz in the finite sum and
Lemma~\ref{lem:signed-random-measure-spectrum},
\begin{equation}
\label{eq:scalar-spectral-upper-bound}
\begin{aligned}
    \int_V|\widehat f|^2\,d\sigma_\eta
    &\leq
    m_p\sum_{\mathbf j\in\mathcal J_p}
        \int_V
        |(2\pi i\xi)^{\mathbf j}|^2
        |\widehat f(\xi)|^2
        \,d\Sigma_{\mathbf j}(\xi).
\end{aligned}
\end{equation}

Choose \(\varphi\in C_c^\infty(V)\) with
\(0\leq\varphi\leq1\), \(\varphi=1\) on \(B_1\), and
\(\operatorname{supp}\varphi\subset B_2\). For \(\varepsilon>0\),
let \(f_\varepsilon\in\mathcal S(V)\) be defined by
\[
    \widehat f_\varepsilon(\xi)
    :=\varphi(\xi/\varepsilon).
\]
Then the left-hand side of
\eqref{eq:scalar-spectral-upper-bound} is at least
\(\sigma_\eta(B_\varepsilon)\), while on
\(\operatorname{supp}\widehat f_\varepsilon\) one has
\[
    |(2\pi i\xi)^{\mathbf j}|^2
    \leq
    (4\pi\varepsilon)^{2p}.
\]
Consequently,
\begin{equation}
\label{eq:scalar-Bartlett-trace-bound}
    \sigma_\eta(B_\varepsilon)
    \leq
    m_p(4\pi)^{2p}\varepsilon^{2p}
    \sum_{\mathbf j\in\mathcal J_p}
        \Sigma_{\mathbf j}(B_{2\varepsilon}).
\end{equation}
Local finiteness of the finitely many \(\Sigma_{\mathbf j}\) gives
\eqref{eq:Bartlett-big-O}. Under \(V\)-ergodicity,
Lemma~\ref{lem:signed-random-measure-spectrum} gives
\(\Sigma_{\mathbf j}(\{0\})=0\) for every \(\mathbf j\). Hence, by
continuity from above,
\[
    \Sigma_{\mathbf j}(B_{2\varepsilon})
    \downarrow
    \Sigma_{\mathbf j}(\{0\})=0
    \qquad(\varepsilon\downarrow0),
\]
so the sum on the right of \eqref{eq:scalar-Bartlett-trace-bound} tends to
zero, proving \eqref{eq:Bartlett-little-o}.
\end{proof}



\subsection{Number variance}
\label{subsec:number-variance}

Translation boundedness also gives the following tail estimate. If
\(\sigma\) is a translation-bounded positive Radon measure,
then for every \(a>0\) and \(\beta>d\),
\[
    \int_{\|\xi\|\geq a}
        \|\xi\|^{-\beta}\,d\sigma(\xi)<\infty.
\]
Indeed, the shell \(\{n\leq\|\xi\|<n+1\}\) can be covered by
\(O(n^{d-1})\) translates of a fixed ball and hence has \(\sigma\)-mass
\(O(n^{d-1})\); summability follows from
\(\sum_{n\geq1}n^{d-1-\beta}<\infty\). The remaining compact region away
from the origin has finite \(\sigma\)-mass because \(\sigma\) is Radon.

A low-frequency power bound stronger than order \(d+1\) gives the following
estimate for ball indicators.

\begin{lemma}[Surface-order estimate for ball indicators]
\label{lem:surface-order-variance}
Let \(\sigma\) be a translation-bounded positive Radon measure on
\(V\). Suppose that there exist \(s>d+1\), \(C<\infty\), and
\(\varepsilon_0>0\) such that
\begin{equation}
\label{eq:low-frequency-power-bound}
    \sigma(B_\varepsilon)\leq C\varepsilon^s
    \qquad(0<\varepsilon\leq\varepsilon_0).
\end{equation}
Then
\begin{equation}
\label{eq:ball-Fourier-variance-bound}
    \int_V
        |\widehat{\mathbf 1}_{B_R}(\xi)|^2
        \,d\sigma(\xi)
    =
    O(R^{d-1})
    \qquad(R\to\infty).
\end{equation}
\end{lemma}

\begin{proof}
With the Fourier convention of
Subsection~\ref{subsec:Euclidean-Fourier-conventions},
\[
    \widehat{\mathbf 1}_{B_R}(\xi)
    =
    R^d\widehat{\mathbf 1}_{B_1}(R\xi).
\]
The Fourier transform of the unit ball satisfies
\begin{equation}
\label{eq:ball-Fourier-decay}
    |\widehat{\mathbf 1}_{B_1}(\xi)|
    \leq
    C_d(1+\|\xi\|)^{-(d+1)/2}.
\end{equation}
Indeed, for \(\xi\neq0\),
\[
    \widehat{\mathbf 1}_{B_1}(\xi)
    =
    \|\xi\|^{-d/2}
    J_{d/2}(2\pi\|\xi\|),
\]
and \(|J_{d/2}(t)|=O(t^{-1/2})\) as \(t\to\infty\).

Fix \(R\geq4/\varepsilon_0\). On \(B_{1/R}\),
\eqref{eq:low-frequency-power-bound} gives
\[
    \int_{B_{1/R}}
        |\widehat{\mathbf 1}_{B_R}|^2\,d\sigma
    \leq
    C R^{2d-s}
    =
    O(R^{d-1}).
\]
For \(j\geq0\), put
\[
    E_j
    :=
    B_{2^{j+1}/R}\setminus B_{2^j/R}.
\]
As long as \(2^{j+1}/R\leq\varepsilon_0\),
\eqref{eq:ball-Fourier-decay} and
\eqref{eq:low-frequency-power-bound} give
\[
\begin{aligned}
    \int_{E_j}
        |\widehat{\mathbf 1}_{B_R}|^2\,d\sigma
    &\leq
    C R^{2d}
        2^{-j(d+1)}
        \sigma(B_{2^{j+1}/R})\\
    &\leq
    C R^{2d-s}
        2^{j(s-d-1)}.
\end{aligned}
\]
Since \(s-d-1>0\), these bounds form an increasing geometric progression
in \(j\), so their sum is controlled by its final term. At the last
admissible scale
\(2^j\lesssim R\), and hence
\[
    R^{2d-s}2^{j(s-d-1)}
    \lesssim
    R^{d-1}.
\]
Summing these terms up to scale \(\varepsilon_0\) gives
\(O(R^{d-1})\).

By this weighted-tail estimate,
\begin{equation}
\label{eq:translation-bounded-weight-integrability}
    \int_{\|\xi\|\geq\varepsilon_0/4}
        \|\xi\|^{-(d+1)}
        \,d\sigma(\xi)
    <\infty.
\end{equation}

For \(\|\xi\|\geq\varepsilon_0/4\),
\eqref{eq:ball-Fourier-decay} gives
\[
    |\widehat{\mathbf 1}_{B_R}(\xi)|^2
    \leq
    C R^{d-1}\|\xi\|^{-(d+1)}.
\]
Together with
\eqref{eq:translation-bounded-weight-integrability}, this contributes
\(O(R^{d-1})\) and completes the proof.
\end{proof}

\begin{proposition}[Surface-order number variance]
\label{prop:surface-order-number-variance}
If \(2p>d+1\), then
\begin{equation}
\label{eq:surface-order-number-variance}
    \operatorname{Var}_\eta(N_{B_R})=O(R^{d-1})
    \qquad(R\to\infty).
\end{equation}
In particular, \(\eta\) is hyperuniform.
\end{proposition}

\begin{proof}
By Proposition~\ref{prop:Bartlett-spectral-decay}, the hypothesis of
Lemma~\ref{lem:surface-order-variance} holds with \(s=2p\). The Bartlett
spectrum is translation bounded by
Subsection~\ref{subsec:linear-statistics-Bartlett}. Hence
\[
    \int_V
        |\widehat{\mathbf 1}_{B_R}(\xi)|^2
        \,d\sigma_\eta(\xi)
    =
    O(R^{d-1}).
\]
By Lemma~\ref{lem:Bartlett-indicators}, the left-hand side is
\(\operatorname{Var}_\eta(N_{B_R})\). Since \(R^{d-1}=o(R^d)\), \(\eta\) is
hyperuniform.
\end{proof}

\begin{proposition}[Optimality of surface-order variance]
\label{prop:surface-order-optimality}
For the point process \(\eta\) constructed above,
\[
    0<\limsup_{R\to\infty}
    \frac{\operatorname{Var}_\eta(N_{B_R})}{R^{d-1}}<\infty.
\]
\end{proposition}

\begin{proof}
For the centered signed random measure
\(\omega-\rho\lambda_V\), the covariance spectrum is the Bartlett spectrum
\(\sigma_\eta\), by uniqueness in
Lemma~\ref{lem:signed-random-measure-spectrum}. The spectrum is nonzero:
otherwise every centered smooth linear statistic would vanish almost surely.
Taking a common conull set for a countable determining family in
\(C_c^\infty(V)\) would then give
\(\omega=\rho\lambda_V\), impossible because \(\omega\) is a locally finite
counting measure whereas \(\rho\lambda_V\) is diffuse. Since
\(V\)-ergodicity gives \(\sigma_\eta(\{0\})=0\), choose \(R_0>0\) so that
\[
    I_{R_0}
    :=
    \int_{\|\xi\|\ge R_0^{-1}}
        \frac{d\sigma_\eta(\xi)}{\|\xi\|^{d+1}}
    >0.
\]
The averaged Beck inequality of \cite[Theorem~5.1]{BB26} then gives, for
all \(R\ge R_0\),
\[
    \frac1R\int_0^R
        \operatorname{Var}_\eta(N_{B_r})\,dr
    \ge
    C_d I_{R_0} R^{d-1}.
\]
If
\(\limsup_{R\to\infty}R^{1-d}\operatorname{Var}_\eta(N_{B_R})=0\), then
for every \(\delta>0\) there is \(R_\delta\) such that
\(\operatorname{Var}_\eta(N_{B_r})\le \delta r^{d-1}\) for
\(r\ge R_\delta\). Consequently,
\[
    \frac1R\int_0^R
        \operatorname{Var}_\eta(N_{B_r})\,dr
    \le
    \left(\frac{\delta}{d}+o(1)\right)R^{d-1},
\]
contradicting the preceding lower bound for sufficiently small \(\delta\).
The upper bound follows from
Proposition~\ref{prop:surface-order-number-variance}.
\end{proof}



\subsection{Linear rigidity}
\label{subsec:rigidity}

\begin{lemma}[Polynomial cutoffs]
\label{lem:polynomial-cutoff}
Let \(p\geq1\), and let \(\sigma\) be a translation-bounded positive Radon
measure satisfying
\begin{equation}
\label{eq:polynomial-cutoff-spectral-hypothesis}
    \sigma(B_\varepsilon)
    =
    o(\varepsilon^{2p})
    \qquad(\varepsilon\downarrow0).
\end{equation}
Let \(P\) be a homogeneous polynomial of degree \(m\), with
\(p>d+m\), and let \(A\subset V\) be bounded. Then there are
functions \(f_R\in C_c^\infty(V)\), equal to \(P\) on \(A\) for
all sufficiently large \(R\), such that
\begin{equation}
\label{eq:polynomial-cutoff-variance}
    \int_V|\widehat f_R(\xi)|^2\,d\sigma(\xi)
    \longrightarrow0
    \qquad(R\to\infty).
\end{equation}
\end{lemma}

\begin{proof}
Choose \(\phi\in C_c^\infty(V)\) with
\(\phi=1\) on \(B_1\), and set
\[
    f_R(u):=P(u)\phi(u/R).
\]
Since \(A\) is bounded, \(f_R=P\) on \(A\) for all sufficiently
large \(R\). Writing \(g(u):=P(u)\phi(u)\), homogeneity gives
\begin{equation}
\label{eq:polynomial-cutoff-Fourier-scaling}
    f_R(u)=R^m g(u/R),
    \qquad
    \widehat f_R(\xi)
    =
    R^{d+m}\widehat g(R\xi).
\end{equation}

Choose \(\varepsilon_0>0\) and \(C<\infty\) such that
\[
    \sigma(B_\varepsilon)
    \leq C\varepsilon^{2p}
    \qquad(0<\varepsilon\leq\varepsilon_0).
\]
On \(B_{1/R}\), \eqref{eq:polynomial-cutoff-spectral-hypothesis}
and \eqref{eq:polynomial-cutoff-Fourier-scaling} give
\begin{equation}
\label{eq:polynomial-cutoff-lowest-frequency}
    \int_{B_{1/R}}|\widehat f_R|^2\,d\sigma
    \leq
    C R^{2d+2m}\sigma(B_{1/R})
    =
    o\bigl(R^{2d+2m-2p}\bigr)
    =
    o(1),
\end{equation}
because \(p>d+m\).

Choose \(N>p\). For
\[
    E_j
    :=
    B_{2^{j+1}/R}\setminus B_{2^j/R},
\]
provided \(2^{j+1}/R\leq\varepsilon_0\), the Schwartz decay of
\(\widehat g\) gives
\[
    \sup_{\xi\in E_j}
        |\widehat g(R\xi)|
    \leq C_N2^{-jN}.
\]
Hence
\begin{equation}
\label{eq:polynomial-cutoff-small-annuli}
\begin{aligned}
    \int_{E_j}|\widehat f_R|^2\,d\sigma
    &\leq
    C_N
    R^{2d+2m-2p}
    2^{j(2p-2N)}.
\end{aligned}
\end{equation}
Since \(N>p\), the factors \(2^{j(2p-2N)}\) are summable
uniformly in the number of annuli. Thus their total contribution is
\(O(R^{2d+2m-2p})\), which tends to zero because \(p>d+m\).

Since \(2N>d\), the weighted-tail estimate at the start of
Subsection~\ref{subsec:number-variance} gives
\begin{equation}
\label{eq:polynomial-cutoff-tail-integrability}
    \int_{\|\xi\|\geq\varepsilon_0/2}
        \|\xi\|^{-2N}\,d\sigma(\xi)
    <\infty.
\end{equation}
By Schwartz decay,
\[
    |\widehat f_R(\xi)|^2
    \leq
    C_N R^{2d+2m-2N}\|\xi\|^{-2N}
    \qquad
    (\|\xi\|\geq\varepsilon_0/2).
\]
Since \(N>p>d+m\),
\eqref{eq:polynomial-cutoff-tail-integrability} shows that this part
of the integral also tends to zero. This proves
\eqref{eq:polynomial-cutoff-variance}.
\end{proof}

Closely related implications from low-frequency suppression to higher-order
rigidity are developed by Lachi\`eze-Rey \cite{LR24}; see also
Lotz--Klatt \cite[Proposition~3.17]{LK26}. In the isotropic power-law case
\(s(\xi)\asymp\|\xi\|^\alpha\), the borderline in \cite{LR24} is
\(\alpha=2\ell+d\) for the spectral density, hence \(2(\ell+d)\) for the
ball-mass exponent because \(\sigma(B_\varepsilon)\asymp
\varepsilon^{\alpha+d}\). The cutoff below is a slightly stronger sufficient
condition for arbitrary translation-bounded spectral measures, with no
absolute-continuity assumption.

\begin{proposition}[Finite-order linear rigidity]
\label{prop:polynomial-rigidity}
Let \(\ell\geq0\). If \(p>d+\ell\), then \(\eta\) is linearly \(\ell\)-rigid and
\(\ell\)-rigid.
\end{proposition}

\begin{proof}
Fix a bounded Borel set \(A\subset V\) and a multi-index
\(\mathbf j\) with \(|\mathbf j|\leq\ell\). Put
\(P(u):=u^{\mathbf j}\) and \(m:=|\mathbf j|\). Proposition~\ref{prop:Bartlett-spectral-decay}
and \(p>d+m\) allow us to apply Lemma~\ref{lem:polynomial-cutoff}, which gives
\(f_R\in C_c^\infty(V)\), with \(f_R=P\) on \(A\) for large \(R\),
such that
\begin{equation}
\label{eq:rigidity-small-variance}
    \|(\bS f_R)^\circ\|_{L^2(\eta)}^2
    =
    \int_V|\widehat f_R|^2\,d\sigma_\eta
    \longrightarrow0.
\end{equation}

For such \(R\), put \(h_R:=f_R-P\mathbf 1_A\). Then \(h_R\) is bounded,
compactly supported, and vanishes on \(A\).
By linearity,
\begin{equation}
\label{eq:rigidity-exterior-approximation}
    \mathfrak m_{\mathbf j}(\,\cdot\,;A)^\circ
    =
    (\bS f_R)^\circ-(\bS h_R)^\circ.
\end{equation}
The second term belongs to \(\mathcal H_{A^c}\), while the first
tends to zero in \(L^2(\eta)\) by
\eqref{eq:rigidity-small-variance}. Equivalently,
\[
    -(\bS h_R)^\circ
    \longrightarrow
    \mathfrak m_{\mathbf j}(\,\cdot\,;A)^\circ
    \quad\text{in }L^2(\eta).
\]
Since \(\mathcal H_{A^c}\) is closed,
\[
    \mathfrak m_{\mathbf j}(\,\cdot\,;A)^\circ
    \in\mathcal H_{A^c}.
\]
Thus \(\eta\) is linearly \(\ell\)-rigid, and therefore
\(\ell\)-rigid by Subsection~\ref{subsec:rigidity-definitions}.
\end{proof}



\section{Proofs of the main results}
\label{sec:proofs-main-theorems}

The constructions have already been completed in
Sections~\ref{sec:dimension-one} and \ref{sec:local-filling-assembly}.
We now collect their conclusions, together with the estimates of
Section~\ref{sec:spectral-decay-rigidity}, and verify the statements in
Subsection~\ref{subsec:introduction-main-results}.


\subsection{One-dimensional theorem and corollary}
\label{subsec:proof-main-dimension-one}

Let \(\mathbb R\curvearrowright(X,\mu)\) be an essentially free ergodic
p.m.p.\ Borel action. Choose the bounded suspension representation from
Subsection~\ref{subsec:bounded-suspension}, with
\[
    0<\tau_-\leq\tau(z)\leq\tau_+<\infty,
\]
and choose \(\rho_0\) so that \(\rho_0\tau_-\geq20\). Fix
\(\rho\geq\rho_0\).

Section~\ref{sec:dimension-one} constructs a Borel translation-equivariant
map
\[
    \kappa_\rho:X\longrightarrow\mathcal N_s(\mathbb R),
    \qquad
    \kappa_\rho(x)=\delta_{P_x},
\]
with uniformly Delone images and with \(\kappa_\rho\) injective. It also
defines
\[
    Y_\rho:=\{x\in X:0\in P_x\},
    \qquad
    \eta_\rho:=(\kappa_\rho)_*\mu.
\]
By Proposition~\ref{prop:configuration-realization}, \(Y_\rho\) is a
generating Delone translation cross-section with \((Y_\rho)_x=P_x\). The
intensity computation following Proposition~\ref{prop:one-dimensional-configuration-map}
gives \(\rho_{\eta_\rho}=\rho\).

On the invariant conull suspension model fixed in
Section~\ref{sec:dimension-one},
Proposition~\ref{prop:one-dimensional-Euclidean-discrepancy} gives
\[
    \left|
        \#\bigl((Y_\rho)_x\cap I\bigr)-\rho|I|
    \right|
    \leq
    C_\rho\log(2+|I|)
\]
for every bounded interval \(I\subset\mathbb R\). Transporting this model back
to the original action gives the invariant conull set \(X_0\) in
Theorem~\ref{thm:main-dimension-one}. Proposition~\ref{prop:one-dimensional-maximal-rigidity}
gives maximal rigidity. Moreover,
Corollary~\ref{cor:one-dimensional-number-variance} and
Proposition~\ref{prop:one-dimensional-Bartlett-decay} give
\[
    \operatorname{Var}_{\eta_\rho}(N_{[0,R)})=O(\log^2R)
    \qquad(R\to\infty)
\]
and
\[
    \sigma_{\eta_\rho}([-\varepsilon,\varepsilon])
    =
    O\!\left(
        \varepsilon^2\log^2\!\frac{e}{\varepsilon}
    \right)
    \qquad(\varepsilon\downarrow0),
\]
respectively. In particular, \(\eta_\rho\) is hyperuniform.

Finally, the Borel injective equivariant map \(\kappa_\rho\) identifies
\((X,\mu)\), by Lusin--Souslin, with its Borel image of full
\(\eta_\rho\)-measure. Hence \(\eta_\rho\) is translation ergodic. This proves
Theorem~\ref{thm:main-dimension-one}.

For Corollary~\ref{cor:main-dimension-one-point-process}, apply this theorem to
the canonical translation action on
\((\mathcal N_s(\mathbb R),\eta_0)\). The return-time map \(T=\kappa_{Y_\rho}\)
of the resulting generating cross-section is a translation-equivariant
measurable isomorphism onto \((\mathcal N_s(\mathbb R),\eta_\rho)\), and the
variance and Bartlett estimates above give the remaining conclusions.


\subsection{Higher-dimensional theorem and corollaries}
\label{subsec:proof-main-higher-dimensional}

Let \(d\geq2\), let \(K\leq O(V)\) be closed, and let
\(G_K\curvearrowright(X,\mu)\) be an essentially free p.m.p.\ Borel action
whose restricted \(V\)-action is ergodic. Fix \(p\geq1\), let \(\rho_0\) be
the final threshold from Sections~\ref{sec:orbit-geometry-integerization} and
\ref{sec:local-filling-assembly}, and fix \(\rho\geq\rho_0\).

Section~\ref{sec:local-filling-assembly} constructs the Borel
\(G_K\)-equivariant injective map
\[
    \kappa_\rho(x)=\delta_{P_x}
\]
with uniformly Delone images, and defines \(Y_\rho\) and \(\eta_\rho\) by
\eqref{eq:higher-dimensional-cross-section}. Proposition~\ref{prop:configuration-realization},
applied there, gives a \(K\)-invariant generating Delone translation
cross-section with
\[
    (Y_\rho)_x=P_x.
\]
The equivariance and injectivity of \(\kappa_\rho\), together with
Lusin--Souslin, give
\[
    (\mathcal N_s(V),\eta_\rho)\cong(X,\mu)
\]
as p.m.p.\ \(G_K\)-spaces. Thus \(\eta_\rho\) is \(G_K\)-invariant and its
restricted \(V\)-action is ergodic.

Propositions~\ref{prop:higher-dimensional-exact-intensity} and
\ref{prop:higher-dimensional-maximal-rigidity} give
\(\rho_{\eta_\rho}=\rho\) and maximal rigidity. Proposition~\ref{prop:Bartlett-spectral-decay}
gives
\begin{equation}
\label{eq:main-theorem-higher-dimensional-spectral}
    \sigma_{\eta_\rho}(B_\varepsilon)
    =
    o(\varepsilon^{2p})
    \qquad(\varepsilon\downarrow0),
\end{equation}
and, when \(2p>d+1\),
Proposition~\ref{prop:surface-order-number-variance} gives
\[
    \operatorname{Var}_{\eta_\rho}(N_{B_R})=O(R^{d-1})
    \qquad(R\to\infty).
\]
This proves Theorem~\ref{thm:main-higher-dimensional}.

Corollary~\ref{cor:main-point-process-realization} follows by applying this
theorem to the canonical \(G_K\)-action on
\((\mathcal N_s(V),\eta_0)\), choosing \(p\geq q\) with \(p>d+\ell\), and
using Proposition~\ref{prop:polynomial-rigidity}. The return-time map of the
new generating cross-section is the equivariant measurable isomorphism \(T\)
in the statement. Proposition~\ref{prop:surface-order-optimality} gives the
optimality statement following that corollary.

For Corollary~\ref{cor:main-bounded-displacement},
Proposition~\ref{prop:uniform-boundary-discrepancy} gives the stated estimate
for finite unions of unit cubes. Proposition~\ref{prop:bounded-displacement-lattice}
and Laczkovich's criterion give the consequence recorded in the following
remark.

For \(K=SO(V)\), \(G_K\)-invariance is isotropy, and
Corollary~\ref{cor:isotropic-ergodic-weak-mixing} upgrades the established
\(V\)-ergodicity to \(V\)-weak mixing. Applying
Corollary~\ref{cor:main-point-process-realization} to the canonical action of
\(\eta_0\) proves Corollary~\ref{cor:main-isotropic-realizations}. The same
argument with \(K=O(V)\) gives the orthogonally invariant version.


\section*{Statements and Declarations}

\paragraph{\textbf{Funding.}}
This work was supported by the Swedish Research Council under grant
VR 11253322.

\paragraph{\textbf{Competing interests.}}
The author has no relevant financial or non-financial interests to disclose.

\paragraph{\textbf{Data availability.}}
No datasets were generated or analysed during the current study.

\paragraph{\textbf{Use of artificial-intelligence tools.}}
The author used ChatGPT (OpenAI) during the preparation of the manuscript
for editorial assistance and consistency checking, and used Claude
(Anthropic), Gemini (Google DeepMind), and Aristotle (Harmonic) as additional
tools for independent manuscript and proof audits. Aristotle was also used
to formally check several isolated algebraic and analytic computations. All
outputs were independently evaluated by the author, who takes full
responsibility for the mathematical content, references, and final text.


\end{document}